\documentclass[12pt,a4paper]{article}
\usepackage{amsmath}
\usepackage{amssymb}
\usepackage{amsthm}
\usepackage{amsfonts}
\usepackage{latexsym}

\theoremstyle{plain}
\newtheorem{thm}{Theorem}
\newtheorem{cor}[thm]{Corollary}
\newtheorem{lem}[thm]{Lemma}
\newtheorem{prop}[thm]{Proposition}

\theoremstyle{remark}
\newtheorem{rem}[thm]{Remark}
\theoremstyle{definition}
\newtheorem{defn}[thm]{Definition}
\newtheorem{cav}[thm]{Caveat}

\title{Poisson and Szeg\"{o} kernel scaling asymptotics \\ on Grauert tube
boundaries\\
(after Zelditch, Chang and Rabinowitz)}
\author{Roberto Paoletti\footnote{\noindent{\bf Address:}
Dipartimento di Matematica e Applicazioni, Universit\`a degli Studi
di Milano Bicocca, Via R. Cozzi 55, 20125 Milano,
Italy; {\bf e-mail}: roberto.paoletti@unimib.it }}
\date{}

\begin{document}
\maketitle

\begin{center}
\textit{Dedicated to the memory of Steve Zelditch}
\end{center}

\begin{abstract}
We review and elaborate on recent work of Chang and Rabinowitz on scaling asymptotics of Poisson and 
Szeg\"{o} kernels on Grauert tubes, providing additional results that may be useful in applications. In particular, 
focusing on the near-diagonal case, 
we give an explicit description of the leading order coefficients, and an estimate on the growth of the degree of certain polynomials describing the rescaled asymptotics. Furthermore, we allow rescaled asymptotics in a range 
$O\left(\lambda^{\delta-1/2}\right)$ in all the variables involved, 
where $\lambda\rightarrow+\infty$ is the
asymptotic parameter, rather than rescale according to Heisenberg type.
 
\end{abstract}

\section{Introduction}
It was shown by Bruhat and Whitney (\cite{bw})
that a real-analytic compact and connected manifold 
$M$ has an essentially unique complexification $(\tilde{M},J)$, that is, 
a complex manifold in which $M$ sits as a totally real submanifold.
If furthermore $\kappa$ is a real-analytic Riemannian metric on $M$, it was proved independently by Guillemin and Stenzel
(\cite{gs1}, \cite{gs2}) and by Lempert and Sz\"{o}ke (\cite{l},
\cite{ls}, \cite{szo}) that
$\tilde{M}$ can be endowed with a canonically determined K\"{a}hler structure
$(\tilde{M},J,\Omega)$, with the following two properties. First,
the symplectic manifold $(\tilde{M},\Omega)$ is symplectomorphic to
a tubular neighbourhood of the zero section in the cotangent bundle $T^\vee M$ 
of $M$,
endowed with its canonical symplectic structure
(and the symplectomorphism, of course, 
is the identity on $M$). Second, the square norm of
$\kappa$ pulls back on $\tilde{M}$ to a K\"{a}hler potential for $\Omega$,
while the norm pulls back on $\tilde{M}\setminus M$
to a solution of the complex Monge-Amp\`{e}re equation. 
Viewed the other way, this construction endows the cotangent bundle of $M$,
near the zero section, with
a canonical compatible complex structure,
with the property that the Riemannian and Monge-Amp\`{e}re
foliations coincide. This intrinsic complex struture
was called \textit{adapted} by Lempert and Sz\"{o}ke.
The sphere bundles of sufficiently small radii $\tau>0$ 
in $T^\vee M$ correspond to boundaries $X^\tau$ of strictly pseudoconvex domains in $M^\tau\subseteq \tilde{M}$, 
so-called \textit{Grauert tubes},
corresponding to disc bundles in $T^\vee M$;
furthermore, the homogeneous geodesic flow on the sphere bundles
is closely related to the flow of the Reeb vector field on $X^\tau$.

The analysis and geometry of the Grauert tubes and 
their boundaries have been 
studied extensively in recent years, giving rise to a vast literature.
Obviously with no pretense of completeness, let us mention a few
foundational references most relevant to the present
discussion:
\cite{bdm1} \cite{bdm2},
\cite{b}, \cite{bh}, \cite{g}, \cite{gs1}, \cite{gs2}, 
\cite{pw1}, \cite{pw2}, 
\cite{l}, \cite{ls}, 
\cite{szo}, \cite{lgs}, 
\cite{s1}, \cite{s2}, 
\cite{z07}, \cite{z10}, \cite{z12},
\cite{z13}, \cite{z14}, \cite{z17}, 
\cite{cr1}, \cite{cr2}.
 
A basic problem in this context, partly motivated by certain analogies with the setting of positive line bundles on complex projective manifolds, is to study
the asymptotic concentration behaviour of the eigenfunctions of the generator of
the homogeneous geodesic flow (transported to $X^\tau$). 
This involves 
the local study of certain
smoothed spectral spectral projectors, which are
the Toeplitz counterpart of commonly studied objects 
in spectral theory (see e.g. \cite{dg}, \cite{gs}). 
However,
in sharp contrast to the line bundle setting, the latter flow is generally not CR-holomorphic. This is a source of difficulty 
in the adaptation to the Grauert tube setting
of the Szeg\"{o} kernel techniques 
that have proved successful in the line bundles setting
(\cite{z98}, \cite{sz}, \cite{bsz}).

A related issue, which is instead genuinely intrinsic to 
the Grauert tube setting, concerns analytic continuation of the 
eigenfunctions of the non-negative Laplacian $\Delta$
of $(M,\kappa)$. Let
$(\varphi_j)_{j=1}^\infty$ be a complete orthonormal system
of $L^2(M)$ composed of eigenfuntions of $\Delta$.
Being real-analytic, each $\varphi_j$ extends to a 
holomorphic function $\tilde{\varphi}_j$ to some sufficiently small Grauert tube
$M^\tau\subseteq \tilde{M}$; 
here in principle $\tau$ depends on $j$. 
It was a deep discovery of 
Boutet de Monvel (\cite{bdm1}, \cite{bdm2}) that for sufficiently small $\tau>0$ 
every $\varphi_j$ extends to
$M^\tau$ (see also
\cite{gs2}, \cite{leb}, \cite{s1}, \cite{s2}, \cite{z12},
\cite{z17}  for discussions
and different proofs). As explained in \cite{z07}, \cite{z10},
\cite{z12}, the relation between the
eigenfunctions $\varphi_j$ and their analytic continuations
$\tilde{\varphi}_j$ is governed by the so-called Poisson-wave
operator, obtained by analytically continuing 
the Schwartz kernel of the wave operator on $M$
(in both time and space). The study of the asymptotic 
distribution of the analytic continued eigenfunctions, 
pioneered by Zeldtich, involves certain \lq 
complexified spectral projectors\rq, 
bearing resemblance to the smoothed spectral
projectors above.

In two striking recent papers (\cite{cr1}, \cite{cr2}), 
Chang and Rabinowitz have made significant progress
in pushing forward the analogy between the line bundle
and the Grauert tube settings. Their analysis rests on two pillars. One is the description, due to Zelditch, of 
certain unitary groups of Toeplitz operators as 
\lq dynamical Toeplitz operators\rq (\cite{z97},
\cite{z10}, \cite{z12}); another is a clever use of the Heisenberg (or normal) local coordinates adapted to a hypersurface in a complex manifold
introduced by Folland and Stein in \cite{fs1} and \cite{fs2}. 

The goal of the present paper is partly to present a gentle
introduction to
the promising and efficient approach of Chang and Rabinowitz, and partly to provide some 
complementary results that may be useful in future applications.
We shall restrict the present discussion to the near-diagonal situation.
In particular, we shall give near-diagonal scaling asymptotics for the smoothing kernels
hinted at above, in a
range $O\left(\lambda^{\delta-1/2}\right)$ in all variables
involved, rather than according to Heisenberg type. More
precisely, in an appropriate set of local coordinates centered 
at $x\in X^\tau$,
we shall consider rescaled dispacements of the form
form $x+(\theta/\sqrt{\lambda},\mathbf{v}/\sqrt{\lambda})$
with $\|(\theta,\mathbf{v})\|\le C\,k^{\delta}$ (here
$\delta\in (0,1/6)$, say). Furthermore, we shall provide an
explicit numerical determination of the leading order factor in the
asymptotic expansion. Also, we shall give a bound in the degree 
of the polynomials in the rescaled variables 
$(\theta,\mathbf{v})$
that appear in
the lower order terms of the expansion; this is useful in
ensuring that, when $(\theta,\mathbf{v})$ are allowed to
expand at a controlled pace as above, one actually obtains an asymptotic expansion.

In order to give a more precise description of the content
of this paper, it is in order to premise a more detailed
account of the geometric setting involved.

Let $M$ be a $d$-dimensional 
compact connected real-analytic 
(in the following, $\mathcal{C}^\varpi$) manifold.
As mentioned, there is a complex manifold $(\tilde{M},J)$, 
the Bruhat-Whitney
complexification of $M$ \cite{bw}, in which 
$M$ embeds as a totally real submanifold.
If $\jmath:M\hookrightarrow \tilde{M}$
is the inclusion, then
$\tilde{M}$ is uniquely
determined locally along $\jmath(M)$, up to unique biholomorphism. Therefore, 
since 
$(\tilde{M},J)$ and $(\tilde{M},-J)$ are both
complexifications of $M$,
there is an anti-holomorphic involution $\sigma:\tilde{M}\rightarrow \tilde{M}$
with fixed locus $\jmath (M)$
(\cite{bw}, \cite{gs1}, \cite{lgs}). We shall identify $M$ with $\jmath (M)$ in the following.

Furthermore, the choice of  a $\mathcal{C}^\omega$ Riemannian metric $\kappa$ on $M$ 
determines a canonical 
K\"{a}hler structure on $(\tilde{M},J)$, with a canonical K\"{a}hler 
potential (\cite{pw1},
\cite{gs1}, \cite{ls}). More precisely, perhaps after 
replacing $\tilde{M}$ with some smaller tubular neighborhood of $M$, there exists 
a unique $\mathcal{C}^\varpi$ function
$\rho:\tilde{M}\rightarrow [0,+\infty)$ with the following properties:

\begin{enumerate}
\item $\rho^{-1}(0)=M$ and $\rho\circ \sigma=\rho$;
\item $\rho$ is strictly plurisubharmonic, that is, $\Omega:=\imath\,\partial\,\overline{\partial}\rho$
is a K\"{a}hler  form on $(\tilde{M},J)$, 
whose associated associated Riemannian metric  on $\tilde{M}$ 
will be
denoted by $\hat{\kappa}=\Omega\big(\cdot,J(\cdot)\big)$;

\item $\sqrt{\rho}$ satisfies the complex Monge-Amp\`{e}re equation on $\tilde{M}\setminus M$: in local holomorphic coordinates $(z_j)$,
$$
\det \left(\frac{\partial^2}{\partial z_i\,\partial \overline{z}_j}\sqrt{\rho}\right)=0.
$$

\end{enumerate}

Therefore, for some $\tau_{max}\in (0,+\infty]$
and for every $\tau\in (0,\tau_{max})$, 
$$\tilde{M}^\tau:=\rho^{-1}\left([0,\tau^2)\right)$$
is a strictly pseudoconvex domain in $\tilde{M}$; its boundary is the compact smooth hypersurface
$$X^\tau:=\rho^{-1}(\tau^2)\subset \tilde{M}\setminus M.$$
One calls $\tilde{M}^\tau$ 
the \textit{open Grauert tube of radius $\tau$},
and $\sqrt{\rho}$ the
\textit{tube function} of $(M,\kappa)$.
%
We shall denote by $H(X^\tau)\subset L^2(X^\tau)$ the Hardy space of $X^\tau$,
and by
\begin{equation}
\label{eqn:Pi tau}
\Pi^\tau:L^2(X^\tau)\rightarrow H(X^\tau)
\end{equation}
the orhogonal projector, that is, the Szeg\"{o} projector of the
CR manifold $X^\tau$\footnote{For a precise description of the volume form
on $X^\tau$, see \S \ref{sctn:vol form cr}}.

One also has the following alternative perspective.
Rather than starting from the complexification $(\tilde{M},J)$, and 
then considering 
a K\"{a}hler structure $\Omega$ on $(\tilde{M},J)$ canonically
induced by $\kappa$, one can
start instead from the
canonical symplectic structure $\Omega_{can}$ on the cotangent bundle $T^\vee M$,
locally given by $\mathrm{d}\mathbf{q}\wedge\mathrm{d}\mathbf{p}$,
where $\mathbf{q}$ denotes local coordinates on $M$ and $\mathbf{p}$ the induced fiberwise linear coordinates on the cotangent spaces; then one introduces
a compatible complex structure $J_{ad}$ on a suitable neighborhood of the zero section in $(T^\vee M,\Omega_{can})$, again canonically induced by $\kappa$. In fact, $J_{ad}$ is uniquely determined by the condition that the maps 
$\mathbb{C}\rightarrow
TM$ parametrizing the leaves of the Riemann foliation are holomorphic, when suitably restricted.
Lempert and Sz\"{o}ke call
$J_{ad}$  \textit{adapted} (\cite{gs1}, \cite{l}, \cite{ls},
\cite{szo}). 

More explicitly, if $(m,v)\in TM$ and $v\neq 0$, 
let $\gamma :\mathbb{R}\rightarrow M$ denote the unique geodesic with initial condition
$(m,v)$, and $\dot{\gamma}:I\rightarrow TM$ its velocity vector. 
For every $b>0$, let $N_b:TM\rightarrow TM$ denote dilation by
$b$.
Consider the smooth map
\begin{equation}
\label{eqn:riemann fol}
\psi_\gamma:a+\imath\,b\in \mathbb{C}\mapsto 
N_b\left(\dot{\gamma}(a)\right)\in
TM.
\end{equation}
The submanifolds 
$\psi_\gamma (\mathbb{C}\setminus \mathbb{R})$ are the leaves of the 
\textit{Riemann foliation} of $TM\setminus M_0$ 
(here $M_0$ is the
zero section); $J_{ad}$ is characterized 
by the property that the maps $\psi_\gamma$ are holomorphic
from a suitable strip $S\subseteq \mathbb{C}$ to $T^\epsilon M$
(see \cite{pw1} and \cite{ls}, \S 3-4, \cite{szo}).

After \cite{ls} and \cite{lgs}, the two approaches are related 
by the imaginary time exponential map of $\kappa$ (see also the discussions
in \cite{z07}, \cite{z10}, \cite{z12}
).
Let us identify $TM$ and $T^\vee M$ by $\kappa$, and view
$\Omega_{can}$ as a symplectic structure on $TM$.
Given $m\in M$, let $\exp_m:U_m\rightarrow
M$ be the exponential map at $m$ for $\kappa$;
here $U_m\subseteq T_mM$ is some neighborhood of the origin.
Let $T_m^\mathbb{C}M:=T_mM\otimes_\mathbb{R}\mathbb{C}$ be the complexified tangent space
at $m$; 
being $\mathcal{C}^\varpi$,
$\exp_m$ extends to a holomorphic map 
$E_m:\tilde{U}_m\rightarrow \tilde{M}$, where $\tilde{U}_m\subseteq
T_m^\mathbb{C}M$ is some open neighborhood of $U_m$.
As $M$ is compact,
we may assume that for some
$\epsilon>0$ one has
$$
\tilde{U}_m\supseteq
\left\{(m,\imath\,v)\,:\,v\in T_m M,\,\|v\|_m<\epsilon\right\}
\quad
\forall m\in M.
$$
If
$$
T^\epsilon M:=
\left\{(m,v)\,:m\in M,\,v\in T_m M,\,\|v\|_m<\epsilon\right\},
$$ 
one obtains a $\mathcal{C}^\varpi$ map 
 \begin{equation}
\label{eqn:imaginary exponential map}
E^\epsilon\,:\,(m,v)\in T^\epsilon M\mapsto E_m(\imath\,v)\in \tilde{M},
\end{equation}
such that $E^\epsilon (m,0)=m\quad\forall\,m\in M$; 
$E^\epsilon$ has the following properties:

\begin{enumerate}
\item $E^\epsilon$ intertwines the square
norm function
$\|\cdot\|^2:TM\rightarrow \mathbb{R}$ and $\rho$:
\begin{equation}
\label{eqn:key prop E1}
{E^\epsilon}^*(\rho)=\|\cdot\|^2,
\end{equation} 
hence it maps $T^\tau M$ to $\tilde{M}^\tau$, for all
$\tau\in (0,\epsilon)$;
\item for sufficiently small $\epsilon>0$,
$E^\epsilon$ is a $\mathcal{C}^\varpi$ symplectomorphism
between $(T^\epsilon M,\Omega_{can})$ and $(\tilde{M}^\epsilon,\Omega)$, that is,
\begin{equation}
\label{eqn:key prop E}
{E^\epsilon}^*(\Omega)=\Omega_{can},
\end{equation}
and similarly replacing $\epsilon$ with any $\tau\in (0,\epsilon)$;

\item consequently, $E^\epsilon$ intertwines the Hamiltonian flow
of $\|\cdot\|$ on $(T^\epsilon M\setminus (0),\Omega_{can})$, 
which is the homogeneous geodesic
flow, with the Hamiltonian flow of $\sqrt{\rho}$ on 
$(\tilde{M}\setminus M,\Omega)$;

\item the relation ${E^\epsilon}^*(\Omega)=\Omega_{can}$ also
implies that ${E^\epsilon}^*(J)$ is a compatible complex structure on
$(T^\epsilon M,\Omega_{can})$, and in fact 
$J_{ad}={E^\epsilon}^*(J)$;

\item $E^\epsilon$
intertwines the Riemann foliation
of $(M,\kappa)$ with the Monge-Amp\`{e}re foliation of
$\sqrt{\rho}$ (\cite{ls}, \cite{gls}).
\end{enumerate}

In short, $E^\epsilon$ yields an isomorphism of K\"{a}hler manifolds 
$E^\epsilon$
between $(T^\epsilon M,\,\Omega_{can},\,J_{ad})$ and
$(\tilde{M}^\epsilon,\Omega,J)$, and if
$\tau\in (0,\epsilon)$ then the Grauert tube boundary
$X^\tau\subseteq \tilde{M}^\epsilon$ is identified with
the (co)sphere bundle of radius $\tau$ in $T^\vee M\cong TM$.

There are natural choice for primitives of both $\Omega_{can}$ and $\Omega$,
and these also get intertwined by $E^\epsilon$.
More precisely, let $\lambda_{can}\in \Omega^1(T^\vee M)$
be the canonical $1$-form, locally expressed as
$\mathbf{p}\,\mathrm{d}\mathbf{q}$; thus
$\Omega_{can}=-\mathrm{d}\lambda_{can}$. 
Furthermore, let 
\begin{equation}
\label{eqn:defn lambda}
\lambda:=\Im (\overline{\partial}\rho)\in \Omega^1\big(\tilde{M}\big);
\end{equation}
thus 
$\Omega=-\mathrm{d}\lambda$.
Then on $T^\epsilon M$
\begin{equation}
\label{eqn:taut 1form}
\lambda_{can}={E^\epsilon}^*(\lambda).
\end{equation}
Let us set 
\begin{equation}
\label{eqn:defn alpha}
\alpha:=-\lambda=\Im (\partial \rho)
=\frac{1}{2\,\imath}\,\left(\partial\rho-\overline{\partial}\rho\right)\qquad
\text{and}
\qquad \omega:=\frac{1}{2}\,\Omega,
\end{equation}
so that
\begin{equation}
\label{eqn:Omega alpha}
\Omega=-\mathrm{d}\lambda=\mathrm{d}\alpha=2\,\omega
=\imath\,\partial\,\overline{\partial}\rho .
\end{equation}

If $M'\subseteq \tilde{M}$ is open and $f\in \mathcal{C}^\infty (M')$,
$\upsilon_f\in \mathfrak{X}(M')$ will denote the Hamiltonian vector field of
$f$ with respect to $\Omega$.
By the above, the homogeneous geodesic flow on $T^\vee M\setminus M_0$
is intertwined by $E^\epsilon$ with
the flow of $\upsilon_{\sqrt{\rho}}$ on $\tilde{M}\setminus M$.
The following holds (see \S \ref{scnt:preliminaries}).
%
%
\begin{enumerate}
\item $\alpha$ is invariant under the flow of $\upsilon_{\sqrt{\rho}}$;
equivalently, $\lambda_{can}$ is preserved by the homogeneous geodesic flow;

\item if $\jmath^\tau:X^\tau\hookrightarrow \tilde{M}$
is the inclusion, then $\alpha^\tau
:={\jmath^\tau}^*(\alpha)$ is a contact form;

\item the cone 
\begin{equation}
\label{eqn:symplectic cone}
\Sigma^\tau:=\left\{(x,\,r\,\alpha_x^\tau)\,:\,
x\in X^\tau,\,r>0\right\}\subseteq T^\vee X^\tau\setminus (0)
\end{equation}
is symplectic (for the standard symplectic structure of $T^\vee X^\tau$);

\item being tangent to $X^\tau$, $\upsilon_{\sqrt{\rho}}$ induces by restriction a smooth vector field $\upsilon_{\sqrt{\rho}}^\tau\in
\mathfrak{X}(X^\tau)$, and by the above the flow of $\upsilon_{\sqrt{\rho}}^\tau$
preserves the volume form on $X^\tau$; 

\item consequently, the differential operator 
\begin{equation}
\label{eqn:defn di Dsqrtrho}
D_{\sqrt{\rho}}^\tau:=\imath\,\upsilon_{\sqrt{\rho}}^\tau
\end{equation}
is formally self-adjoint on $L^2(X^\tau)$;

\item the principal symbol $\sigma (D_{\sqrt{\rho}}^\tau)$ of $D_{\sqrt{\rho}}^\tau$
is positive along $\Sigma^\tau$.

\end{enumerate}

These facts have the following consequence. Consider
the composition
\begin{equation}
\label{eqn:toeplitz operator Drho}
\mathfrak{D}_{\sqrt{\rho}}^\tau:=\Pi^\tau\circ D_{\sqrt{\rho}}^\tau\circ \Pi^\tau;
\end{equation}
in the terminology of \cite{g} and \cite{bg}, 
$\mathfrak{D}_{\sqrt{\rho}}^\tau$ is a 
self-adjoint first-order Toeplitz operator on $X^\tau$. 
Its principal symbol 
(as a Toeplitz operator) is, by definition, 
the restriction of $\sigma (D_{\sqrt{\rho}}^\tau)$ to $\Sigma^\tau$,
and is therefore strictly positive. Hence $\mathfrak{D}_{\sqrt{\rho}}^\tau$ 
is an elliptic Toeplitz operator.
By the theory in \S 2 of \cite{bg} (especially Proposition 2.4),
the spectrum 
of $\mathfrak{D}_{\sqrt{\rho}}^\tau$ is discrete,
bounded from below and has only $+\infty$ as accumulation point (see \S \ref{sctn:symbol toeplitz} below).

Let then
$$
\lambda_1^\tau< \lambda_2^\tau<\ldots \uparrow +\infty
$$ 
be the distinct eigenvalues of
$\mathfrak{D}_{\sqrt{\rho}}^\tau$; for $j=1,2,\ldots$ let $1\le \ell_j^\tau<+\infty$
denote the multiplicity of $\lambda_j$, and  
let $H_j(X^\tau)\subseteq H(X^\tau)$ be the ($\ell_j^\tau$-dimensional)
eigenspace of $\lambda_j^\tau$.
If we choose orthonormal basis $(\rho_{j,k}^\tau)_{k=1}^{\ell_j^\tau}$ 
of every 
$H_j(X^\tau)$, then
$(\rho_{j,k}^\tau)_{j,k}$ is a complete orthonormal system of
$H(X^\tau)$. 
We shall henceforth leave dependence on
$\tau$ of these spectral data for
$\mathfrak{D}_{\sqrt{\rho}}^\tau$ implicit, 
and drop the suffix $\tau$.

In order to obtain spectral or eigenfunction asymptotics, it
is common to consider smoothed versions of the
spectral kernels associated to each eigenvalue
(see e.g. \cite{dg} and \cite{gs}). In the present setting,
the $\mathcal{C}^\infty$ function
\begin{equation}
\label{eqn:Pijexpandedk}
\Pi^\tau_j(x,y):=\sum_{k=1}^{\ell_j}\,\rho_{j,k}(x)
\cdot \overline{\rho_{j,k}(y)}
\end{equation}
on $X^\tau\times X^\tau$ is the Schwartz kernel of the orthogonal
projector $\Pi^\tau_j:L^2(X^\tau)\rightarrow H_j(X^\tau)$\footnote{We shall use the same symbol for operators and their distributional kernels}.
Suppose
$\chi\in \mathcal{C}^\infty_c\big((-\epsilon,\epsilon)\big)$ 
with $\epsilon>0$ suitably small; 
as in \cite{z10}, \cite{cr1} and \cite{cr2}, let us consider the \lq smoothed spectral projector\rq
\begin{equation}
\label{eqn:smoothed spectral projector}
\Pi^\tau_{\chi,\lambda}:=
\sum_{j=1 }^{+\infty}\,\hat{\chi}(\lambda-\lambda_j)\,\Pi^\tau_j,
\end{equation}
whose distributional kernel is
the $\mathcal{C}^\infty$ function of 
$(\lambda,x,y)\in \mathbb{R}\times X\times X$
\begin{equation}
\label{eqn:smoothed kernel}
\Pi^\tau_{\chi,\lambda}(x,y):=
\sum_{j=1 }^{+\infty}\,\hat{\chi}(\lambda-\lambda_j)\,\sum_{k=1}^{\ell_j}\,\rho_{j,k}(x)
\cdot \overline{\rho_{j,k}(y)};
\end{equation}
heuristically, (\ref{eqn:smoothed kernel}) is a slight
smoothing of a spectral projector kernel 
relative to a 
spectral band travelling to the right as $\lambda\rightarrow +\infty$. 
The near-diagonal asymptotics of 
(\ref{eqn:smoothed kernel})
encode information about the local concentration behaviour of the 
$\rho_{j,k}$'s, and globally on
the global asymptotic distribution of the $\lambda_j$'s.
For a discussion in the specific context
of Grauert tubes, see \cite{z12} and \cite{z17};
for applications to the Toeplitz quantization
of K\"{a}hler manifolds, with an emphasis on scaling asymptotics,
see \cite{p09}, \cite{p10}, \cite{p11}, 
\cite{p12},
\cite{p17}, \cite{p18}, \cite{zz18}, \cite{zz18}, \cite{zz219}). 
For recent  
work on general CR manifolds, see \cite{hhms}.

In \cite{cr1} an \cite{cr2}, the authors provide near-diagonal and near-graph
scaling asymtptotics for (\ref{eqn:smoothed kernel}), strikingly similar to the
Szeg\"{o} kernel asymptotics holding in the line bundle setting (\cite{z98},
\cite{sz}, \cite{bsz}) and to those holding for Toeplitz spectral projectors
in \cite{p10}.
In the present work, we shall survey the approach of \cite{cr1},
focusing on the near-diagonal case, 
and provide some  complements to their results.

Prior to precise statements,
some additional notation is in order. To begin with, we introduce
an invariant, denoted $\psi_2$, which is
related to a Hermitian vector space and
ubiquitously appears in various guises as the exponent
controlling equivariant Szeg\"{o} kernel scaling asymptotics
\cite{sz}, \cite{bsz}.

\begin{defn}
\label{defn:ps12}
Let $V$ be a finite-dimensional compòex vector space, and
let $h:V\times V\rightarrow \mathbb{C}$ be a Hermitian product. 
Then $g:=\Re(h)$ and $\omega:=-\Im(h)$ are, respectively,
an Euclidean scalar product and a symplectic form on $V$ (viewed as a real
vector space), compatible with the complex structure.  
If $\|\cdot\|_h$ is the norm function on $V$ associated to
$h$ (or $g$), let us define
$$
\psi_2^h:(u,v)\in V\times V\rightarrow 
-\frac{1}{2}\|u-v\|^2_h-\imath\,\omega(u,v)\in\mathbb{C}.
$$
Clearly, $\psi_2^h(u,v)$ is positively homogeneous
of degree $2$ in the pair $(u,v)$, and of degree
$1$ in $h$ (or $\omega$).

\end{defn}

For instance, if $V=\mathbb{C}^k$ and
$h_{st}:\mathbb{C}^k\times \mathbb{C}^k\rightarrow
\mathbb{C}$ is the standard Hermitian product,
then 
$h_{st}=g_{st}-\imath\,\omega_{st}$, where
$g_{st}$ and $\omega_{st}$ denote, respectively, the standard Euclidean product
and the standard symplectic structure on $\mathbb{R}^{2k}\cong \mathbb{C}^k$.
Then $\psi_2:=\psi^{h_{st}}:\mathbb{C}^k\times \mathbb{C}^k\rightarrow
\mathbb{C}$ is given by
$$
\psi_2(\mathbf{u},\mathbf{v}):=
-\frac{1}{2}\,\|\mathbf{u}-\mathbf{v}\|^2-\imath\,
\omega_{st}(\mathbf{u},\mathbf{v})=
-\frac{1}{2}\,\|\mathbf{u}\|^2-\frac{1}{2}\,\|\mathbf{v}\|^2
+\mathbf{u}\cdot \overline{\mathbf{v}}.
$$

If $k=\dim_\mathbb{C}(V)$, and $\mathcal{B}$
is an orthonormal basis of $(V,h)$, let $M_\mathcal{B}:V\rightarrow
\mathbb{C}^k$ be the associated unitary isomorphism.
Then $\psi^h_2=\psi_2\circ (M_\mathcal{B}\times M_\mathcal{B})$.
With this in mind, when no confusion seems likely,
we shall occasionally leave dependence on 
$h$ implicit and write $\psi_2$ for $\psi_2^h$.

We shall set $\omega:=\frac{1}{2}\,\Omega$;
thus $(\tilde{M},\omega,J)$ is a K\"{a}hler manifold,
with associated Riemannian metric 
$\tilde{\kappa}:=\frac{1}{2}\,\hat{\kappa}$.
If $x\in\tilde{M}$, with tangent space $T_x\tilde{M}$,
we correspondingly have a function 
\begin{equation}
\label{eqn:psi2x}
\psi_2^{\omega_x}: T_x\tilde{M}\times T_x\tilde{M}
\rightarrow \mathbb{C}.
\end{equation}

As mentioned, the near-diagonal asymptotic expansion  
for $\Pi^\tau_{\chi,\lambda}$ at $x\in X^\tau$
in \cite{cr1}
rests on the choice of so-called Heisenberg 
local coordinates (called normal in \cite{fs1}, \cite{fs2}).
The concept of Heisenberg local coordinates is twofold:
one first introduces Heisenberg local coordinates on
$\tilde{M}$ centered at $x$, and adapted to $X^\tau$; then
Heisenberg local coordinates on $X^\tau$ centered at $x$.
The latter will be induced by the former by restriction
and projection. More precisely, Heisenberg local coordinates
on $\tilde{M}$ will be a system of holomorphic local coordinates
centered at $x$, in which the defining equation 
$\phi^\tau:=\rho-\tau^2$ for $X^\tau$ takes a certain canonical 
form (Definition \ref{defn:heisenberglocalcoordinates}).
Let $U\subset \tilde{M}$ be an open neighbourhood of $x\in X^\tau$
on which Heisenberg local coordinates $(z_0,z_1,\ldots,z_{d-1}):
U\rightarrow \mathbb{C}^d$ (centered at $x$ and adapted to $X^\tau$)
have been chosen. 
Set $\theta:=\Re(z_0):U\rightarrow \mathbb{R}$
and $U^\tau:=X^\tau\cap U$; then 
$$
\varphi^\tau:
\left.(\theta,z_1,\ldots,z_{d-1})\right|_{U^\tau}:
{U^\tau}\rightarrow \mathbb{R}\times \mathbb{C}^{d-1}$$ 
will be a system of
Heisenberg local coordinates on $X^\tau$ centered at $x$.
We shall generally redefine the $z_i$'s and omit symbols of restriction,
and write $z'=(z_1,\ldots,z_{d-1})$.
%
Alternatively, we shall  use real notation and
write $$\varphi^\tau=(\theta,\mathbf{u}):U^\tau\rightarrow
\mathbb{R}\times \mathbb{R}^{2d-2}.$$
Furthermore, we shall often adopt the additive short-hand 
$x+(\theta,\mathbf{u}):={\varphi^\tau}^{-1}(\theta,\mathbf{u})$.
Actually, it will be convenient
to work with a slightly more restrictive class 
of local coordinates, that will be called
\textit{normal Heisenberg local coordinates adapted to $X^\tau$ at $x$}
(see \S \ref{sctn:heislochyper}).

In the line bundle setting, $X^\tau$ is a fixed-radius 
circle bundle
in the dual of the polarizing line bundle. Translation in
$\theta$ may be then be assumed to correspond to fiberwise rotation,
hence to the flow of the Reeb vector field, which is
CR-holomorphic. 

In the Grauert tube setting, instead, neither is the Reeb
flow generally CR-holomorphic nor may it be assumed to correspond to
translation in $\theta$.
Namely, let $\mathcal{R}^\tau\in \mathfrak{X}(X^\tau)$ be the Reeb vector field
of $(X^\tau,\alpha^\tau)$. While we do have
$\mathcal{R}^\tau(x)=\partial/\partial\theta|_x$
(see (\ref{eqn:reeb at x}) below), there is no reason to expect
that $\mathcal{R}^\tau=\partial/\partial\theta$ 
on $U^\tau$ (see e.g. Theorem 18.5 in \cite{fs1}).  
Hence the curves $\theta\mapsto x+(\theta,\mathbf{0})$
deflect from the trajectories of $\mathcal{R}^\tau$,
which are a rescaling of the geodesic flow.
This is a sharp difference with the line bundle situation, 
and contributes to making the derivation of the
asymptotics technically more involved.

We need one last piece of notation before formulating 
the near-diagonal scaling asymptotics 
of the smoothed spectral projectors 
$\Pi^\tau_{\chi,\lambda}$.
Let $\Gamma^\tau_t:X^\tau\rightarrow X^\tau$ be, with abuse of language 
warranted by the previous identifications, the homogeneous geodesic
flow at time $t$ (to be precise, this is the
flow of $\upsilon_{\sqrt{\rho}}$).

\begin{defn}
\label{defn:y^chi}
If $\chi\in \mathcal{C}^\infty_c(\mathbb{R})$ and 
$x\in X^\tau$, let us set
\begin{equation*}
x^\chi:=\left\{\Gamma^\tau_t(x)\,:\,t\in \mathrm{supp}(\chi)\right\}.
\end{equation*}
\end{defn}

\begin{thm}
\label{thm:absolute main thm}
Suppose $\tau\in (0,\tau_{max})$, $x\in X^\tau$,
and $\chi\in \mathcal{C}^\infty_c\big((-\epsilon,\epsilon)\big)$
for some suitably small $\epsilon>0$.
Then the following holds.
 
\begin{enumerate}
\item 
Suppose $C,\,\delta>0$ are constants. 
Then (with the notation of Definition \ref{defn:y^chi})
\begin{equation*}
\Pi^\tau_{\chi,\lambda}(x_1,x_2)=O\left(\lambda^{-\infty}  \right)
\end{equation*}
uniformly for $\mathrm{dist}_{X^\tau}\left(x_1,x_2^\chi\right)
\ge C\,\lambda^{\delta-1/2}$.

\item Suppose $x\in X^\tau$, and let 
us choose a system of normal Heisenberg local
coordinates on $X^\tau$ centered at $x$.
If $C>0$ and $\delta\in (0,1/6)$, uniformly for
$\|(\theta_j/\tau,\mathbf{v}_j)\|\le C\,\lambda^{\delta}$ there is
an asymptotic expansion
\begin{eqnarray*}
\lefteqn{\Pi^\tau_{\chi,\lambda}\left(
x+\left(\frac{\theta_1}{\sqrt{\lambda}},\frac{\mathbf{v}_1}{\sqrt{\lambda}} \right),
x+\left(\frac{\theta_2}{\sqrt{\lambda}}, 
\frac{\mathbf{v}_2}{\sqrt{\lambda}} \right)
\right)}\\
&\sim&\frac{1}{\sqrt{2\,\pi}}\cdot\left(\frac{\lambda}{2\,\pi\tau}\right)^{d-1}\,e^{\imath\,\sqrt{\lambda}\,\frac{\theta_1-\theta_2}{\tau}}\cdot
\chi (0)\cdot 
e^{\frac{1}{\tau}\,\psi_2(\mathbf{v}_1,\mathbf{v}_2)}\\
&&\cdot \left[1+\sum_{j\ge 1}\lambda^{-j/2}\,A_j^\tau(x;\theta_1,\mathbf{v}_1,
\theta_2,\mathbf{v}_2)\right],
\end{eqnarray*}
where, as a function of $(\theta_1,\mathbf{v}_1,
\theta_2,\mathbf{v}_2)$, $A_j^\tau$ is a polynomial of degree $\le 3\,j$ and parity 
$j$.
\end{enumerate}

\end{thm}

\begin{cor}
\label{cor:on diagonal asy}
Under the same assumptions as in Theorem \ref{thm:absolute main thm},
\begin{eqnarray*}
\Pi^\tau_{\chi,\lambda}\left(
x,
x \right)
\sim\frac{1}{\sqrt{2\,\pi}}\cdot\left(\frac{\lambda}{2\,\pi\tau}\right)^{d-1}\cdot 
\chi (0)\cdot 
 \left[1+\sum_{j\ge 1}\lambda^{-j}\,B_j^\tau(x)\right],
\end{eqnarray*}
for certain smooth functions $B_j\in \mathcal{C}^\infty(X^\tau)$.
\end{cor}

We recover the near-diagonal asymptotic expansion of Chang and Rabinowitz 
(Theorem 1.1 of \cite{cr1}) rescaling according to Heisenberg type, 
holding $\theta_j$ and $\mathbf{v}_j$ fixed.

\begin{cor}
\label{cor:heisenberg type asy}
If $\theta_1,\theta_2\in \mathbb{R}$ and $\mathbf{v}_1,
\,\mathbf{v}_2\in \mathbb{R}^{2n-2}$, then
\begin{eqnarray*}
\lefteqn{\Pi^\tau_{\chi,\lambda}\left(
x+\left(\frac{\theta_1}{\lambda},\frac{\mathbf{v}_1}{\sqrt{\lambda}} \right),
x+\left(\frac{\theta_2}{\lambda}, 
\frac{\mathbf{v}_2}{\sqrt{\lambda}} \right)
\right)}\\
&\sim&\frac{1}{\sqrt{2\,\pi}}\cdot\left(\frac{\lambda}{2\,\pi\tau}\right)^{d-1}\cdot 
e^{\imath\,\frac{\theta_1-\theta_2}{\tau}}\cdot\chi (0)\cdot 
e^{\frac{1}{\tau}\,\psi_2(\mathbf{v}_1,\mathbf{v}_2)}\\
&&\cdot \left[1+\sum_{j\ge 1}\lambda^{-j/2}\,R_j^\tau(x;\theta_1,\mathbf{v}_1,
\theta_2,\mathbf{v}_2)\right],
\end{eqnarray*}
where, as a function of $(\theta_1,\mathbf{v}_1,
\theta_2,\mathbf{v}_2)$,
 $R_j^\tau$ is a polynomial (again, of degree $\le 3\,j$ and parity 
$j$).

\end{cor}

Similar considerations apply to the complexified 
spectral projectors associated to the Poisson wave operator.
This issue relates to 
a fundamental extension property of the Laplacian eigenfunctions on a 
compact real-analytic Riemannian manifold, and is specific to Grauert tubes (\cite{bdm1}, \cite{bdm2},\cite{gs2}, \cite{gls}, \cite{leb}, \cite{s1}, \cite{s2}
\cite{z12}). 

Let $0=\mu_1^2< \mu_2^2< \ldots$ 
be the distinct eigenvalues
of the positive Laplacian $\Delta$ of $(M,\kappa)$,
with respective multiplicities $\ell'_j$.
Let $V_j\subset \mathcal{C}^\infty(M)$
denote the corresponding eigenspaces,
of dimension $\dim (V_j)=\ell'_j$.
For  $j=1,2,\ldots$, let $(\varphi_{j,k})_{k=1}^{\ell'_j}$ be an 
orthonormal basis of $V_j$, so that $(\varphi_{j,k})_{j,k}$ 
is a complete orthonormal
system in $L^2(M)$ for the $L^2$-norm defined by the 
Riemannian density. 

Being of class 
$\mathcal{C}^\omega$, each $\varphi_{j,k}$ admits a holomorphic extension
$\tilde{\varphi}_{j,k}$ to some open neighborhood
of $M$ in $\tilde{M}$, which \textit{a priori} depends on $(j,k)$.
Boutet de Monvel discovered that a much stronger result is true:
there exists $\tau_0>0$ such that 
\textit{every} 
$\varphi_{j,k}$ extends holomorphically to $\overline{\tilde{M}^{\tau_0}}$. 
Therefore, for $\tau\in (0,\tau_0]$ the restriction $\tilde{\varphi}_{j,k}^\tau:=
\left.\tilde{\varphi}_{j,k}\right|_{X^\tau}$ is a CR function.

This collective extension property is closely related to the 
analytic extension of the Schwartz kernel of
the Poisson operator 
$$U(\imath\,t)=e^{-\tau\,\sqrt{\Delta}}:L^2(M)\rightarrow L^2(M)$$
for $\tau>0$ (see \cite{bdm1}, \cite{gs2},
\cite{gls}, \cite{z10}, \cite{z12}, \cite{cr1} and \cite{cr2}
for discussion and motivation).
Assuming that the $\varphi_{j,k}$'s are real,
the distributional kernel of
$U(\imath\,\tau)$ admits the spectral representation
\begin{equation}
\label{eqn:sectral repr poisson0}
U(\imath\,\tau,m,n):=\sum_{j=1}^{+\infty}\,e^{-\tau\,\mu_j}\,\sum_{k=1}^{\ell'_j} 
\varphi_{j,k}(m)\,\varphi_{j,k}(n) \qquad
(\tau\in \mathbb{R}_+,\, m,n\in M),
\end{equation}
which is globally real-analytic on $M\times M$ for any $\tau>0$
\cite{z12}. If $\tau>0$ is sufficiently small, analytic extension in $m$, followed by restriction to $X^\tau$,
yields the kernel
\begin{equation}
\label{eqn:sectral repr poisson0tau}
P^\tau(x,n):=\sum_{j=1}^{+\infty}\,e^{-\tau\,\mu_j}\,\sum_{k=1}^{\ell'_j} 
\tilde{\varphi}_{j,k}(x)\,\varphi_{j,k}(n) \qquad
(x\in X^\tau,\,n\in M).
\end{equation}
As an operator, $P^\tau$ is a Fourier integral
operator with complex phase of positive type and 
order $-(d-1)/4$. It is in fact a Fourier-Hermite
operator in the sense of \cite{bg}, adapted to
a homogeneous symplectic equivalence 
$\chi_{\tau}:T^*M\setminus\{0\}\rightarrow\Sigma^\tau$ (see  (\ref{eqn:symplectic cone})). 
Hence $P^\tau$ is continuous 
$L^2(M)\rightarrow W^{\frac{d-1}{4}}(X^\tau)$, where 
$W^s(X^\tau)$ denotes the $s$-th Sobolev spaces of $X^\tau$.
More precisely, 
$P^\tau$ is a continuous isomorphism
$P^\tau:L^2(M)\rightarrow  \mathcal{O}^{\frac{d-1}{4}}(X^\tau)$,
where
$\mathcal{O}^{s}(X^\tau)$ be the space of boundary values 
of holomorphic functions on $X^\tau$ that are 
in $W^s(X^\tau)$; thus $H(X^\tau)=\mathcal{O}^2(X^\tau)$ (see \cite{z12}, \cite{s1},
\S 3 of 
\cite{z14}).

The complexified Poisson kernel $P^\tau$ governs the analytic continuation of eigenfunctions:
for any $j$,
\begin{equation}
\label{eqn:complexified eigenfunction poisson}
\tilde{\varphi}_{j,k}^\tau=e^{\tau\,\mu_{j}}\,
P^\tau (\varphi_{j,k}).
\end{equation}
The composition 
\begin{equation}
\label{eqn:UC2imathtau}
U_\mathbb{C}(2\,\imath\,\tau):=P^\tau\circ {P^\tau}^*
:H(X^\tau)\rightarrow H(X^\tau)
\end{equation}
is a Fourier integral operator with complex phase of positive type  
and of degree $-(d-1)/2$
on $X^\tau$; it is in fact a Fourier-Hermite operator adapted to the
identity of $\Sigma_\tau$,
hence with the same complex canonical relation as $\Pi^\tau$
(the relation between $U_\mathbb{C}(2\,\imath\,\tau)$
and $\Pi^\tau$ is discussed in
\S 3.1 of \cite{z14}). The definition of (\ref{eqn:UC2imathtau})
depends on the choice of a Riemannian density on $X^\tau$; 
given this, we may
identify its 
distributional kernel with the generalized function
\begin{eqnarray}
\label{eqn:poisson wave complex0}
U_\mathbb{C}(2\,\imath\,\tau,x,y)
&=&\sum_{j} e^{-2\,\tau\,\mu_{j}}\,\sum_k
\tilde{\varphi}_{j,k}^\tau(x)\,\overline{\tilde{\varphi}_{j,k}^\tau(y)}=
\sum_{j=1}^{+\infty} U^\tau_j(x,y), 
\end{eqnarray}
where $U^\tau_j\in \mathcal{C}^\infty(X^\tau\times X^\tau)$
is given by
$$
U^\tau_j(x,y):=
\sum_{k=1}^{\ell'_j}\,
P^\tau (\varphi_{j,k})(x)\,\overline{P^\tau(\varphi_{j,k})(y) }.
$$
As $\left(P^\tau (\varphi_{j,k})\right)_{j,k}$ is not an 
orthonormal system, neither
$U_\mathbb{C}(2\,\imath\,\tau)$ nor
$U^\tau_j$ are orthogonal
projectors. 
Nonetheless, $U_\mathbb{C}(2\,\imath\,\tau)$ plays a role in the asymptotic
study of analytic extensions
reminiscent of $\Pi^\tau$ (\S 6 of \cite{z10}).

Suppose as above that 
$\chi\in \mathcal{C}^\infty_0\big((-\epsilon,\epsilon)\big)$, for a suitably small
$\epsilon>0$. 
The asymptotic concentration of the complexified eigenfunctions 
pertaining to a spectral band drifting to infinity is 
probed by the following smoothed version of (\ref{eqn:poisson wave complex0}):
\begin{eqnarray}
\label{eqn:smooth tempered spproj0}
P_{\chi,\lambda}^\tau(x,y)&:=&\sum_{j} \hat{\chi}(\lambda-\mu_j)\,e^{-2\,\tau\,\mu_j}\,\sum_{k=1}^{\ell_j'}
\tilde{\varphi}_{j,k}(x)\cdot \overline{\tilde{\varphi}_{j,k}(y)}\nonumber\\
&=&\sum_{j} \hat{\chi}(\lambda-\mu_j)\,U^\tau_j(x,y)\quad (x,
y\in X^\tau).
\end{eqnarray}
Heuristically, 
$P_{\chi,\lambda}^\tau\in \mathcal{C}^\infty(X^\tau\times X^\tau)$ 
is a 
complex (tempered) analogue of the smoothed spectral projector
kernel (\ref{eqn:smoothed kernel}). 
The diagonal restriction of $P_{\chi,\lambda}^\tau$ is the non-negative function
\begin{equation}
\label{eqn:smooth tempered spproj0 diag}
P_{\chi,\lambda}^\tau(x,x):=\sum_{j,k} \hat{\chi}(\lambda-\mu_j)\,e^{-2\,\tau\,\mu_j}\,
\left|\tilde{\varphi}_{j,k}(x)\right|^2\quad (x\in X^\tau);
\end{equation}


%
 
The complexified Poisson operator is a special instance of the complexified
Poisson wave operator (see (\ref{eqn:poisson wave complex eigenf expansion}) below),
which was proved by Zelditch to be 
describable in terms of dynamical Toeplitz operators
(see e.g. \cite{z12}, especially \S 8-9). Building on this, and on
their use of the normal local coordinates of Folland and Stein, Chang and Rabinowitz
proved in \cite{cr1} a near-diagonal asymptotic expansion for
$P_{\chi,\lambda}^\tau$ very
similar to the one holding for $\Pi^\tau_{\chi,\lambda}$.
The corresponding version that we shall provide here runs parallel to Theorem \ref{thm:absolute main thm}.

\begin{thm}
\label{thm:Poisson main thm}
Suppose $\tau\in (0,\tau_{max})$, $x\in X^\tau$,
and $\chi\in \mathcal{C}^\infty_c\big((-\epsilon,\epsilon)\big)$
for some suitably small $\epsilon>0$.
Then the following holds.
 
\begin{enumerate}
\item 
Suppose $C,\,\delta>0$ are constants. 
Then (with the notation of Definition \ref{defn:y^chi})
\begin{equation*}
P^\tau_{\chi,\lambda}(x_1,x_2)=O\left(\lambda^{-\infty}  \right)
\end{equation*}
uniformly for $\mathrm{dist}_{X^\tau}\left(x_1,x_2^\chi\right)
\ge C\,\lambda^{\delta-1/2}$.

\item Suppose $x\in X^\tau$, and let 
us choose a system of normal Heisenberg local
coordinates on $X^\tau$ centered at $x$.
If $C>0$ and $\delta\in (0,1/6)$, uniformly for
$\|(\theta_j/\tau,\mathbf{v}_j)\|\le C\,\lambda^{\delta}$ there is
an asymptotic expansion
\begin{eqnarray*}
\lefteqn{P^\tau_{\chi,\lambda}\left(
x+\left(\frac{\theta_1}{\sqrt{\lambda}},\frac{\mathbf{v}_1}{\sqrt{\lambda}} \right),
x+\left(\frac{\theta_2}{\sqrt{\lambda}}, 
\frac{\mathbf{v}_2}{\sqrt{\lambda}} \right)
\right)}\\
&\sim&
\frac{\gamma_{0,0}^\tau(x)}{\sqrt{2\,\pi}}\cdot\left(\frac{1}{2\,\pi\tau}\right)^{d-1}\cdot\lambda^{\frac{d-1}{2}}\cdot e^{\imath\,\sqrt{\lambda}\,\frac{\theta_1-\theta_2}{\tau}}\cdot
\chi (0)\cdot 
e^{\frac{1}{\tau}\,\psi_2(\mathbf{v}_1,\mathbf{v}_2)}\\
&&\cdot \left[1+\sum_{j\ge 1}\lambda^{-j/2}\,F_j^\tau(x;\theta_1,\mathbf{v}_1,
\theta_2,\mathbf{v}_2)\right],
\end{eqnarray*}
for a certain constant 
$\gamma_{0,0}^\tau(x)$ (to specified below) and
where, as a function of $(\theta_1,\mathbf{v}_1,
\theta_2,\mathbf{v}_2)$, $F_j^\tau$ 
is a polynomial of degree $\le 3\,j$ and parity 
$j$.
\end{enumerate}

\end{thm}

\begin{rem}
We shall verify \textit{a posteriori} that
$\gamma_{0,0}^\tau(x)=\tau^{(d-1)/2}$ (see below).
\end{rem}

\begin{cor}
\label{cor:on diagonal asy poisson}
Under the assumptions of Theorem \ref{thm:Poisson main thm},
\begin{eqnarray*}
P^\tau_{\chi,\lambda}\left(
x,
x \right)
\sim\frac{\gamma_{0,0}^\tau(x)}{\sqrt{2\,\pi}}\cdot\left(\frac{1}{2\,\pi\tau}\right)^{d-1}\cdot\lambda^{\frac{d-1}{2}}\cdot
\chi (0)\cdot 
 \left[1+\sum_{j\ge 1}\lambda^{-j}\,Q_j^\tau(x)\right],
\end{eqnarray*}
for certain smooth functions $Q_j^\tau\in \mathcal{C}^\infty(X^\tau)$.
\end{cor}

We shall leave it to the reader to formulate the
corresponding analogue of Corollary 
\ref{cor:heisenberg type asy}.
As an application of the on-diagonal asymptotic expansions for
$\Pi^\tau_{\chi,\lambda}$ and $P^\tau_{\chi,\lambda}$, we provide local Weyl laws by a standard argument,
similar to the one in \cite{p09} for the line bundle setting
(inspired by the discussion in \cite{gs}).
The stated equality $\gamma_{0,0}^\tau(x)=\tau^{(d-1)/2}$
will follow by comparison with the local Weyl law in Proposition
3.8 of \cite{z14}. 
For $x\in X^\tau$, let us note the identities
$$
\Pi_j^\tau(x,x)=
\sum_{k=1}^{\ell_j}\,
\left|
\rho_{j,k}(x)
\right|^2,\quad
U^\tau_j(x,x)=e^{-2\,\tau\,\mu_{j}}\,
\sum_{k=1}^{\ell'_j} 
\left|\tilde{\varphi}_{j,k}^\tau(x)
\right|^2,
$$
where notation is as in (\ref{eqn:Pijexpandedk}) and (\ref{eqn:poisson wave complex0}).
Let us further define, for $x\in X^\tau$ and $\lambda\in \mathbb{R}$,
\begin{equation}
\label{eqn:weyl function Pi}
\mathcal{W}_x^\tau(\lambda):=\sum_{j:\lambda_j\le \lambda}\,
\Pi_j^\tau(x,x)=\sum_j\,H(\lambda-\lambda_j)\,
\Pi_j^\tau(x,x),
\end{equation}
%
\begin{equation}
\label{eqn:weyl function Poisson}
\mathcal{P}_x^\tau(\lambda):=\sum_{j:\lambda_j\le \lambda}\,
U_j^\tau(x,x)=\sum_j\,H(\lambda-\lambda_j)\,
U_j^\tau(x,x),
\end{equation}
where $H$ is the Heaviside function. 

\begin{prop}
\label{prop:weyl}
Uniformly in $x\in X^\tau$, we have for $\lambda\rightarrow +\infty$
\begin{equation}
\label{eqn:weyl_D}
\mathcal{W}_x(\lambda)=\frac{\tau}{d}\cdot\left(\frac{\lambda}{2\,\pi\,\tau}\right)^d
+O\left(\lambda^{d-1}\right)
\end{equation}
and 
\begin{equation}
\label{eqn:weyl poisson}
\mathcal{P}_x^\tau(\lambda)=
\frac{1}{(2\,\pi)^d}\,\left(\frac{\lambda}{\tau}\right)^{\frac{d-1}{2}}\cdot
\frac{\gamma_{0,0}^\tau(x)}{\tau^{\frac{d-1}{2}}}\,\left[
\frac{\lambda}{\frac{d-1}{2}+1}+O(1)\right].
\end{equation}
\end{prop}

Comparing with Proposition 3.8 of \cite{z14}
we finally conclude

\begin{cor}
$\gamma_{0,0}^\tau(x)=\tau^{\frac{d-1}{2}}$.
\end{cor}

\section{An example}

The geometric setting is clarified by the following example
(see \cite{pw1}, \cite{gs1}, and 
\S 2.1 of \cite{z07} for this and 
other model examples).

Consider the compact torus $M:=\mathbb{R}^d/\mathbb{Z}^d$;
its complexification is $\tilde{M}:=\mathbb{C}^d/\mathbb{Z}^d$.
If $\mathbf{u}\in M$, the (real) exponential map at $\mathbf{u}$ is
$\exp_{\mathbf{u}}(\mathbf{v})=\mathbf{u}+\mathbf{v}$ 
($\mathbf{v}\in T_{\mathbf{u}}M\cong \mathbb{R}^d$);
here, by abuse of notation, elements of $\mathbb{R}^d$ are 
identified with their classes in $M$, and the sum is meant in
$T$. The complexified exponential map 
$E_{\mathbf{u}}:T_{\mathbf{u}}^\mathbb{C}M\rightarrow\tilde{M}$ at $u$ is thus
$E_{\mathbf{u}}(\mathbf{v}_1+\imath\,\mathbf{v}_2):=
(\mathbf{u}
+\mathbf{v}_1)+\imath\,\mathbf{v}_2$. Hence the imaginary time
exponential $E:TM\cong M\times \mathbb{R}^d\rightarrow
\tilde{M}$ is 
$$
E(\mathbf{u},\mathbf{v}):=
\mathbf{u}
+\imath\, 
\mathbf{v}
\qquad (\mathbf{u}\in M,\,\mathbf{v}\in \mathbb{R}^d).
$$
In particular, writing $\mathbf{z}=\mathbf{u}+\imath\,\mathbf{v}$, 
we see
that the standard symplectic structure
$\Omega:=(\imath/2)\,\mathrm{d}\mathbf{z}\wedge \mathrm{d}\overline{\mathbf{z}}$ 
pulls back to 
$$
\mathrm{d}\mathbf{u}\wedge\mathrm{d}\mathbf{v}=
-\mathrm{d}(\mathbf{v}\,\mathrm{d}\mathbf{u})
=-\mathrm{d}\lambda_{can}=\Omega_{can}.
$$
Hence, the pull-back of the standard complex structure on
$\tilde{M}$ by $E$ is indeed compatible with $\Omega_{can}$.

The geodesic $\gamma =\gamma_{\mathbf{u},\mathbf{v}}$
with initial conditions $(\mathbf{u},\mathbf{v})$ is
$\gamma (\sigma)=\mathbf{u}+\sigma\,\mathbf{v}$, with tangent lift
$\dot{\gamma}(t):=(\mathbf{u}+\sigma\,\mathbf{v},\mathbf{v})$. The composition $E\circ \psi_\gamma$,
with
$\psi_\gamma:\mathbb{C}\rightarrow TM$ as 
in (\ref{eqn:riemann fol}),
is then 
\begin{equation}
\label{eqn:riem fol torus}
E\circ \psi_\gamma(\sigma+\imath\,\tau)=
E(\mathbf{u}+\sigma\,\mathbf{v},\tau\,\mathbf{v})=\mathbf{u}
+(\sigma+\imath\,\tau)\,\mathbf{v},
\end{equation}
which is holomorphic. Hence $J_{ad}={E}^*(J)$
is the adapted complex structure of $TM$,
in the terminology of \cite{ls}.

Let us henceforth identify $TM$ with $\tilde{M}$ by $E$, and
consider the functions
$$
\rho(\mathbf{u}+\imath\,\mathbf{v}):=\|\mathbf{v}\|^2,\quad 
\sqrt{\rho}(\mathbf{u}+\imath\,\mathbf{v})=\|\mathbf{v}\|.
$$
We have, with $z_j:=u_j+\imath\,v_j$,
$$
\partial\rho=\frac{1}{\imath}\sum_j\,v_j\,\mathrm{d}z_j,\quad
\alpha:=\Im (\partial \rho)=-\sum_j\,v_j\,\mathrm{d}u_j=-\lambda_{can}
=-\Im\left(\overline{\partial}\rho\right).
$$
Furthermore, 
$$
\imath\,\partial\,\overline{\partial}\rho
=\frac{\imath}{2}\,
\sum_j\,\mathrm{d}z_j\wedge\mathrm{d}\overline{z}_j
=-\mathrm{d}\lambda_{can}=\Omega_{can}.$$
Thus $\rho$ is a K\"{a}hler potential for $\Omega_{can}$.

A direct computation shows that
the differential form 
$\imath\,\partial\,\overline{\partial}\,\sqrt{\rho}$,
which is defined where $\mathbf{v}\neq \mathbf{0}$, is given by
\begin{eqnarray*}
\imath\,\partial\,\overline{\partial}\,\sqrt{\rho}
&=&\frac{\imath}{2\,\rho^{3/2}}\,\left[
\|\mathbf{v}\|^2\,\sum_j\mathrm{d}z_j\wedge\mathrm{d}\overline{z}_j-\sum_{j,k}v_j\,v_k\,\mathrm{d}z_j\wedge\mathrm{d}\overline{z}_k\right],
\end{eqnarray*}
with constant rank $2d-2$. Furthermore, its $2$-dimensional kernel
at $\mathbf{u}+\imath\,\mathbf{v}$ is generated by
the tangent vectors $(\mathbf{v},\mathbf{0})$ and $(\mathbf{0},\mathbf{v})$;
by (\ref{eqn:riem fol torus}), this is the tangent space of the
Riemann foliation, which therefore coincides with the Monge-Amp\`{e}re
foliation.

Since $\mathrm{d}\rho=2\,\sum_jv_j\,\mathrm{d}v_j$
and $\Omega_{can}=\sum_j\mathrm{d}u_j\wedge\mathrm{d}v_j$,
letting $\upsilon_\rho$ denote 
the Hamiltonian vector field of $\rho$ with respect to $\Omega$ 
we have
$$
\upsilon_\rho=2\,\sum_j\,v_j\,\frac{\partial}{\partial u_j},\qquad
\iota(\upsilon_\rho)\,\alpha=-2\,\|\mathbf{v}\|^2=-2\,\rho,
\qquad 
\upsilon_{\sqrt{\rho}}=\frac{1}{\|\mathbf{v}\|}\,
\sum_j v_j\,\frac{\partial}{\partial u_j}.
$$
Let us set
\begin{equation}
\label{eqn:Reeb torus}
\mathcal{R}:=-\frac{1}{2\,\rho}\,\upsilon_\rho=-\frac{1}{2}\,\upsilon_{\ln (\rho)}=
-\frac{1}{\sqrt{\rho}}\,\upsilon_{\sqrt{\rho}}.
\end{equation}
Then $\alpha(\mathcal{R})\equiv 1$, and $\mathcal{R}$ is itself a Hamiltonian vector field
$$
\iota(\mathcal{R})\,\mathrm{d}\alpha=\iota(\mathcal{R})\,\Omega=
\mathrm{d}\ln\left(\rho^{-1/2}\right).
$$
For $\tau>0$, set $X^\tau:=\left\{\mathbf{u}+\imath\,\mathbf{v}\,:\,
\|\mathbf{v}\|=\tau\right\}$, so that $X^\tau$ is the boundary of a strictly pseudoconvex domain in $\tilde{M}$. 
Let $\jmath^\tau:X^\tau\hookrightarrow \tilde{M}$
be the inclusion, and set $\alpha^\tau:={\jmath^\tau}^*(\alpha)$. 
Then $\mathcal{R}$ is tangent to $X^\tau$
and restricts to a vector field $\mathcal{R}^\tau$ on $X^\tau$,
which satisfies
$$
\iota(\mathcal{R}^\tau)\,\alpha^\tau= 1,\qquad \iota(\mathcal{R}^\tau)\,
\mathrm{d}\alpha^\tau=0.
$$
In other words, $\mathcal{R}^\tau$ is the Reeb vector field of the
contact manifold $(X^\tau,\alpha^\tau)$.

Similarly, $\upsilon_{\sqrt{\rho}}$ is also tangent to $X^\tau$,
and 
in view of (\ref{eqn:Reeb torus}) it
restricts along $X^\tau$ to a vector field $\upsilon_{\sqrt{\rho}}^\tau$
satisfying 
$\upsilon_{\sqrt{\rho}}^\tau=-\tau\,\mathcal{R}^\tau$.

We have $\|\upsilon_{\sqrt{\rho}}\|=1$, hence $\sqrt{\rho}$ is a 
distance function on the Riemannian manifold $(\tilde{M},\kappa)$.
The (Hamiltonian) flow of $\upsilon_{\sqrt{\rho}}$ is 
$$
\Gamma_t(\mathbf{u}+\imath\,\mathbf{v})=
\left(\mathbf{u}+t\,\frac{\mathbf{v}}{\|\mathbf{v}\|}\right)+
t\,\mathbf{v},
$$
which (with the identification provided by $E$) corresponds to the 
homogeneous geodesic flow.

Finally, let us consider the cones 
$$
\Sigma^\tau =\left\{(\mathbf{u}+\imath\,\mathbf{v},
r\,\alpha^\tau_{\mathbf{u}+\imath\,\mathbf{v}})\,:\,
\|\mathbf{v}\|=\tau,\,r>0\right\}\subseteq T^*X^\tau,
$$
where as above $\alpha^\tau=-
\sum_j v_j\,\mathrm{d}u_j$. The symbol of the differential operator
$D_{\sqrt{\rho}}^\tau:=\imath\,\upsilon_{\sqrt{\rho}}^\tau$ along
$\Sigma^\tau$ is then 
\begin{eqnarray*}
\sigma (D_{\sqrt{\rho}}^\tau)\left(r\,\alpha^\tau_{\mathbf{u}+\imath\,\mathbf{v}}   \right)   
&=&e^{\imath \,r\,\sum_kv_k\,u_k}\,\imath\,\upsilon_{\sqrt{\rho}}^\tau\left(e^{-\imath \,r\,\sum_kv_k\,u_k} \right)\\
&=& 
\frac{\imath}{\|\mathbf{v}\|}\,
\sum_j v_j\,\frac{\partial}{\partial u_j} 
\left( 
-\imath \,r\,\sum_kv_k\,u_k   
\right)    =r\,\|\mathbf{v}\|=r\,\tau>0.
\end{eqnarray*}

\section{Preliminaries}
\label{scnt:preliminaries}

\subsection{Notation}

For the reader's convenience, we collect here some of the  notation and conventions adopted in the paper.

\paragraph{Fourier transform.}
Given $f\in \mathcal{S}^d(\mathbb{R}^k)$ 
(smooth functions of rapid decay on
$\mathbb{R}^k$) its Fourier transorm $\hat{f}\in\mathcal{S}^d(\mathbb{R}^k)$
is given by 
\begin{equation}
\label{eqn:fourtransform}
\hat{f}(\xi):=\frac{1}{(2\,\pi)^{k/2}}\,
\int_{\mathbb{R}^k}e^{-\imath\,\langle\xi,x\rangle}\,f(x)\,
\mathrm{d}x\quad (\xi\in \mathbb{R}^k),
\end{equation}
so that the Fourier inversion formula has the form
$$
f(x)=\frac{1}{(2\,\pi)^{k/2}}\,
\int_{\mathbb{R}^k}e^{\imath\,\langle\xi,x\rangle}\,\hat{f}(\xi)\,
\mathrm{d}\xi\quad (x\in \mathbb{R}^k).
$$

\paragraph{Densities and functions.}
Our manifolds will be endowed with naturally given 
volume forms, and we shall identify densities, half-densities
and functions.

\paragraph{Schwartz kernels.}
Given a manifold $N$, we shall use the same letter 
for a continuous linear operator
$F:\mathcal{C}^\infty_0(N)\rightarrow \mathcal{D}'(N)$
and its
distributional kernel 
$F(\cdot,\cdot)\in \mathcal{D}'(N\times N)$. 

\paragraph{Induced vector fields.}
Given a smooth action of a Lie group $G$ on a manifold $R$,
for any $\xi\in \mathfrak{g}$ (the Lie algebra of $G$) 
we shall denote by $\xi_R\in 
\mathfrak{X}(R)$ (the Lie algebra of smooth vector fields on $R$)
the vector field induced by $\xi$.

\paragraph{Riemannian and K\"{a}hler structures, I.}
$(M,\kappa)$ is the reference real-analytic Riemannian manifold,
$\rho$ the associated strictly pseudoconvex function defined on
a neighbourhood on $M$ in $\tilde{M}$, $\Omega=\imath\,\partial\,\overline{\partial}\,\rho$ the corresponding K\"{a}hler form, 
$\hat{\kappa}:=\Omega(\cdot,J\cdot)$ the induced Riemannian structure on $\tilde{M}^\epsilon$; thus $(M,\kappa)$ is a Riemannian submanifold of 
$(\tilde{M}^\epsilon,\hat{\kappa})$.

\paragraph{Hamiltonian vector fields.} Given $f\in \mathcal{C}^\infty (\tilde{M}^\epsilon)$, $\upsilon_f\in \mathfrak{X}(\tilde{M}^\epsilon)$
will denote its Hamiltonian vector field w.r.t. $\Omega$: 
$\mathrm{d}f=\Omega
(\upsilon_f,\cdot)$.

\paragraph{Riemannian and K\"{a}hler structures, II.}
We shall also use $\omega:=\frac{1}{2}\,\Omega$, and the corresponding Riemannian
metric 
$\tilde{\kappa}:=\frac{1}{2}\,\hat{\kappa}$.

\paragraph{Riemannian distance function} 
$\hat{\kappa}^\tau$ is the Riemannian metric on $X^\tau$ given by the
restriction of $\hat{\kappa}$, and 
$\mathrm{dist}_{X^\tau}:X^\tau\times X^\tau\rightarrow \mathbb{R}$ 
denotes
the corresponding distance function.

\paragraph{Reeb vector fields.} $\mathcal{R}$ is the 
\textit{vector field of $(M,\kappa)$} (Definition \ref{defn:reeb Mkappa}); 
it is tangent to every $X^\tau$, and restricts along $X^\tau$ to the Reeb
vector field $\mathcal{R}^\tau$ of $(X^\tau,\alpha^\tau)$.

\paragraph{Volume form.} $\mathrm{vol}^R_{X^\tau}$ is the 
Riemannian volume form on $X^\tau$, in terms of which the Hilbert structure
on $L^2(X^\tau)$ is defined.

\paragraph{Geodesic flow.} $\Gamma^\tau_t:X^\tau\rightarrow X^\tau$ denotes the homogeneous
geodesic flow along $X^\tau$ at time $t$.

\subsection{The geometric setting}
\label{sctn:geometric setting}

\subsubsection{The homogeneous geodesic flow}

As remarked, $E^\epsilon$ intertwines the
homogenous geodesic flow on 
$T^\epsilon M\setminus M_0$
(that is, the Hamiltonian flow with respect to
$\Omega_{can}$ of the norm function induced by $\kappa$
on $TM$) with the Hamiltonian flow of $\sqrt{\rho}$
on $\tilde{M}\setminus M$ with respect to
$\Omega$. Neither flow is generally holomorphic
(of course, the former is if and only if so is the latter). 

Recall that we identify $TM$ and $T^\vee M$ by means of $\kappa$.

\begin{lem}
\label{lem:primitive geod inv}
The canonical $1$-form $\lambda_{can}$ on $T^\vee M$ is invariant
under the homogenous geodesic flow.
Similarly, $\alpha$ in (\ref{eqn:defn alpha}) is invariant under 
the flow of $\upsilon_{\sqrt{\rho}}$.
\end{lem}

\begin{proof}
The two statements are equivalent by (\ref{eqn:taut 1form}).
To verify the former, we may work in a local coordinate chart
$(q,p)$ for the cotangent bundle associated 
to a system of local coordinates $q$ for $M$. 
Let $K=(k^{ij})$, where
$\kappa^{ij}=\kappa^{ij}(q)$, denote the inverse metric tensor, 
and set $\varrho:=\|\cdot\|^2$; then
\begin{equation*}
\sqrt{\varrho} =\sqrt{\kappa^{ij}p_i\,p_j}=\sqrt{p^t\,K\,p}.
\end{equation*}
Since locally $\Omega_{can}=\mathrm{d}q\wedge\mathrm{d}p$, the Hamiltonian vector field of $\sqrt{\rho}$ is
\begin{eqnarray*}
V_{\sqrt{\varrho}}& =&\left(\frac{\partial\sqrt{\varrho}}{\partial p}\right)^t\,\frac{\partial}{\partial q}-\left(\frac{\partial\sqrt{\varrho}}{\partial q}\right)^t\,\frac{\partial}{\partial p}\nonumber\\
&=&  \frac{1}{\sqrt{\varrho}}\,(p^t\,K)\,\,\frac{\partial}{\partial q}-
\left(\frac{\partial\sqrt{\varrho}}{\partial q}\right)^t\,\frac{\partial}{\partial p}.
\end{eqnarray*}
Since locally $\lambda_{can}=p\,\mathrm{d}q$, 
$$
\iota(V_{\sqrt{\varrho}})\,\lambda_{can}= \frac{1}{\sqrt{\varrho}}\,(p^t\,K\,p)=\sqrt{\varrho}.
$$
Hence the Lie derivative of $\lambda_{can}$ along $V_{\sqrt{\varrho}}$ is
$$
L_{V_{\sqrt{\varrho}}}(\lambda_{can})=\mathrm{d}\big(\iota(V_{\sqrt{\varrho}})\,\lambda_{can}\big)-\iota(V_{\sqrt{\varrho}})\,\Omega=\mathrm{d}\sqrt{\varrho}-\mathrm{d}\sqrt{\varrho}=0.
$$
\end{proof}

For a more general statement, see Lemma \ref{lem:Lie der ham}.

\subsubsection{The Reeb vector field of $(M,k)$}

As shown in \cite{gs1}, the condition that $\sqrt{\rho}$
satisfies the complex Monge-Amp\`{e}re equation may be reformulated
in terms of the norm of the gradient of $\rho$. 
Let us briefly recall the argument in \cite{gs1}. Let 
$\Xi\in \mathfrak{X}(\tilde{M})$ be defined by the identity
\begin{equation}
\label{eqn:identity for Theta}
\iota(\Xi)\,\Omega=\alpha.
\end{equation}
Then $\Xi$ is the gradient vector field with respect to
$\hat{\kappa}$ of $\rho/2$, that is,
$\Xi=J(\upsilon_\rho)/2$.

\begin{lem} (Guillemin and Stenzel)
\label{lem:MA eq and norm}
The following conditions are equivalent:
\begin{enumerate}
\item $\sqrt{\rho}$ satisfies the complex Monge-Amp\`{e}re
equation on $\tilde{M}\setminus M$;
\item $\Xi(\rho)=2\,\rho$.
\end{enumerate}

\end{lem}

\begin{proof}
(See \cite{gs1})
Given $f\in \mathcal{C}^\infty(\mathbb{R}_+)$, 
let us consider the
composition $f(\rho):\tilde{M}\setminus M\rightarrow \mathbb{R}$.
We have
\begin{eqnarray}
\label{eqn:partial overlinepartial}
\partial\,\overline{\partial}f(\rho)&=&
\partial \left[f'(\rho)\,\overline{\partial}\rho\right]\\
&=&
f''(\rho)\,\partial\rho\wedge\,\overline{\partial}\rho+f'(\rho)\,
\partial\,\overline{\partial}\rho\nonumber\\
&=&f''(\rho)\,\partial\rho\wedge\,\overline{\partial}\rho
-\imath\,f'(\rho)\,
\Omega.\nonumber
\end{eqnarray}
Furthermore,
\begin{eqnarray}
\label{eqn:drhoalpha}
\mathrm{d}\rho\wedge\alpha&=&\frac{1}{2\,\imath}\,\left(
\partial\rho+\overline{\partial}\rho\right)\wedge
\left(\partial\rho-\overline{\partial}\rho\right)\\
&=&\frac{1}{2\,\imath}\,
\left(-2\,\partial\rho\wedge\overline{\partial}\rho
\right)\nonumber\\
&=&-\frac{1}{\imath}\,\partial\rho\wedge\overline{\partial}\rho
\nonumber\\
&=&\imath\,\partial\rho\wedge\overline{\partial}\rho.\nonumber
\end{eqnarray}
Hence
\begin{eqnarray}
\label{eqn:dth power}
\lefteqn{\left[\partial\,\overline{\partial}f(\rho)\right]^n}\\
&=&
(-\imath)^d\,f'(\rho)^d\,
\Omega^d+d\,(-\imath)^{d-1}\,f''(\rho)\,f'(\rho)^{d-1}
\,\partial\rho\wedge\,\overline{\partial}\rho\wedge
\Omega^{d-1}\nonumber\\
&=&
(-\imath)^d\,f'(\rho)^d\,
\Omega^d+d\,(-\imath)^{d}\,f''(\rho)\,f'(\rho)^{d-1}
\,\mathrm{d}\rho\wedge\alpha\wedge
\Omega^{d-1}\nonumber\\
&=&(-\imath)^d\,f'(\rho)^d\,\left[
\Omega^d+d\,\frac{f''(\rho)}{f'(\rho)}\,
\mathrm{d}\rho\wedge\alpha\wedge
\Omega^{d-1}.\nonumber
\right]
\end{eqnarray}
On the other hand,
$$
\iota(\Xi)\,\Omega^d=d\cdot\iota(\Xi)\Omega\wedge
\Omega^{d-1}=
d\cdot\alpha\wedge\Omega^{d-1}.
$$
Therefore
\begin{eqnarray}
\label{eqn:dth power1}
\left[\partial\,\overline{\partial}f(\rho)\right]^n
&=&(-\imath)^d\,f'(\rho)^d\,\left[
\Omega^d+\frac{f''(\rho)}{f'(\rho)}\,
\mathrm{d}\rho\wedge\iota(\Xi)\,\Omega^d
\right].
\end{eqnarray}
Furhtermore, since $\mathrm{d}\rho\wedge \Omega^d=0$
\begin{eqnarray}
\label{eqn:contractionXi}
0=\iota(\Xi)\,\left(\mathrm{d}\rho\wedge \Omega^d\right)=
\mathrm{d}\rho(\Xi)\,\Omega^d-\mathrm{d}\rho\wedge \iota(\Xi)\,\Omega^d.
\end{eqnarray}
We conclude
\begin{eqnarray}
\label{eqn:dth power2}
\left[\partial\,\overline{\partial}f(\rho)\right]^n
&=&(-\imath)^d\,f'(\rho)^d\,\left[
\Omega^d+\frac{f''(\rho)}{f'(\rho)}\,
\mathrm{d}\rho(\Xi)\,\Omega^d
\nonumber
\right]\\
&=&(-\imath)^d\,f'(\rho)^d\,\left[
1+\frac{f''(\rho)}{f'(\rho)}\,
\mathrm{d}\rho(\Xi)
\right]\,\Omega^d.
\end{eqnarray}
If
$f=\sqrt{\cdot}$, then
$$
f'(x)=\frac{1}{2}\,x^{-1/2},\quad
f''(x)=-\frac{1}{4}\,x^{-3/2}\quad
\Rightarrow\quad
\frac{f''(x)}{f'(x)}=-\frac{1}{4}\,x^{-3/2}\,2\,x^{1/2}=
-\frac{1}{2\,x}.
$$
Summing up,
\begin{eqnarray}
\label{eqn:dth power3}
\left[\partial\,\overline{\partial}f(\rho)\right]^n
&=&(-\imath)^d\,f'(\rho)^d\,\left[
\Omega^d+\frac{f''(\rho)}{f'(\rho)}\,
\mathrm{d}\rho(\Xi)\,\Omega^d
\right]\nonumber\\
&=&(-\imath)^d\,f'(\rho)^d\,\left[
1-\frac{1}{2\,\rho}\,
\mathrm{d}\rho(\Xi)
\right]\,\Omega^d,
\end{eqnarray}
which vanishes if and only if 
$\mathrm{d}\rho(\Xi)=2\,\rho$.

\end{proof}

\begin{cor}
\label{cor:sq norm Xi}
With $\Xi$ as in (\ref{eqn:identity for Theta}), 
the square norm of $\Xi$ with respect to $\hat{\kappa}$
is 
$
\|\Xi\|_{\hat{\kappa}}^2=\rho$.

\end{cor}

\begin{cor}
\label{cor:alphaupsilonT}
$\alpha(\upsilon_\rho)=-2\,\rho$.
\end{cor}                     

\begin{proof}
Since $\mathrm{grad}(\rho)=J(\upsilon_\rho)$,
\begin{eqnarray}
4\,\rho=\|\mathrm{grad}(\rho)\|_{\hat{\kappa}}^2&=&\mathrm{d}\rho(\mathrm{grad}(\rho))=\left(\partial\rho+\overline{\partial}\rho\right)(J\,\upsilon_\rho)\\
&=&\imath\,\left(\partial\rho (\upsilon_\rho)-
\overline{\partial}\rho (\upsilon_\rho)\right)=
-2\,\alpha(\upsilon_\rho).\nonumber
\end{eqnarray}
\end{proof}

Let us set
\begin{equation}
\label{eqn:defnT}
\mathcal{R}   :=-\frac{\upsilon_{\rho/2}}{\rho}=
-\frac{\upsilon_{\rho}}{2\,\rho}=\frac{1}{\rho}\,
J(\Xi)
\in \mathfrak{X}(\tilde{M}\setminus M).
\end{equation}
\begin{cor}
\label{cor:reeb}
%
$\mathcal{R}$
is uniquely determined
in $\mathfrak{X}(\tilde{M}\setminus M)$
by the conditions
\begin{enumerate}
\item $\alpha (\mathcal{R})=1$;
\item $\mathcal{R}\in \mathcal{C}^\infty
(\tilde{M}\setminus M)\,\upsilon_\rho$.
\end{enumerate}

\end{cor}

\begin{defn}
\label{defn:reeb Mkappa}
We shall call $\mathcal{R}$ in (\ref{eqn:defnT}) the \textit{Reeb vector
field of $(M,\kappa)$}.
\end{defn}

The motivation for this definition is the following. Suppose
$\tau\in (0,\epsilon)$. Since $\mathcal{R}$ is tangent
to $X^\tau$, it restricts to a vector field $\mathcal{R}^\tau
\in \mathfrak{X}(X^\tau)$.
Furthermore, the Lie derivative
\begin{equation}
\label{eqn:liederTalpha}
L_\mathcal{R}(\alpha)=\iota(\mathcal{R})\,\Omega
+\mathrm{d}(1)=-\frac{1}{2\,\rho}
\,\mathrm{d}\rho
\end{equation}
has vanishing pull-back to $X^\tau$; therefore
$L_{\mathcal{R}^\tau}(\alpha^\tau)=0$. In other words, 
$\mathcal{R}^\tau\in \mathfrak{X}(X^\tau)$ is the (genuine)
Reeb vector field of $(X^\tau,\alpha^\tau)$.
We also have
\begin{equation}
\label{eqn:Tsqrtrho}
\mathcal{R}=-\frac{1}{\sqrt{\rho}}\,\upsilon_{\sqrt{\rho}}.
\end{equation}

Corollary \ref{cor:sq norm Xi} implies:
\begin{equation}
\label{eqn:normsT}
\|\upsilon_{\sqrt{\rho}}\|_{\hat{\kappa}}^2=1,\quad
\|\mathcal{R}\|_{\hat{\kappa}}^2=\frac{1}{\rho},
\end{equation}
hence $\sqrt{\rho}$ is a distance function on $\tilde{M}\setminus M$
for $\hat{\kappa}$.
While the flow of $\upsilon_{\sqrt{\rho}}$ is intertwined
by $E^\epsilon$ with the 
homogeneous geodesic flow, the trajectories of
the gradient vector field
$\mathrm{grad}_{\sqrt{\rho}}=J(\upsilon_{\sqrt{\rho}} )$ are unit speed geodesics for $\hat{\kappa}$,
perpendicular to the hypersurfaces $X^\tau$ and minimizing
the distance between them (see \S 3 of \cite{ls}).
By (\ref{eqn:Tsqrtrho}), the flows of $\mathcal{R}^\tau$ and
$\upsilon_{\sqrt{\rho}}^\tau$ on $X^\tau$ are related by
a rescaling by the factor $-1/\tau$ in the time variable.

\subsubsection{The volume form on $X^\tau$}
\label{sctn:vol form cr}

The Riemannian volume form on the K\"{a}hler manifold 
$(\tilde{M}^\epsilon,\Omega,J)$ is $$\mathrm{vol}_{\tilde{M}^\epsilon}:=\frac{1}{d!}\,\Omega^{\wedge d};$$ it 
pulls back under $E^\epsilon$ to the symplectic volume form 
$\mathrm{vol}_{can}:=\frac{1}{d!}\,\Omega_{can}^{\wedge d}$.
 
There are various natural alternatives in the literature 
for a volume form on $X^\tau$;
let us dwell to specify the choice in this paper.

Given that that 
$\mathrm{grad}_{\sqrt{\rho}}$ 
is a unit normal vector field to $X^\tau$
by (\ref{eqn:normsT}), the 
\textit{Riemannian volume form} is
\begin{eqnarray}
\label{eqn:rvol form tau}
\mathrm{vol}^R_{X^\tau}&:=&{\jmath^\tau}^*\left(
\iota (\mathrm{grad}_{\sqrt{\rho}})\,\mathrm{vol}_{\tilde{M}^\epsilon}\right)
\\
&=&
{\jmath^\tau}^*\left(
\iota (\mathrm{grad}_{\sqrt{\rho}})\,\Omega\wedge
\frac{1}{(d-1)!}\,\Omega^{\wedge (d-1)}\right).\nonumber
\end{eqnarray}
An alternative choice is the \textit{contact volume form}
\begin{equation}
\label{eqn:cvol form tau}
\mathrm{vol}^C_{X^\tau}:={\jmath^\tau}^*\left(
\alpha\wedge
\frac{1}{(d-1)!}\,\Omega^{\wedge (d-1)}
\right).
\end{equation}


Let us clarify the relation between $\mathrm{vol}^R_{X^\tau}$
and $\mathrm{vol}^C_{X^\tau}$.

\begin{lem}
\label{lem:alternative for alfa}
$\alpha=\sqrt{\rho}\,\iota(\mathrm{grad}_{\sqrt{\rho}})\,\Omega
=\frac{1}{2}\,\iota(\mathrm{grad}_\rho)\,\Omega$.
\end{lem}

\begin{proof}
The two equalities are clearly equivalent. As to the latter,
\begin{eqnarray*}
\iota(\mathrm{grad}_\rho)\,\Omega&=&\Omega (\mathrm{grad}_\rho,\cdot)
=\Omega (J\upsilon_\rho,\cdot)=-\Omega(\upsilon_\rho,J\cdot)
=-\mathrm{d}\rho\circ J\\
&=&-\imath\,\left(\partial\rho-\overline{\partial}\rho\right)
=2\,\alpha.
\end{eqnarray*}
\end{proof}

\begin{cor}
\label{cor:rel vol for RC}
$\mathrm{vol}^C_{X^\tau}=\tau\cdot\mathrm{vol}^R_{X^\tau}$.
\end{cor}

Our choice for a volume form on $X^\tau$ will be 
$\mathrm{vol}^R_{X^\tau}$. Let us consider its 
homogeneity properties.
For $\lambda>0$ let $\delta_\lambda:T^\vee M\rightarrow T^\vee M$
denote fibrewise dilation by $\lambda$. 
Then $\delta_\lambda$ 
is intertwined with a diffeomorphism $\delta'_\lambda:
M^{\epsilon/\lambda}\rightarrow M^\epsilon$. 

Let $\Xi_{can}\in \mathfrak{X}(T^\vee M)$ be the vector field
correlated with
$\Xi$ by $E^\epsilon$. Since locally
$\Omega=\mathrm{d}q\wedge\mathrm{d}p$ and $\alpha=-p\,\mathrm{d}q$,
we have
$$
\iota (\Xi_{can})\,\Omega_{can}=\alpha_{can}\,\Rightarrow
\Xi_{can}=p\,\frac{\partial}{\partial p},
$$
which is homogeneous 
of degree zero with respect to $\delta_\lambda$.
Therefore the same holds of $\Xi$ with respect to
$\delta'_\lambda$. Since $\sqrt{\rho}$ is homogenous of degree
$1$, 
$$
\mathrm{grad}_{\sqrt{\rho}}=\frac{1}{\sqrt{\rho}}\,\Xi
$$
is homogenous of degree $-1$. Hence by (\ref{eqn:rvol form tau})
$\mathrm{vol}^R_{X^\tau}$ is homgeneous of degree $d-1$
in $\tau$. We conclude

\begin{lem}
\label{lem:homogenous surface Xtau}
For a constant $D>0$,
$\mathrm{Vol}^R(X^\tau):=\int_{X^\tau}\mathrm{vol}^R_{X^\tau}=D\,\tau^{d-1}$.
\end{lem}

\subsubsection{Induced vector fields and Hamiltonians}

The K\"{a}hler structure  makes $T\tilde{M}^\epsilon$ into a
complex Hermitian vector bundle.
Being everywhere non-vanishing, $\upsilon_{\sqrt{\rho}}$ spans
on $\tilde{M}^\epsilon\setminus M$
a $1$-dimensional complex subundle $\mathcal{V}$ of 
$(T\tilde{M},J)$:
$$
\mathcal{V}_x:=
\mathrm{span}_\mathbb{C}\big(\upsilon_{\sqrt{\rho}}(x)  \big)
=\mathrm{span}_{\mathbb{R}}\big(\upsilon_{\sqrt{\rho}}(x),\,
\mathrm{grad}_{\sqrt{\rho}}(x)  \big)
\qquad (x\in \tilde{M}^\epsilon\setminus M).
$$
Hence there is on $\tilde{M}^\epsilon\setminus M$
a decomposition of $(T\tilde{M},J)$
as the orthogonal direct sum of complex vector sub-bundles
\begin{equation}
\label{eqn:directsum ttildeM}
T\tilde{M}=\mathcal{V}\oplus 
\mathcal{H},\quad
\text{where}\quad
\mathcal{H}:=
\mathcal{V}^{\perp}.
\end{equation}

Let $\mathcal{T}\subset \mathcal{V}$ be the \textit{real}
vector subbundle generated on $\tilde{M}^\epsilon\setminus M$ by 
$\upsilon_{\sqrt{\rho}}$;
thus
\begin{equation}
\label{eqn:ker drho}
\mathcal{T}\oplus \mathcal{H}=\ker(\mathrm{d}\rho)
\subset T\tilde{M}.
\end{equation}
The $\mathcal{C}^\infty$ 
sections of $\mathcal{T}\oplus \mathcal{H}$ are the 
smooth vector fields that are tangent to $X^\tau$, for every
$\tau\in (0,\epsilon)$.
The decomposition (\ref{eqn:ker drho}) restricts to a 
corresponding orthogonal direct sum decomposition for the tangent bundle $TX^\tau$:
\begin{equation}
\label{eqn:directsum tXtau}
T X^\tau = \mathcal{T}^\tau\oplus \mathcal{H}^\tau,
\end{equation}
where $\mathcal{T}^\tau:={\jmath^\tau}^*(\mathcal{T})$
and
$\mathcal{H}^\tau:={\jmath^\tau}^*(\mathcal{H})$.
If $x\in X^\tau$, 
$\mathcal{T}^\tau(x)=\mathrm{span}_\mathbb{R}\big(\mathcal{R}^\tau(x)\big)$, and
$\mathcal{H}^\tau(x)$ is the maximal complex subspace of 
$T_xX^\tau$.

\begin{lem}
We have
$$
\mathcal{T}\oplus \mathcal{H}=\ker(\alpha),\quad
\mathcal{H}^\tau=\ker(\alpha^\tau).
$$
\end{lem}

\begin{proof}
Suppose $x\in \tilde{M}^\epsilon\setminus M$ and 
decompose 
$\mathcal{R}_x$ according to the complex structure:
$\mathcal{R}_x=\mathcal{R}_x^{1,0}+\mathcal{R}_x^{0,1}$.
Since $\mathrm{d}\rho(\mathcal{R})=0$, 
\begin{eqnarray}
2\,\imath\,\alpha_x\big(J_x(\mathcal{R}_x)\big)&=&
\imath\,\left(\partial_x\rho-\overline{\partial}_x\rho\right)
\left(\mathcal{R}_x^{1,0}-\mathcal{R}_x^{0,1}\right)\\
&=&\imath\,\left[\partial_x\rho\left(\mathcal{R}_x^{1,0}\right)+
\overline{\partial}_x\rho\left(\mathcal{R}_x^{0,1}\right)\right]=\imath\,\mathrm{d}_x\rho(\mathcal{R}_x)=0.\nonumber
\end{eqnarray}
Hence $\mathcal{T}_x\subseteq \ker(\alpha_x)$.
One argues similarly for $\mathcal{H}$.
\end{proof}

For a vector bundle $\mathcal{E}$ on $\tilde{M}^\epsilon\setminus M$, let 
$\Gamma (\mathcal{E})$ be the space of its smooth sections.
Any $V\in \Gamma(\mathcal{T}\oplus \mathcal{H})=\Gamma(\mathcal{T})\oplus \Gamma(\mathcal{H})\subseteq \mathfrak{X}(\tilde{M}^\epsilon\setminus M)$
may be uniquely decomposed as
\begin{equation}
\label{eqn:vector field tgtau}
V=V^\sharp-\varphi\,\mathcal{R},
\end{equation}
where, 
since $\alpha (\mathcal{R})=1$,
$$\varphi:=-\alpha(V)\in \mathcal{C}^\infty(\tilde{M}\setminus M),\quad 
V^\sharp\in \Gamma (\mathcal{H}).$$
Then
$$
L_V(\alpha)=\iota(V)\,\Omega-\mathrm{d}\varphi.
$$
This implies the following.

\begin{lem}
\label{lem:Lie der ham}
Let $V\in \Gamma(\mathcal{T}\oplus \mathcal{H})$ be as in (\ref{eqn:vector field tgtau}). Then the following conditions are equivalent:
\begin{enumerate}
\item 
$L_V\alpha=0$;
\item $V$ is the Hamiltonian vector field of $\varphi=-\alpha(V)$
with respect to $\Omega$.
\end{enumerate}

\end{lem}

Since $\upsilon_{\sqrt{\rho}}=-\sqrt{\rho}\,\mathcal{R}$
by (\ref{eqn:Tsqrtrho}), Lemma \ref{lem:Lie der ham} 
generalizes Lemma \ref{lem:primitive geod inv}.

\subsection{Heisenberg local coordinates}
\label{sctn:HLC}

As emphasized in the Introduction, Chang and Rabinowitz in \cite{cr1} and \cite{cr2} considerably simplified 
the application of the
ideas and techniques from the line bundle setting in 
\cite{z98}, \cite{bsz} and \cite{sz} to the Grauert tube context,
and their approach is partly based on Folland and Stein's construction
of Heisenberg local coordinates for a strictly pseudoconvex 
hypersurface in a complex manifold \cite{fs1}, \cite{fs2}. 

\subsubsection{Heisenberg-type order}
\label{sctn:heis type order M}
The notion of Heisenberg local coordinates on a complex manifold
adapted to a strictly pseudoconvex hypersurface rests on 
the concept of Heisenberg-type order of vanishing of a smooth function
at a given point with respect to a local holomorphic chart
(\cite{fs1}, \S 14 and \S 18).

Suppose $x\in \tilde{M}$ and let
$(U,\varphi,A)$ be a local holomorphic chart for 
$\tilde{M}$ \textit{centered at} $x$; thus $U$ is an
open neighborhood of $x$ in $\tilde{M}$, and $A\subseteq \mathbb{C}^d$
is open. 
Let us write
$\varphi=(z_0,z_1,\ldots,z_{d-1})$, where 
$z_j:U\rightarrow \mathbb{C}$.

\begin{defn} 
\label{defn:heisenberg type order}
Let $\mathfrak{J}_x(\tilde{M})$ be the ring of germs of 
(non necessarily smooth, real or complex)
functions on $\tilde{M}$ at $x$, and let 
$\mathfrak{m}_x(\tilde{M})\unlhd \mathfrak{J}_x(\tilde{M})$ be the ideal 
of those germs that vanish at $x$.
Let $\mathcal{C}^\infty(\tilde{M})_x\subseteq \mathfrak{J}_x(\tilde{M})$ 
be the subring of germs of 
smooth functions.
Suppose $f\in \mathfrak{m}_x(\tilde{M})$. Then 
\begin{enumerate}
\item $f$ is said to be $O_\varphi^1$ if, for $\tilde{M}\ni y\sim x$, 
$$
f(y)=O\left( \sum_{j=1}^{d-1}|z_j(y)|+|z_0(y)|^{1/2}\right);
$$
\item $ \mathfrak{O}_\varphi^1(\tilde{M})
:=\left\{f\in \mathfrak{m}_x(\tilde{M})\,:\,
 \text{$f$ is $O_\varphi^1$} \right\}$;
\item inductively, for $k\ge 2$ we define $
 \mathfrak{O}_\varphi^k(\tilde{M}):=\mathfrak{O}_\varphi^{k-1}(\tilde{M})\cdot 
 \mathfrak{O}_\varphi^1(\tilde{M}) $;
 \item for any integer $k\ge 2$, $f$ is said to be 
 $O_\varphi^k$ if 
 $f\in \mathfrak{O}_\varphi^k(\tilde{M})$;
 
\item finally, 
$\mathfrak{C}_\varphi^k(\tilde{M}):=\mathcal{C}^\infty(\tilde{M})_x
\cap\mathfrak{O}_\varphi^k(\tilde{M})$.
\end{enumerate}
\end{defn}

For example, if $x_j:=\Re(z_j)$ and $y_j:=\Im (z_j)$
then $\sqrt{|x_0|},\,\sqrt{|y_0|}\in \mathfrak{O}_\varphi^1(\tilde{M})$;
therefore, 
$x_0=\mathrm{sgn}(x_0)\,\sqrt{|x_0|}\cdot \sqrt{|x_0|}\in 
\mathfrak{C}_\varphi^2(M)$, and similarly for $y_0$ and
$z_0$. For every $k\ge 1$, we have
$|x_0|^{\frac{k+1}{2}}=|x_0|^{\frac{k}{2}}\cdot |x_0|^{\frac{1}{2}}$,
and it follows inductively that $|x_0|^{\frac{l}{2}}\in 
\mathfrak{O}_\varphi^l(\tilde{M})$ for every $l\ge 1$.
On the other hand, $x_j,\,y_j\in \mathfrak{C}_\varphi^1(\tilde{M})$ for
every $j\ge 1$, and one obtains inductively that
$x_j^l\in \mathfrak{C}_\varphi^l(\tilde{M})$ for every $l\ge 1$.

We shall occasionally abridge the notation
$\mathfrak{O}_\varphi^k(\tilde{M})$ to $\mathfrak{O}_\varphi^k$.
By the inductive definition,
\begin{equation}
\label{eqn:Ok prodottidi Dk}
\mathfrak{O}_\varphi^k=
\underbrace{\mathfrak{O}_\varphi^1\cdot 
\mathfrak{O}_\varphi^1\cdots \mathfrak{O}_\varphi^1}_{\text{$k$ times}},
\end{equation}
i.e. any $f\in \mathfrak{O}_\varphi^k$ is a (finite) linear combination
of products $f_1\cdots f_k$ with $f_j\in \mathfrak{O}_\varphi^1$.

\begin{cor}
\label{cor:Okprod}
$\mathfrak{O}_\varphi^k\cdot \mathfrak{O}_\varphi^l
=\mathfrak{O}_\varphi^{k+l}$
for any $k,l\ge 1$.

\end{cor}

Let us set $z':=(z_1,\ldots,z_{d-1}):U\rightarrow
\mathbb{C}^{d-1}$; thus $\varphi=(z_0,\,z')$.

\begin{lem}
\label{lem:growth chactOk}
Suppose $f\in \mathfrak{m}_x(\tilde{M})$
Then the following conditions are equivalent:
\begin{enumerate}
\item $f\in \mathfrak{O}_\varphi^k$;
\item $f(y)=O\left(|z'(y)|^k+|z_0(y)|^{k/2}\right)$ for 
$\tilde{M}\ni y\sim x$.
\end{enumerate}

\end{lem}

\begin{proof}
Let $0<d_k<D_k$ be constants, depending only on $k\ge 1$, such that
for any pair $a,b\ge 0$ one has
$$
d_k\,(a^k+b^k)\le (a+b)^k\le D_k\,(a^k+b^k).
$$

Suppose $f\in \mathfrak{O}_\varphi^k$. Then $f$ is a sum
of products of the form $f_1\cdots f_k$ with each
$f_j\in \mathfrak{O}_\varphi^1$. We may thus assume that
$f$ itself is of this form.
Hence there is a constant $C>0$
(depending on $f$) such that  
$$
|f_j(y)|\le C\,\left(|z'(y)|+|z_0(y)|^{1/2}\right)
\qquad (j=1,\ldots,k);
$$
therefore
\begin{eqnarray*}
|f(y)|&\le& C^k\,\left(|z'(y)|+|z_0(y)|^{1/2}\right)^k\\
&\le&C^k\,D_k\,\left(|z'(y)|^k+|z_0(y)|^{k/2}\right).
\end{eqnarray*}
Hence 2. holds.

Conversely, assume that 2. holds. Then $f\in \mathfrak{O}_\varphi^1$ 
by definition for $k=1$. Suppose that $k\ge 2$.
Then 
$$
f(y)=O\left(|z'(y)|^k+|z_0(y)|^{k/2}\right)=O\left((|z'(y)|+|z_0(y)|^{1/2})^k\right).
$$
In some neighbourhood $U$ of $x$ in $\tilde{M}$,
let us define $g:U\rightarrow \mathbb{C}$ by 
$$
g(y):=\frac{f(y)}{\left(|z'(y)|+|z_0(y)|^{1/2}\right)^{k-1}}\quad\text{if
$y\neq x$;}\quad g(x)=0.
$$
When $y\neq x$, 
$$
|g(y)|\le 
\frac{C\,\left(|z'(y)|+|z_0(y)|^{1/2}\right)^k}{\left(|z'(y)|+|z_0(y)|^{1/2}\right)^{k-1}}= C\,\left(|z'(y)|+|z_0(y)|^{1/2}\right).
$$
Hence $g\in \mathfrak{O}_\varphi^1$, and 
$$
f(y)=g(y)\cdot \left(|z'(y)|+|z_0(y)|^{1/2}\right)^{k-1}.
$$
Thus
$f\in \mathfrak{O}_\varphi^1\cdot \mathfrak{O}_\varphi^{k-1}=
\mathfrak{O}_\varphi^{k}$.

\end{proof}

Let us focus on $\mathfrak{C}_\varphi^{k}$.

\begin{defn}
\label{eqn:defn-weighted-mdg}
For $\mathbf{a},\,\mathbf{b}\in \mathbb{Z}_{\ge 0}^d$,
consider the monomial function of
$z:=\begin{pmatrix}
z_0&\cdots&z_{d-1}
\end{pmatrix}\in \mathbb{C}^d$ given by
\begin{equation}
\label{eqn:monomialtayl}
P_{\mathbf{a},\mathbf{b}}(z):=
z_0^{a_0}\,\overline{z}_0^{b_0}\,\prod_{j=1}^{d-1}
z_j^{a_j}\,\overline{z}_j^{b_j}.
\end{equation}
The \textit{weighted degree} of $P_{\mathbf{a},\mathbf{b}}$ is 
$$
\mathrm{wdg}(P_{\mathbf{a},\mathbf{b}}):=
2\,(a_0+b_0)+\sum_{j=1}^{d-1}(a_j+b_j).
$$
\end{defn}

\begin{prop}
\label{prop:monomialRk}
Suppose $f\in \mathcal{C}^\infty(M)_x$. Then the following conditions
are equivalent.
\begin{enumerate}
\item $f\in \mathfrak{C}^k_\varphi$;
\item every monomial contributing
to the Taylor expansion of $f\circ \varphi^{-1}$ at $\mathbf{0}$ has weighted
degree $\ge k$.
\end{enumerate}
\end{prop}

If $M=\mathbb{C}^d$, $U=A\subseteq \mathbb{C}^d$ is open and $\varphi=\mathrm{id}_A$,
we shall write $\mathfrak{C}^k$
for $\mathfrak{C}_\varphi^{k}$. 
Since in the general setting 
$f\in \mathfrak{C}_\varphi^{k}$ if and only if 
$f\circ \varphi^{-1}\in \mathfrak{C}^k$, 
we may identify $U$ with $A$ 
and work directly
on $A\subseteq \mathbb{C}^d$.

\begin{proof}
We may assume without loss
that $A=D(\mathbf{0},\mathbf{r})$ is a polydisc centered at the
origin, of multiradius $\mathbf{r}=(r_0,\ldots,r_{d-1})$,
with $r_0\in (0,1]$.
Let us denote the linear complex coordinates on $A$ by
$(w_0,w_1,\ldots,w_{d-1})=(w_0,w')$.

Let 
$h:A\rightarrow A$ be given by
$$
h(w_0,w_1,\ldots,w_{d-1}):=
\left( w_0^2,w_1,\ldots,w_d  \right)=(w_0^2,w').
$$
A monomial 
$P_{\mathbf{a},\mathbf{b}}(z)$
%
pulls back to
\begin{equation}
\label{eqn:pull back Pab}
h^*(  P_{\mathbf{a},\mathbf{b}})(w)= w_0^{2\,a_0}\,\overline{w}_0^{2\,b_0}\,\prod_{j=1}^{d-1}
w_j^{a_j}\,\overline{w}_j^{b_j};
\end{equation}
hence 
$$
\mathrm{deg}\left(h^*(  P_{\mathbf{a},\mathbf{b}})\right)=
2\,(a_0+b_0)+\sum_{j=1}^{d-1}(a_j+b_j)=\mathrm{wdg}(P_{\mathbf{a},\mathbf{b}}).
$$
On the other hand, if $f$ is a germ of smooth function at
the origin on $\mathbb{C}^d$, then
the Taylor expansion of $f\circ h$ a the origin 
is the pull-back by $h$ of the one of $f$. 
Hence 2. holds if and only
if the Taylor expansion of $f\circ h$ only contains monomials
of degree $\ge k$, i.e. if $f\circ h$ vanishes to $k$-th
order at the origin.

Assume $f\in \mathfrak{C}^{k}$. Then 
$f\circ h$ is smooth and by Lemma \ref{lem:growth chactOk}
$$
f\circ h(w)=O\left(
|w'|^k+|w_0|^k\right)
$$
for $w\sim 0$;
hence the Taylor expansion of $f\circ h$ at $\mathbf{0}$
only contains terms of degree $\ge k$, i.e. 2. holds.

Suppose, conversely, that 2. holds;
equivalently, the Taylor expansion of $f\circ h$ at the origin
only contains monomials of degree $\ge k$. Hence on a neighbourhood of
the origin
$$
|f\circ h(w)| \le C\,\left(|w'|^k+|w_0|^k\right)
$$
for some constant $C>0$. If $z=h(w)$, this means
$$
|f(z)|\le C\,\left(|z'|^k+|z_0|^{k/2}\right),
$$
so that $f\in \mathfrak{C}^k$.

\end{proof}

\subsubsection{Heisenberg type order and holomorphic extensions}
A notational clarification is in order.
Let be given a $d$-dimensional 
complex manifold $Z$, with complex structure $J$.
We shall denote by $\overline{Z}$ the conjugate complex manifold
(that is, $\overline{Z}=Z$ as differentiable manifolds,
but with complex structure $-J$).
In particular, let $J_0$ be the standard 
complex structure on $\mathbb{C}^d$;
an open subset $A\subseteq \mathbb{C}^d$ is 
a complex manifold with the induced complex structure, which
shall also be denoted $J_0$. Then
$\overline{A}$ is the \textit{same} open subset, 
endowed with the
complex structure $-J_0$. 

On the other hand, let $c:\mathbb{C}^d\rightarrow \mathbb{C}^d$
denote complex conjugation. We shall set
$A^c:=c(A)$, with the complex structure $J_0$. Then
$c$ yields by restriction an anti-holomorphic diffeomorphism
$c:A\rightarrow A^c$, or equivalently a biholomorphism 
$c:\overline{A}\rightarrow A^c$.

Consider a holomorphic local chart $(U,\varphi,A)$ of
$Z$; thus $U\subseteq Z$ is an open subset with the complex structure 
$J$ and $\varphi:U\rightarrow A$ is a biholomorphism
for $J$ and $J_0$. We obtain
two \lq holomorphic charts\rq\, for $\overline{Z}$, both defined on
$U$:
$(\overline{U},\varphi,\overline{A})$, and 
$(\overline{U},c\circ \varphi, A^c)$
(to be precise, 
both are biholomorphisms, but
only the latter is a genuine holomorphic chart). 

Suppose $f\in 
\mathcal{O}(\overline{U})$, that is, $f:U\rightarrow \mathbb{C}$
is $(-J)$-holomorphic. Then $f\circ \varphi^{-1}:
\overline{A}\rightarrow \mathbb{C}$ is holomorphic on
$\overline{A}$; hence it
locally admits a power series expansion in the conjugate variables 
$\overline{z}_j$'s. Instead,
$f\circ ( c\circ \varphi)^{-1}=(f\circ \varphi^{-1})\circ c$
is holomorphic on $A^c$; hence it 
locally admits a power series expansion in the standard variables $z_j$'s.

The diagonal $\Delta\subset\tilde{M}\times \overline{\tilde{M}}$ is a totally real submanifold, real-analytically diffeomorphic
to $\tilde{M}$. Let $(U,\varphi,A)$ be holomorphic local chart
for $M$ centered at $x$. Then $(U\times \overline{U},\varphi\times \varphi,A\times \overline{A})$ is a 
\lq holomorphic local chart\rq\,
for $\tilde{M}\times \overline{\tilde{M}}$ centered at $(x,x)$.
Let $z_j$ and $u_j$ denote, respectively, the 
standard complex linear coordinates
on the two factors of $\mathbb{C}^d\times \mathbb{C}^d$, respectively.
If $F:\tilde{M}\times
\overline{\tilde{M}}\rightarrow \mathbb{C}$
is holomorphic, then $F\circ (\varphi\times \varphi)^{-1}$
can be expanded in a power series in the $z_j$'s and the 
$\overline{u}_j$'s.

Any real-analytic function $f$ on $\tilde{M}$
may be viewed as a real-analytic
function on $\Delta$; as such, it has a holomorphic extension
$\tilde{f}$
to an open neighborhood of $\Delta$ in $\tilde{M}\times \overline{\tilde{M}}$. 
Suppose that $f\in \mathfrak{C}^k_\varphi$ for some $k\ge 1$.
By Proposition \ref{prop:monomialRk}, the only contributions to 
the power series expansion
of $f\circ \varphi^{-1}\in \mathfrak{C}^k$ come from monomials 
(\ref{eqn:monomialtayl}) such that
$2\,(a_0+b_0)+\sum_{j=1}^{d-1}(a_j+b_j)\ge k$.
On the other hand, the holomorphic extension of (\ref{eqn:monomialtayl}) 
to $\mathbb{C}^d\times \overline{\mathbb{C}^d}$
has the form 
\begin{equation}
\label{eqn:monomialtayext}
z_0^{a_0}\,\overline{u}_0^{b_0}\,\prod_{j=1}^{d-1}
z_j^{a_j}\,\overline{u}_j^{b_j}.
\end{equation}
When we match Heisenberg type ordering with holomorphic extension
of real analytic functions,
we are thus led to introduce the following two rings.
\begin{defn}
\label{defn:Okphi}
Let $k\ge 1$ be an integer.

\begin{enumerate}
\item $\mathcal{O}^k$ will denote
the ring of germs of
holomorphic functions $F$ on $\mathbb{C}^d\times \overline{\mathbb{C}^d}$
at $(\mathbf{0},\mathbf{0})$ with the following property: if a monomial (\ref{eqn:monomialtayext})
gives a non-trivial contribution to the power series expansion
of $F$, then $2\,(a_0+b_0)+\sum_{j=1}^{d-1}(a_j+b_j)\ge k$;
\item $\mathcal{O}^k_{\varphi\times \varphi}$ 
will denote the ring of germs
of holomorphic functions $F$ on $\tilde{M}\times \overline{\tilde{M}}$ at $(x,x)$ such that $F\circ (\varphi\times \varphi)^{-1}\in \mathcal{O}^k$.

\end{enumerate}

\end{defn}


\begin{cor}
\label{cor:tildeRphiext}
Let $(U,\varphi,A)$ be a holomorphic local chart for $M$ centered
at $x$. Then the following holds.
\begin{enumerate}
\item $\mathcal{O}^k_{\varphi\times \varphi}$ consists of the holomorphic
extensions (for $(J,-J)$) 
of the 
real-analytic germs in $\mathfrak{C}^k_\varphi$;
\item if $F$ is a germ of holomorphic function on
$\tilde{M}\times \overline{\tilde{M}}$ at $(x,x)$, then 
$F\in \mathcal{O}^k_{\varphi\times \varphi}$ 
if and only if
$$
\left|F\circ (\varphi\times \varphi)^{-1}(z,u)\right|
=O\left(\|z'\|^k+\|u'\|^k+|z_0|^{\frac{k}{2}}+
|u_0|^{\frac{k}{2}}\right)
$$
for $(z,w)\sim (\mathbf{0},\mathbf{0})$.
\end{enumerate}

\end{cor}

We shall occasionally simplify notation and
identify $F(z,u)$ with $F\circ (\varphi\times \varphi)^{-1}(z,u)$ 
in Corollary
\ref{cor:tildeRphiext}.

\subsubsection{Heisenberg local coordinates adapted to a hypersurface}
\label{sctn:heislochyper}
In \S 18 of \cite{fs1}, a special system of local holomorphic
coordinates is constructed on a complex manifold near a point lying on a strictly
pseudoconvex hypersurface; this construction was profitably put to use
in \cite{cr1} and \cite{cr2} to study the asymtptotics of Szeg\"{o}
and Poisson kernels on Grauert tubes. 
In such a system of coordinates,
the local geometry of the hypersurface is well approximated by
the local geometry of the Heisenberg group; 
for this reason, these systems of coordinates are naturally 
referred to as
\textit{Heisenberg local coordinates}
(see e.g. \cite{sz}, \cite{cr1}, \cite{cr2}), 
although they were originally
called \textit{normal coordinates} in \cite{fs1}.
In our setting, up to a simple rescaling this amounts to the
existence of coordinates as in following definition.

\begin{defn}
\label{defn:heisenberglocalcoordinates}
If $\tau\in (0,\epsilon)$ and 
$x\in X^\tau\subseteq \tilde{M}^\epsilon$, a system of
\textit{Heisenberg local coordinates 
on $\tilde{M}$ adapted to
$X^\tau$ at $x$} is a holomorphic local chart
$(U,\varphi,A)$ for $\tilde{M}$
centered at $x$, with the following properties:
\begin{enumerate}
\item $\left.\frac{\partial}{\partial z_0}\right|_x
\in \mathrm{span}_{\mathbb{C}}\left( \mathcal{R}(x) \right)$;
\item $\left.\frac{\partial}{\partial z_j}\right|_x
\in \mathcal{H}^{(1,0)}(x):=
\big(\mathcal{H}(x)\otimes \mathbb{C}\big)\cap 
T^{(1,0)}\tilde{M}^\tau$, for $j=1,\ldots,d-1$;
\item the defining function
$\phi^\tau:=\rho-\tau^2$ for $X^\tau$ takes the form
\begin{eqnarray}
\label{eqn:phitauheis}
\phi^\tau \circ\varphi^{-1}(z)=-2\,\Im (z_0)+\|z'\|^2+f(z),
\end{eqnarray}
where  
$f\in \mathfrak{C}^3
$ 
(clearly $f$ is real valued and real-analytic).
\end{enumerate}

\end{defn}

\begin{cav}
\label{cav:different norm heis}
The norm $\|\cdot\|$ in (\ref{eqn:phitauheis}) is for now
simply the Euclidean norm in the given coordinate system,
but it will be shown below to have an metric intrinsic meaning
(see (\ref{eqn:sq norm tildekappaV})).
\end{cav}

As mentioned in the Introduction, it will be convenient to make a slightly more specific choice
of coordinates (without altering the previous properties). 
Since $f\in\mathfrak{C}^3$ and is real-analytic,
for suitable coefficients $c\in \mathbb{R}$
and $a_l,b_j\in \mathbb{C}$ we have
\begin{equation}
\label{eqn:power series exp f}
f(z)=c\,\,|z_0|^2+\Im\left(
2\,z_0\sum_{j=0}^{d-1}a_j\,z_j+
z_0\sum_{j=1}^{d-1}b_j\,\overline{z}_j\right)
+R_3(z,\overline{z}),
\end{equation}
where $R_3$ is the third order remainder, i.e. 
a convergent power series near
$\mathbf{0}$ in $(z,\overline{z})$ involving only monomials of total 
\textit{ordinary} total degree $\ge 3$.
If we make the change of variables
$$
w_0:=z_0-z_0\sum_{j=0}^{d-1}a_j\,z_j,\quad 
w_j':=z_j\quad\text{for $j=1,\ldots,d-1$},
$$
and replace $z$ by $w$ in
(\ref{eqn:phitauheis}), we reduce to the case where 
all $a_j=0$ (possibily with a new $R_3$).
With this adjustment, we shall refer to
$(U,\varphi,A)$ as a system of \textit{normal Heisenberg local coordinates}
adapted to $X^\tau$ at $x$.

With abuse of notation, if $\gamma$ is a locally defined 
differential form on $\tilde{M}$, 
we shall occasionally also denote by
$\gamma$ its local coordinate representation ${\varphi^{-1}}^*(\gamma)$.

\begin{lem}
\label{lem:b_j_0}
Referring to
(\ref{eqn:power series exp f}), we have
$c>0$ and
$b_j=0$ for every $j=1,\ldots,d-1$.
\end{lem}

\begin{proof}
Since $\Omega=\imath\,\partial\,\overline{\partial}\rho
=\imath\,\partial\,\overline{\partial}\phi^\tau$, one concludes
from (\ref{eqn:phitauheis}) and (\ref{eqn:power series exp f})
that
\begin{eqnarray}
\label{eqn:Omegalocalcoord}
\Omega&=&
\imath\,\left(c\,\mathrm{d}z_0\wedge
\mathrm{d}\overline{z}_0+ \mathrm{d}z'\wedge\mathrm{d}\overline{z}'\right)
+\Re\left(\sum_{j=1}^{d-1} b_j\,\mathrm{d}z_0\wedge
\mathrm{d}\overline{z}_j\right)+R_1(z,\overline{z}).
\end{eqnarray}
Let us write $b_j=b_j'+\imath\, b_j''$, where $b_j',\,b_j''\in \mathbb{R}$,
and let $$\theta_k:=\Re (z_k),\,\eta_k:=\Im(z_k), 
\quad\forall\,k=0,\ldots,d-1.$$
Then 
\begin{eqnarray*}
\lefteqn{\Re\left(\sum_{j=1}^{d-1} b_j\,\mathrm{d}z_0\wedge
\mathrm{d}\overline{z}_j\right)}\\
&=&\sum_{j=1}^{d-1}\left[b'_j\,\big(\mathrm{d}\theta_0
\wedge \mathrm{d}\theta_j+\mathrm{d}\eta_0
\wedge \mathrm{d}\eta_j\big)-b_j''\,\big( \mathrm{d}\eta_0\wedge
\mathrm{d}\theta_j-\mathrm{d}\theta_0\wedge
\mathrm{d}\eta_j\big)\right]. 
\end{eqnarray*}
On the other hand, given the Hermitian orthogonality of $\partial/\partial z_0$ and
$\partial/\partial z_j$ for $j\ge 1$ at $x$, $\partial/\partial \theta_0$ and $\partial/\partial\eta_0$ 
are symplectically orthogonal to $\partial/\partial \theta_j$ 
and $\partial/\partial\eta_j$ at $x$ for $j\ge 1$. Hence 
$b_j'=b_j''=0$.
\end{proof}

Therefore,
$$
\Omega_x=
\imath\,\left(c\,\mathrm{d}_x z_0\wedge
\mathrm{d}_x\overline{z}_0+ \mathrm{d}_xz'\wedge
\mathrm{d}_x\overline{z}'\right)
=2\,\left(c\,\mathrm{d}_x \theta_0\wedge
\mathrm{d}_x\eta_0+ \sum_{j=1}^{d-1}\mathrm{d}_x\theta_j\wedge
\mathrm{d}_x\eta_j\right).
$$
Hence $\hat{\kappa}_x(\cdot,\cdot)=\Omega_x\big(\cdot,J_x(\cdot)\big)$ is given by
\begin{eqnarray}
\label{eqn:tildekappax}
\hat{\kappa}_x&=&2\,c\,\left(\mathrm{d}_x \theta_0\otimes
\mathrm{d}_x \theta_0+\mathrm{d}_x \eta_0\otimes
\mathrm{d}_x \eta_0\right)\nonumber\\
&&+2\,\sum_{j=1}^{d-1}
\left(\mathrm{d}_x \theta_j\otimes
\mathrm{d}_x \theta_j+\mathrm{d}_x \eta_j\otimes
\mathrm{d}_x \eta_j\right).
\end{eqnarray}
In particular,
$$
\left.\frac{\partial}{\partial \eta_0}\right|_x
=2\,c\,\mathrm{grad}^{\hat{\kappa}}_x(\eta_0).
$$
We conclude that 
\begin{enumerate}
\item $x$ is a critical point for the restriction of $\eta_0$ to $X^\tau$
(recall (\ref{eqn:phitauheis}));

\item therefore, $\partial/\partial \eta_0|_x$ is orthogonal to $T_xX^\tau$ with respect
to $\hat{\kappa}$, whence a non-zero multiple of 
$\Theta (x)$; 

\item hence, $\partial/\partial\theta_0$ and $\partial/\partial\theta_j$,
$\partial/\partial\eta_j$ for $j=1,\ldots,d-1$ 
are all tangent to $X^\tau$ at $x$;

\item $(\theta_0,\theta_1,\eta_1,\ldots,\theta_{d-1},\eta_{d-1})$
restrict to a system of local coordinates on $X^\tau$ centered at $x$;
\item $\left.\partial/\partial\theta_0\right|_x$ is a non-zero multiple of
$\mathcal{R}(x)$. 
\end{enumerate}

Given 
(\ref{eqn:defn alpha}) and (\ref{eqn:phitauheis}),
the local coordinate expression for $\alpha$ is then
\begin{eqnarray}
\label{eqn:aphalocalcoord}
\alpha&=&\Im \partial\rho=\Im  \partial\phi^\tau\\
&=&\mathrm{d}\theta_0+\frac{c}{2\,\imath}\,\left(\overline{z}_0\,\mathrm{d}z_0-
z_0\,\mathrm{d}\overline{z}_0\right)
+\frac{1}{2\,\imath}\,\left( \overline{z}'\cdot\mathrm{d}z'
-z'\,\mathrm{d}\overline{z}'     \right)\nonumber\\
&&+R_2(z,\overline{z}),
\nonumber
\end{eqnarray}
where $R_2(z,\overline{z})$ denotes a differential
$1$-form whose coefficients vanish
to second order at $x$.

In particular, $\alpha_x=\mathrm{d}_x\theta_0$. On the other hand
$\mathcal{R}(x)$ is a multiple of $\partial/\partial\theta_0|x$ and
$\alpha(\mathcal{R})\equiv 1$. Hence 
\begin{equation}
\label{eqn:reeb at x}
\mathcal{R}(x)=\left.\frac{\partial}{\partial \theta_0}\right|_x.
\end{equation}


\begin{lem}
\label{lem:value of c}
$c=\dfrac{1}{2\,\tau^2}$.
\end{lem}

\begin{proof}
By (\ref{eqn:normsT}) and (\ref{eqn:tildekappax})
\begin{eqnarray}
\label{eqn:value cR}
\frac{1}{\tau^2}&=&\|\mathcal{R}(x)\|_{\hat{\kappa}_x}^2=
\left\|\left.
\frac{\partial}{\partial \theta_0}\right|_x\right\|^2_{\hat{\kappa}_x}=
2\,c.
\end{eqnarray}
\end{proof}

We reach the following conclusion.

\begin{prop}
\label{prop:loc coord expr}
Let $(U,\varphi,A)$ be a normal Heisenberg local chart 
for $\tilde{M}$, centered at
$x$ and adapted to $X^\tau$.
Then the local coordinate expressions of $\phi^\tau$, $\alpha$ and
$\Omega$ are as follows:
\begin{eqnarray*}
\phi^\tau \circ\varphi^{-1}(z)&=&-2\,\Im (z_0)+
\frac{1}{2\,\tau^2}\,\,|z_0|^2+\|z'\|^2
+R_3(z,\overline{z}),\\
%
\alpha&=&\mathrm{d}\theta_0-\frac{\imath}{4\,\tau^2}\,\left(\overline{z}_0\,\mathrm{d}z_0-
z_0\,\mathrm{d}\overline{z}_0\right)
+\frac{1}{2\,\imath}\,\left( \overline{z}'\cdot\mathrm{d}z'
-z'\,\mathrm{d}\overline{z}'     \right)\\
&&+R_2(z,\overline{z}),\\
\Omega&=&\imath\,\left(\frac{1}{2\,\tau^2}\,\mathrm{d}z_0\wedge
\mathrm{d}\overline{z}_0+ \mathrm{d}z'\wedge\mathrm{d}\overline{z}'\right)
+R_1(z,\overline{z}),\\
\hat{\kappa}&=&\frac{1}{\tau^2}\,\left(\mathrm{d}\theta_0\otimes
\mathrm{d} \theta_0+\mathrm{d} \eta_0\otimes
\mathrm{d} \eta_0\right)\nonumber\\
&&+2\,\sum_{j=1}^{d-1}
\left(\mathrm{d} \theta_j\otimes
\mathrm{d} \theta_j+\mathrm{d} \eta_j\otimes
\mathrm{d} \eta_j\right)+R_1\left(z,\overline{z}\right),
\end{eqnarray*}
where $R_j$ denotes an expression of the appropriate type
(function, differential 1- or 2-form, metric tensor respectively) 
vanishing to $j$-th order at the origin.
\end{prop}

In Heisenberg local coordinates for $\tilde{M}$ 
at $x\in X^\tau$, 
$(w,\mathbf{u})\in \mathbb{C}\times \mathbb{C}^{d-1}$
corresponds to the real tangent vector
\begin{equation}
\label{eqn:Vtgvecthei}
V:=w\,\left.\frac{\partial}{\partial z_0}\right|_x+\overline{w}\,
\left.\frac{\partial}{\partial \overline{z}_0}\right|_x
+\mathbf{u}\cdot \left.\frac{\partial}{\partial z'}\right|_x
+\overline{\mathbf{u}}\cdot \left.\frac{\partial}{\partial\overline{ z}'}
\right|_x\in T_x\tilde{M}.
\end{equation}

\begin{cor}
\label{cor:heis coord norms}
With $V$ as in (\ref{eqn:Vtgvecthei}),
the square norm of $V$ with respect to $\hat{\kappa}$ 
is
$$
\|V\|_{\hat{\kappa}_x}^2
=\frac{1}{\tau^2}\,|w|^2+2\,\|\mathbf{u}\|^2.
$$
\end{cor}

As in the Introduction, let us set $\omega:=\frac{1}{2}\,\Omega$; thus
the Riemannian
metric on the K\"{a}hler manifold 
$(\tilde{M}^\epsilon,\omega,J)$ is 
$\tilde{\kappa}:=\frac{1}{2}\,\hat{\kappa}$. With $V$ as in (\ref{eqn:Vtgvecthei}) then
\begin{equation}
\label{eqn:sq norm tildekappaV}
\|V\|_{\tilde{\kappa}_x}^2=
\omega_x\big(V,J_x(V)\big)=
\frac{1}{2\,\tau^2}\,|w_0|^2+\|\mathbf{u}\|^2.
\end{equation}

\subsubsection{Heisenberg local coordinates on $X^\tau$.}
The local expression for $\phi^\tau$ in Proposition
\ref{prop:loc coord expr} yields an estimate
on $\Im(z_0)$ on $X^\tau$.

\begin{cor}
\label{cor:imaginary part Xtau}
Assume $x\in X^\tau$, and let $(U,\varphi,A)$ be a normal Heisenberg local chart on
$\tilde{M}$ adapted to $X^\tau$ at $x$.
If $y\in U\cap X^\tau$ and $\varphi(y)=(z_0,z')$
with $z_0=\Re (z_0)+\imath \,\Im(z_0)$, then
$$
\Im(z_0)=\frac{1}{4\,\tau^2}\,\,\Re (z_0)^2+\frac{1}{2}\,\|z'\|^2
+R_3\left(\Re (z_0),z',\overline{z}'\right).
$$
Under the same assumptions, 
$|z_0|^2=\Re (z_0)^2+R_4(\Re(z_0),z')$.
\end{cor}

\begin{cor}
\label{cor:eta quadratic}
There exists a constant $C_\tau>0$, such that
if $y\in U\cap X^\tau$
and $\varphi(y)=(z_0,z')$, then

$$
|\Im(z_0)|\le C_\tau\,\left(\theta_0^2+\|z'\|^2\right).
$$
\end{cor}

In the following, $\theta:=\left.\theta_0\right|_{X^\tau}$
(with $\theta_0=\Re(z_0)$).
Let us set $U^\tau:=U\cap X^\tau$ and define
$\varphi^\tau
:U^\tau\rightarrow \mathbb{R}\times \mathbb{C}^{d-1}$
by
\begin{equation}
\label{eqn:betatau}
\varphi^\tau (x'):=\begin{pmatrix}
\theta(x')&z'(x')
\end{pmatrix}.
\end{equation}
Set $A^\tau:=\varphi^\tau (U^\tau)$. 
Perhaps after restricting $U$,
$A^\tau$ is an open subset of $\mathbb{R}\times \mathbb{C}^{d-1}$
and $(U^\tau,\varphi^\tau,A^\tau)$ 
is local coordinate chart for $X^\tau$ centered at $x$.

\begin{defn}
\label{defn:heislc Xtau}
We shall call $(U^\tau,\varphi^\tau,A^\tau)$ 
\textit{the Heisenberg
local chart for $X^\tau$ at $x$ induced by $(U,\varphi,A)$},
and say that $(U^\tau,\varphi^\tau,A^\tau)$ is a
\textit{normal} Heisenberg local chart for $X^\tau$ if so is 
$(U,\varphi,A)$ for
$\tilde{M}$.
We shall often use additive notation for $\varphi^\tau$,
in the following ways.
First, if $x'\in U^\tau$ and $\varphi^\tau (x')$ is as in (\ref{eqn:betatau}),
we shall write
$x'=x+\big(\theta(x'),\,z'(x')\big)$.
When viewing $z'\in \mathbb{C}^{d-1}$ as an element
of $\mathbb{R}^{2d-2}$, we shall use bold notation and write
$x'=x+\big(\theta(x'),\,\mathbf{v}(x')\big)$.

Furthermore, let us identify $T_xX^\tau$ with $\mathbb{R}\times \mathbb{C}^{d-1}$,
by letting $(a,u)\in \mathbb{R}\times \mathbb{C}^{d-1}$
correspond to the tangent vector
\begin{equation}
\label{eqn:tg vector Xtau gen}
W:=a\,\left.\frac{\partial}{\partial\theta}\right|_x
+u\cdot \left.\frac{\partial}{\partial z'}\right|_x
+\overline{u}\cdot 
\left.\frac{\partial}{\partial \overline{z}'}\right|_x.
\end{equation}
We shall then also write 
$x+W:={\varphi^\tau}^{-1}(a,u)=x+(a,u)$.
\end{defn}

\begin{rem}
\label{rem:coord and restr}
By (\ref{eqn:reeb at x}), if the tangent vectors on the right hand side of (\ref{eqn:tg vector Xtau gen})
are meant in terms of the coordinates on $\tilde{M}$,
they are actually all tangent to $X^\tau$
at $x$.
Hence (\ref{eqn:tg vector Xtau gen}) may as well be interpreted in terms of the (restricted) local coordinates on $X^\tau$.
Thus the additive short-hand $x+W$ has 
different meanings according
to whether we think of $W$ as tangent to $X^\tau$ and refer
to $\varphi^\tau$, or to $\tilde{M}$ and refer to $\varphi$. The context
should clarify the potential ambiguity.

\end{rem}

We can extend the notion of Heisenberg-type
order of vanishing to functions on $X^\tau$
with respect to 
$(U^\tau,\varphi^\tau,A^\tau)$, by the following variant of Definition 
\ref{defn:heisenberg type order} (see \S 18 of \cite{fs1}).

\begin{defn}
\label{defn:O^kperXtau}
Let $(U^\tau,\varphi^\tau,A^\tau)$, $\varphi^\tau=(\theta,z')$, 
be a system of
Heisenberg local coordinates 
on $X^\tau$ cetnered at $x$.
Let $\mathfrak{J}_x(X^\tau)$ be the ring of germs of 
(non necessarily smooth, real or complex)
functions on $X^\tau$ at $x$; let 
$\mathfrak{m}_x(X^\tau)\unlhd \mathfrak{J}_x(X^\tau)$ be the ideal 
of those germs that vanish at $x$.
Let $\mathcal{C}^\infty(X^\tau)_x\subseteq \mathfrak{J}_x(X^\tau)$ 
be the subring of germs of 
smooth functions.
Suppose $f\in \mathfrak{m}_x(X^\tau)$. Then 
\begin{enumerate}
\item $f$ is said to be $O_{\varphi^\tau}^1$ if, 
for $X^\tau\ni y\sim x$, 
$$
f(y)=O\left( \sum_{j=1}^{d-1}|z_j(y)|+|\theta(y)|^{1/2}\right);
$$
\item $ \mathfrak{O}_{\varphi^\tau}^1(X^\tau)
:=\left\{f\in \mathfrak{m}_x(X^\tau)\,:\,
 \text{$f$ is $O_{\varphi^\tau}^1$} \right\}$;
\item inductively, for $k\ge 2$ we define $
 \mathfrak{O}_{\varphi^\tau}^k(X^\tau):=
 \mathfrak{O}_{\varphi^\tau}^{k-1}(X^\tau)\cdot 
 \mathfrak{O}_{\varphi^\tau}^1(X^\tau) $;
 \item for any integer $k\ge 2$, $f$ is said to be 
 $O_{\varphi^\tau}^k$ if 
 $f\in \mathfrak{O}_{\varphi^\tau}^k(X^\tau)$;
 
\item finally, 
$\mathfrak{C}_{\varphi^\tau}^k(X^\tau):=\mathcal{C}^\infty(X^\tau)_x
\cap\mathfrak{O}_{\varphi^\tau}^k(X^\tau)$.
\end{enumerate}

\end{defn}

Let $(U^\tau,\varphi^\tau,A^\tau)$ be induced by 
the system of Heisenberg local coordinates 
$(U,\varphi,A)$ adapted to $X^\tau$ at $x$.
The definition of $O_{\varphi^\tau}^k$ 
entails the following. 

\begin{lem}
\label{lem:res heistype order}
Let $\jmath^\tau:X^\tau\hookrightarrow \tilde{M}$ be the inclusion.
Then
$$\mathfrak{O}_{\varphi^\tau}^k(X^\tau)=
{\jmath^\tau}^*\left(\mathfrak{O}_{\varphi}^k(\tilde{M})\right).
$$

\end{lem}

\begin{proof}
The statement follows readily from the definition in case $k=1$. 
For general $k$, $\mathfrak{O}_{\varphi^\tau}^k(X^\tau)=
\mathfrak{O}_{\varphi^\tau}^1(X^\tau)\cdots 
\mathfrak{O}_{\varphi^\tau}^1(X^\tau)$ ($k$ times). The claim
follows from this and (\ref{eqn:Ok prodottidi Dk}) since
${\jmath^\tau}^*$ is multiplicative morphism.
\end{proof}

Let us express $\mathrm{vol}^R_{X^\tau}(x)$ 
in terms of $\varphi^\tau$ (recall (\ref{eqn:cvol form tau}) and
Corollary \ref{cor:rel vol for RC}). 
By Corollary \ref{cor:imaginary part Xtau}, 
${\jmath^\tau}^*(\mathrm{d}_x z_0)=
\mathrm{d}_x\theta_0$. In view 
of Proposition \ref{prop:loc coord expr}
\begin{eqnarray}
\label{eqn:loc vol form0}
\mathrm{vol}^R_{X^\tau}(x)&=&\frac{1}{\tau}\,
\mathrm{vol}^C_{X^\tau}(x)
\\
&=&
\frac{1}{\tau}\,\left.{\jmath^\tau}^*\left(
\alpha\wedge
\frac{1}{(d-1)!}\,\Omega^{\wedge (d-1)}
\right)\right|_x
\nonumber\\
&=&
\frac{\imath^{d-1}}{\tau\,(d-1)!}\,
\mathrm{d}_x\theta\wedge{\jmath^\tau}^*
\left(\frac{1}{2\,\tau^2}\,\mathrm{d}_xz_0\wedge
\mathrm{d}_x\overline{z}_0+ \mathrm{d}_x z'\wedge
\mathrm{d}_x\overline{z}'\right)^{\wedge (d-1)}.\nonumber
\end{eqnarray}
By Corollary \ref{cor:imaginary part Xtau}, 
${\jmath^\tau}^*(\mathrm{d}_x z_0)=
\mathrm{d}_x\theta_0$. Hence
\begin{eqnarray}
\label{eqn:loc vol form}
\mathrm{vol}^R_{X^\tau}(x)&=&
\frac{\imath^{d-1}}{\tau\,(d-1)!}\,
\mathrm{d}_x\theta\wedge\left(\mathrm{d}_x z'\wedge
\mathrm{d}_x\overline{z}'\right)^{\wedge (d-1)}\nonumber\\
&=&\frac{2^{d-1}}{\tau}\,\mathrm{d}_x\theta
\wedge \frac{1}{(d-1)!}\,\left(\frac{\imath}{2}\,
\mathrm{d}z'\wedge\mathrm{d}\overline{z}'\right)^{\wedge (d-1)}.
\end{eqnarray}
The latter factor is the standard volume form on
$\mathbb{C}^{d-1}\cong \mathbb{R}^{2d-2}$ 
in the linear coordinates $z'$.

\subsubsection{Comparison of Heisenberg local coordinates}

Suppose $x\in X^\tau$ and let
$\varphi=(z_0,z')$ and $\Phi=(w_0,w')$ are 
normal Heisenberg local charts adapted to $X^\tau$ at $x$.
By (\ref{eqn:reeb at x}) and 
(\ref{eqn:sq norm tildekappaV}) 
\begin{equation}
\label{eqn:coord changeH}
w_0=z_0+f(z_0,z'),\quad 
w'=A\,z'+\mathbf{f}(z_0,z'),
\end{equation}
where $A\in U(d-1)$ and $f,\,\mathbf{f}$ are holomorphic 
and vanish to second order at $x$. 
Consider $y\sim x$ and suppose 
$(z_0,z')=\varphi (y)$, $(w_0,w')=\Phi(y)$.
Let as usual $R_j$ denote a generic smooth function vanishing to
$j$-th order at $x$; by Proposition
\ref{prop:loc coord expr} and (\ref{eqn:coord changeH}),
\begin{eqnarray*}
\phi^\tau(y)&=&-2\,\Im (w_0)+
\frac{1}{2\,\tau^2}\,|w_0|^2+\|w'\|^2+R_3(w_0,w')\\
&=&-2\,\Im (z_0)-2\,\Im\big(f(z_0,z')  \big)
+\frac{1}{2\,\tau^2}\,|z_0|^2+\|z'\|^2+R_3(z_0,z').\nonumber 
\end{eqnarray*}
Given that $\varphi$ is also a normal Heisenberg chart,
in view of the same Proposition we also have
\begin{eqnarray*}
\phi^\tau(y)&=&-2\,\Im (z_0)
+\frac{1}{2\,\tau^2}\,|z_0|^2+\|z'\|^2+R_3(z_0,z').\nonumber 
\end{eqnarray*}
Thus
$\Im\big(f(z_0,z')  \big)$ 
vanishes to third order at $x$ (that is, at the origin). Since $f$ is holomorphic,
$f$ itself vanishes to third order at $x$. 
We conclude the following.

\begin{lem}
\label{lem:comp heis nc}
Let $\varphi=(z_0,z'),\,\Phi=(w_0,w'):U\rightarrow \mathbb{C}^d$
be two normal Heisenberg local charts on $\tilde{M}$ adapted to
$X^\tau$ at $x$.
Then $w_0-z_0$ vanishes to third order at $x$.
\end{lem}

\subsubsection{The geodesic flow in Heisenberg coordinates}

Since the vector field $\mathcal{R}$ of Definition \ref{defn:reeb Mkappa} 
is tangent to the compact hypersurfaces
$X^\tau$, 
it is complete on $\tilde{M}\setminus M$.
Given $x\in X^\tau$, let us choose a system of Heisenberg
normal coordinates $(U,\varphi,A)$ 
adapted to $X^\tau$ at $x$.
Let $\Lambda_x:\mathbb{R}\rightarrow X^\tau$ 
be the integral curve of
$\mathcal{R}$ passing through $x$ at $t=0$. 
For $t$ sufficiently small,
$$
\Lambda_x(t)=x+\big(z_0(t),z'(t)\big),
\quad \text{where}\quad z_0(0)=0\in \mathbb{C},\,z'(0)=\mathbf{0}
\in \mathbb{C}^{d-1}.
$$
Let us write 
$z_0(t)=\theta(t)+\imath\,\eta(t)$, where 
$\theta(t)=\Re \big(z_0(t)\big)$ and 
$\eta (t)=\Im \big(z_0(t)\big)$;
in view of (\ref{eqn:reeb at x}), we have
$$
\theta(0)=0,\,\dot{\theta}(0)=1,\quad\eta(0)=
\dot{\eta}(0)=0,\quad z'(0)=\dot{z}'(0)=\mathbf{0}.
$$
Hence $\theta(t)-t$, $\eta(t)$, and
$z'(t)$ vanish to second order at the origin.
Thus 
\begin{equation}
\label{eqn:Lambda_x0}
z_0(t)=t+f(t),
\quad z'(t)=\mathbf{F}(t),
\end{equation}
where $f$ and $\mathbf{F}$ are smooth and vanish to second order at $0\in \mathbb{R}$.

Since $\Lambda_x$ is an integral curve of $\mathcal{R}$, 
$$
\langle\alpha_{\gamma_x (t)},\dot{\gamma}_x(t)\rangle
=\langle\alpha,\mathcal{R}\rangle\circ \Lambda_x(t)
\equiv 1.
$$
Expressing this condition by means of Proposition 
\ref{prop:loc coord expr} yields
\begin{eqnarray}
\label{eqn:integral curve R}
1&=&\dot{\theta}(t)\\
&&-\frac{\imath}{4\,\tau^2}\,\left[\big( t+\overline{f}(t)\big)\cdot \big( 1+\dot{f}(t)  \big)-
\big( t+f(t)   \big)\cdot 
\left( 1+\dot{\overline{f}}(t)  \right)\right]\nonumber\\
&&+\frac{1}{2\,\imath}\,\left( \overline{\mathbf{F}}(t)\cdot 
\dot{\mathbf{F}}(t)
-\mathbf{F}(t)\,\dot{\overline{\mathbf{F}}}(t)     
\right)+R_2(t)=\dot{\theta}(t)+R_2(t). \nonumber
\end{eqnarray}
Thus $\theta(t)=t+R_3(t)$. We conclude:

\begin{lem}
\label{lem:integral curve Rlocal}
If $\Lambda_x:\mathbb{R}\rightarrow X^\tau$ is the integral
curve of $\mathcal{R}$ through $x$, then in normal Heisenberg 
local coordinates for $X^\tau$ at $x$ 
we have
$$
\varphi^\tau\circ\Lambda_x(t)=\big(t+R_3(t),\mathbf{R}_2(t)   \big).
$$

\end{lem}

In additive notation as in Definition
\ref{defn:heislc Xtau}, 
$\Lambda_x(t)=x+\big(t+R_3(t),\mathbf{R}_2(t)   \big)$.

Since any smooth function vanishing to first order
at $x$ is in $\mathfrak{C}_{\varphi^\tau}^1(X^\tau)$,
we reach the following conclusion (a slight refinement
of Lemma 3.6 of \cite{cr1}).

\begin{cor}
\label{cor:integral curve Rlocal}
Suppose 
$y=x+(\theta,\mathbf{u})\in U^\tau$ and let
$\Lambda_y:\mathbb{R}\rightarrow X^\tau$ be the integral 
curve of $\mathcal{R}$ with initial condition $y$. Then
for $t$ small we have
$$
\Lambda_y(t)=x+\big(\theta +t+R_3(t)+
t\cdot f(t,\theta,\mathbf{u}),\mathbf{u}+ \mathbf{R}_2(t)+
t\cdot \mathbf{f}(t,\theta,\mathbf{u})\big),
$$
where $f(t,\cdot,\cdot)$ and (every component of) 
$\mathbf{f}(t,\cdot,\cdot)$ are 
in 
$\mathfrak{C}_{\varphi^\tau}^1(X^\tau)$ (that is, they
are $O_{\varphi^\tau}^1$).
\end{cor}

The previous statement may be converted into one concerning
the homogeneous geodesic flow. The latter is intertwined by 
$E^\epsilon$ with the flow of $\upsilon_{\sqrt{\rho}}$;
on the other hand by (\ref{eqn:Tsqrtrho}) we have
$\upsilon_{\sqrt{\rho}}^\tau=-\tau\,\mathcal{R}^\tau$
on $X^\tau$. Thus $\gamma(\cdot)$ is an integral curve
of $\mathcal{R}^\tau$ if and only if
$\gamma (-\tau\cdot)$ 
is an integral curve of $\upsilon_{\sqrt{\rho}}^\tau$.
Let us denote by $\Gamma^\tau_t:X^\tau\rightarrow X^\tau$
the restricted geodesic flow at time $t$.

\begin{cor}
\label{cor:integral curve geodlocal}
Suppose 
$y=x+(\vartheta,\mathbf{u})\in U^\tau$. Then
for $t$ small we have
$$
\Gamma^\tau_t(y)=x+\big(\vartheta -\tau\,t+R_3(\tau\,t)+
\tau\,t\cdot f(\tau \,t,\vartheta,\mathbf{u}),\mathbf{u}
+ \mathbf{R}_2(\tau\,t)+
\tau\,t\cdot \mathbf{f}(\tau\,t,\vartheta,\mathbf{u})\big),
$$
where $f(t,\cdot,\cdot)$,
$\mathbf{f}(t,\cdot,\cdot)$ are as in the statement of
Corollary \ref{cor:integral curve Rlocal}.
\end{cor}

\begin{cor}
\label{cor:derivata upsilon theta}
$\left.\upsilon_{\sqrt{\rho}}^\tau(\theta)\right|_y=
-\tau+\tau\,f(0,\vartheta,\mathbf{u})$.
\end{cor}

\subsubsection{Horizontal curves in Heisenberg coordinates}

Let
$I\subseteq \mathbb{R}$ 
be an interval; a smooth curve $\gamma:I\rightarrow X^\tau$ will be called \textit{horizontal}
if $\langle\alpha_{\gamma(t)},\dot{\gamma}(t)\rangle =0$ $\forall\,
t\in I$.
Suppose $I=(-\epsilon',\epsilon')$
for some $\epsilon'>0$, and let
$\gamma$
be smooth, horizontal, and such that
$\gamma(0)=x$. 

For $t\sim 0$, in normal Heisenberg coordinates 
$\gamma(t)=x+\big(\theta(t),z'(t)\big)$; recall that $\theta=\left.\theta_0\right|_{U^\tau}$. By assumption
$\dot{\theta}(0)=0$, so that $\theta(t)=R_2(t)$.
Furthermore,
$z'(t)=R_1(t)$, hence 
$z'(t)=t\,\mathbf{u}+\mathbf{R}_2(t)$ 
where $\mathbf{u}\in \mathbb{C}^{d-1}$ and $\mathbf{R}_2$
is smooth and vanishes to second order at the origin.
By Corollary \ref{cor:eta quadratic}, $\eta\big(\gamma(t)\big)
=R_2(t)$. Thus $z_0\big(\gamma(t)\big)=R_2(t)$.

Hence by Proposition \ref{prop:loc coord expr} 
\begin{eqnarray}
\label{eqn:horizcurve heis}
0=\alpha_{\gamma(t)}\left( \dot{\gamma}(t) \right)=
\dot{\theta}_0(t)+R_2(t)\quad
\Rightarrow\quad \theta(t)=R_3(t).
\end{eqnarray}
We conclude the following.

\begin{lem}
\label{lem:heis loc coord horiz}
If $\gamma:(-\epsilon,\epsilon)\rightarrow X^\tau$ is horizontal
and $\gamma(0)=x$, then in Heisenberg local coordinates for
$X^\tau$ at $x$ we have
$$
\gamma (t)=x+\big(R_3(t),\mathbf{R}_1(t)\big).
$$
\end{lem}

\subsection{Symbols and Toeplitz operators}
\label{sctn:symbol toeplitz}

Suppose $\tau\in (0,\epsilon)$ and $W\in \mathfrak{X}(X^\tau)$.
Let us decompose $W$ in terms of the direct sum
(\ref{eqn:directsum tXtau}):
\begin{equation}
\label{eqn:vector field tgtau1}
W=W^\sharp-\lambda\,\mathcal{R}^\tau.
\end{equation}
Here
$W^\sharp$ is a smooth section of $\mathcal{H}^\tau$, and
$\lambda=-\alpha^\tau(W)\in \mathcal{C}^\infty(X^\tau)$.

Assume that $L_W(\alpha^\tau)=0$. 
Then the flow of $W$
preserves $\mathrm{vol}^C_{X^\tau}$ by (\ref{eqn:cvol form tau});
hence it also preserves $\mathrm{vol}^R_{X^\tau}$ by Corollary \ref{cor:rel vol for RC}. 
Therefore the flow of $W$ induces in a standard manner 
a one-parameter group of unitary automorphisms of $L^2(X^\tau)$.
It follows that, as a differential operator on $X^\tau$,
$W$ is skew-symmetric; thus the Toeplitz
operator 
$$\mathfrak{W}:=\imath\,\Pi^\tau\circ W\circ \Pi^\tau$$ 
is formally self-adjoint.

By \cite{bg}, Toeplitz operators on $X^\tau$
have
well-defined principal symbols, which are smooth functions on
the closed
symplectic cone
\begin{equation}
\label{eqn:symplcone}
\Sigma^\tau:=\left\{ 
\left(x,\,r\,\alpha^\tau_x\right)\,:\,
x\in X^\tau,\,r>0
\right\}\subseteq T^\vee X^\tau\setminus (0).
\end{equation}
Let us compute 
the principal symbol $\sigma(\mathfrak{W})$ of $\mathfrak{W}$
at $\left(x,\,r\,\alpha^\tau_x\right)$. 
We consider a system of normal Heisenberg local coordinates 
on $\tilde{M}$ adapted to $X^\tau$ at $x$ (Definition
\ref{defn:heisenberglocalcoordinates}),
and the corresponding 
Heisenberg
local chart for $X^\tau$ at $x$ (Definition
\ref{defn:heislc Xtau}); we denote the latter by $\varphi^\tau=(\theta,z')$.
By Proposition \ref{prop:loc coord expr},  $\alpha^\tau_x=\mathrm{d}_x\theta$. 
On the other hand, by (\ref{eqn:reeb at x}) and 
(\ref{eqn:vector field tgtau1}), 
$$
W(x)=W^\sharp(x)-\lambda(x)\,
\left.\frac{\partial}{\partial \theta}\right|_x.
$$
Thus, 
\begin{equation}
\label{eqn:symbol Wtoeplitz}
\sigma(\mathfrak{W})(x,r\,\alpha^\tau_x)=
\imath\,e^{-\imath\,r\,\theta}\,W(x)
\left( e^{\imath\,r\,\theta}  \right)=
\imath\,\big(-\lambda(x)\,\imath\,r\big)=r\,\lambda(x).
\end{equation}
If $\lambda>0$, therefore, $\mathfrak{W}$ is a positive self-adjoint Toeplitz operator, 
so that its spectrum is discrete, bounded from below 
and accumulates at $+\infty$ (see \cite{bg}). 
This applies in particular 
if $W=\upsilon_{\sqrt{\rho}}$ by (\ref{eqn:Tsqrtrho}), with 
$\lambda=\sqrt{\rho}$. Then $\mathfrak{W}=\mathfrak{D}_{\sqrt{\rho}}^\tau$
(see (\ref{eqn:toeplitz operator Drho})).

\subsection{The Szeg\"{o} kernel and its phase}

Recall that $L^2(X^\tau)$ denotes the Hilbert space of square summable
functions on $X^\tau$ with respect to $\mathrm{vol}^R_{X^\tau}$
in \S \ref{sctn:vol form cr}, $H(X^\tau)\subseteq L^2(X^\tau)$
is the corresponding Hardy space,
and 
$\Pi^\tau:L^2(X^\tau)\rightarrow H(X^\tau)$,
the Szeg\"{o} projector, 
is the orthogonal projector.
By \cite{bs}, $\Pi^\tau$ is a Fourier integral operator with complex phase;
its wave front $\mathrm{WF}(\Pi^\tau)=
{\Sigma^\tau}^\sharp$ is the anti-diagonal 
of $\Sigma^\tau$ in (\ref{eqn:symplectic cone}):
\begin{eqnarray}
\label{eqn:sigmatauantidiag}
{\Sigma^\tau}^\sharp &=&
\left\{\left(x,r\alpha^\tau_x,x,-r\,\alpha^\tau_x\right)\,:\,
x\in X^\tau,\,r>0\right\}\\
&\subseteq &\left(T^\vee X^\tau\setminus (0)\right)
\times \left(T^\vee X^\tau\setminus (0)\right).
\nonumber
%
\end{eqnarray} 
More precisely, up to a smoothing term the distributional kernel of
$\Pi^\tau$ (a.k.a. the Szeg\"{o} kenel of $X^\tau$) 
is microlocally of the form
\begin{equation}
\label{eqn:fioszego}
\Pi^\tau(x',x'')\simeq\int_0^{+\infty}e^{\imath\,u\,\psi^\tau(x',x'')}
\,s^\tau(x',x'',u)\,\mathrm{d}u,
\end{equation}
where the amplitude $s^\tau$ and the phase 
$\psi^\tau$ are as follows (see \cite{bs}).
\begin{enumerate}
\item $s^\tau$ is a semiclassical symbol admitting an asymptotic
expansion
\begin{equation}
\label{eqn:semiclasstau}
s^\tau(x',x'',u)\sim \sum_{j\ge 0}\,u^{d-1-j}\,s^\tau_j(x',x'').
\end{equation}
\item $\psi^\tau$ satisfies $\Im(\psi^\tau)\ge 0$ and is essentially
determined along the diagonal of $X^\tau$ by the Taylor expansion
of the defining function $\phi^\tau$; in the present 
real-analytic setting we may assume that
\begin{equation}
\label{eqn:psitauextol}
\psi^\tau:=\left.\frac{1}{\imath}\,
\tilde{\phi}^\tau\right|_{X^\tau\times X^\tau},
\end{equation}
where $\tilde{\phi}^\tau$ denotes the holomorphic extension
of $\phi^\tau$ to $\tilde{M}\times \overline{\tilde{M}}$
(see the discussion preceding Definition \ref{defn:Okphi}).

\end{enumerate}

Let us express $\psi^\tau$ in the neighbourhood of $(x,x)$
in $X^\tau\times X^\tau$ using normal Heisenberg local coordinates
$\varphi^\tau=(\theta,z')$
on $X^\tau$ centered at $x$, defined on an open subset
$U^\tau\subseteq X^\tau$. Let $\psi_2^{\omega_x}$
be as in (\ref{eqn:psi2x}).

\begin{prop}
\label{prop:approx expr psitau}
Suppose that $$x',\,x''\in U^\tau,\quad 
(\theta,z')=\varphi^\tau (x'),\quad
(\eta,u')=\varphi^\tau(x'').$$
Then
$$
\imath\,\psi^\tau(x',x'')=
\imath\,(\theta-\eta)-\frac{1}{4\,\tau^2}\,
\left(\theta-\eta\right)^2
+\psi_2^{\omega_x}\left(z',u'\right)
+R_3(\theta,\eta,z',\overline{z}',u',\overline{u}'),
$$
where the latter term denotes a power series in the indicated variables,
involving only terms of total degree $\ge 3$.

\end{prop}
 
\begin{proof}
Let $\varphi$ be the normal Heisenberg local chart on $\tilde{M}$
centered at $x$ inducing $\varphi^\tau$.
Let $(z_0,z'):=\varphi(x')$ and $(u_0,u'):=\varphi(x'')$, 
so that $\theta=\Re(z_0)$, $\eta=\Re(w_0)$.
By Proposition \ref{prop:loc coord expr} and (\ref{eqn:psitauextol}),
\begin{eqnarray}
\label{eqn:local express psitau heis}
\lefteqn{\imath\,\psi^\tau(x',x'')}  \\
&=& \imath\,(z_0-\overline{u}_0)
+\frac{1}{2\,\tau^2}\,\,z_0\,\overline{u}_0+z'\cdot \overline{u}'
+R_3(z,\overline{u}) \nonumber\\    
&=&\imath\,(\theta-\eta)-\big(\Im(z_0)+\Im(u_0)\big)
+\frac{1}{2\,\tau^2}\,\,z_0\,\overline{u}_0
+z'\cdot \overline{u}'+R_3(z,\overline{u}).   \nonumber
\end{eqnarray}
Let us abridge the third order term
to $R_3$. 
Applying Corollary \ref{cor:imaginary part Xtau}
and (\ref{eqn:sq norm tildekappaV}),
we obtain
\begin{eqnarray}
\label{eqn:local express psitau heisexpl}
\lefteqn{\imath\,\psi^\tau(x',x'')}  \\
&=&\imath\,(\theta-\eta)-\frac{1}{4\,\tau^2}\,\left(\theta^2+\eta^2
-2\,\theta\,\eta\right)
-\frac{1}{2}\,\left(\|z'\|^2+\|u'\|^2-2\,z'\cdot \overline{u}'\right)
+R_3  \nonumber\\
&=&\imath\,(\theta-\eta)-\frac{1}{4\,\tau^2}\,\left(\theta-\eta\right)^2
+\psi_2^{\omega_x}\left(z',u'\right)
+R_3 . \nonumber
\end{eqnarray}

\end{proof}

The following property follows from the general 
construction of the phase of the Szeg\"{o}
kernel in \cite{bs}; here it can be read immediately
from Propositions \ref{prop:loc coord expr} and \ref{prop:approx expr psitau}.

\begin{cor} 
\label{cor:diffpsi diagonal}
For any $x\in X^\tau$,
$\mathrm{d}_{(x,x)}\psi^\tau=(\alpha_x,-\alpha_x)$.
\end{cor}

\begin{proof}
Let notation be as in
Proposition \ref{prop:approx expr psitau}. Then
$\mathrm{d}_{(x,x)}\psi^\tau=
(\mathrm{d}_x\theta,-\mathrm{d}_x\eta)$. The statement
then follows from Proposition \ref{prop:loc coord expr}.
\end{proof}

Again, the following statement follows from the general
theory of \cite{bs},
but it can also be verified by direct inspection of (\ref{eqn:local express psitau heisexpl}). Let $\mathrm{dist}_{X^\tau}:X^\tau\times X^\tau 
\rightarrow \mathbb{R}$ be the Riemannian distance function.

\begin{cor}
\label{cor:bound psi distance}
There are a neighborhood $X'\subseteq X^\tau\times X^\tau$ of
the diagonal and a constant $C^\tau>0$ such that
$$
\Im \psi^\tau(x',x'')\ge C^\tau\,\mathrm{dist}_{X^\tau}(x',x'')^2
$$
for all $(x',x'')\in X'$.
\end{cor}

\subsubsection{The leading order term of the amplitude}

We aim to determine the evaluation $s^\tau_0(x,x)$ 
of the leading order term in (\ref{eqn:semiclasstau});
we shall follow
the argument in \S 4 of 
\cite{bs}, and apply Proposition \ref{prop:approx expr psitau}.

\begin{thm}
\label{thm:leading order term Pi}
In a system of normal Heisenberg local coordinates on $X^\tau$
centered at $x$, 
$s^\tau_0(x,x)=\tau/(2\,\pi)^d$.
\end{thm}

\begin{rem}
Recall that $\Pi^\tau$ is the
Szeg\"{o} kernel for the Hermitian structure on $L^2(X^\tau)$
associated to $\mathrm{vol}_{X^\tau}^R$.
Integration with respect to
$\mathrm{vol}_{X^\tau}^R$ in a variable $y$ will be denoted
by the short-hand
$\mathrm{d}V_{X^\tau}(y)$. A different choice of volume form would
clearly lead to a different result.
\end{rem}

\begin{proof}
By the idempotency of
$\Pi^\tau$, for all $(x,x)\in X^\tau\times X^\tau$ we have
\begin{equation}
\label{eqn:idempotent}
\Pi^\tau(x',x'')=\left(\Pi^\tau\circ\Pi^\tau\right)(x',x'')
= \int_{X^\tau}\,\Pi^\tau(x',y)\,\Pi^\tau(y,x'')\,
\mathrm{d}V_{X^\tau}(y).
\end{equation}
The singular support of $\Pi^\tau$ is
$\mathrm{S.S.}(\Pi^\tau)=\mathrm{diag}(X^\tau)$
(the diagonal in $X^\tau\times X^\tau$). Hence we need only consider
the situation for $(x',x'')$ close to diagonal. Suppose then
$x\in X^\tau$ and that $x',\,x''$ both belong to a small neighbourhood
of $x$.
For the same reason, 
only a smoothing term (not contributing 
to the asymptotic expansion of the amplitude) is lost, if integration 
in (\ref{eqn:idempotent})
is restricted to a
suitable neighborhood of $x$. We may thus multiply
the integrand by some cut-off
function in $y$, identically equal to $1$ near $x$, and assume that $x',x'',y$ belong to some
$\delta$-neighbourhood of $x$; for notational simplicity, 
the latter cut-off will be left implicit. 
In view of (\ref{eqn:fioszego}), up to a smoothing contribution we have

\begin{eqnarray}
\label{eqn:composition szego}
\left(\Pi^\tau\circ\Pi^\tau\right)(x',x'')
&\simeq& 
\int_0^{+\infty}\mathrm{d}u\,\int_0^{+\infty}\mathrm{d}v\,\int_{X^\tau}
\mathrm{d}V_{X^\tau}(y)\\
&&\left[e^{\imath\,[u\,\psi^\tau(x',y)+v\,\psi^\tau (y,x'')]}\,
s^\tau(x',y,u)\, s^\tau(y,x'',v)  \right]\nonumber\\
&=&\int_0^{+\infty}\,I(u,x',x'')\,\mathrm{d}u,\nonumber
\end{eqnarray}
where, setting $v=u\,\sigma$,
\begin{eqnarray}
\label{eqn:inner integrand Iu}
\lefteqn{I(u,x',x'')}\\
&:=&\int_0^{+\infty}\mathrm{d}\sigma\,\int_{X^\tau}
\mathrm{d}V_{X^\tau}(y)
\left[e^{\imath\,u\,\Upsilon^\tau (x',x'';y,\sigma)}\,u\,
s^\tau(x',y,u)\, s^\tau(y,x'',u\,\sigma)  \right]\nonumber
\end{eqnarray}
with
\begin{equation}
\label{eqn:faseUsilontau}
\Upsilon^\tau (x',x'';y,\sigma):=
\psi^\tau(x',y)+\sigma\,\psi^\tau (y,x'').
\end{equation}

Since $x',x'',y$ all belong to a $\delta$-neighborhood of $x$,
$$
\mathrm{d}_y\Upsilon^\tau (x',x'';y,\sigma)=
-\alpha_y+\sigma\,\alpha_y+O(\delta);
$$
therefore, iterated \lq integration by parts\rq\, in $y$
shows that only a negligible
contribution in $u\rightarrow +\infty$ is lost, if integration in 
$\sigma$ is restricted to a suitable neighborhood of
$1$. Hence the asymptotics in $u\rightarrow +\infty$
are unaltered, if the integrand is multiplied by
$\gamma(\sigma)$, with $\gamma\in \mathcal{C}^\infty_0(\mathbb{R}_+)$
identically equal to $1$ on $(\epsilon',1/\epsilon')$, for some
$\epsilon'>0$. Thus integration in
$\sigma$ may also be assumed to be compactly supported. 
Again, the latter cut-off will be left implicit. 

As in \cite{bs},
in order to evaluate (\ref{eqn:inner integrand Iu})
asymptotically 
we first look for stationary points of $\Upsilon^\tau$
when $x'=x''=x$. To this end, let us fix Heisenberg local
coordinates on $X^\tau$ at $x$, and set $y=x+(\theta,\mathbf{v})$.
By Proposition \ref{prop:approx expr psitau},
\begin{eqnarray}
\label{eqn:local exp Upsilontau}
\lefteqn{F_x(\sigma,\theta,\mathbf{v}):=\Upsilon^\tau (x,x;x+(\theta,\mathbf{v}),\sigma)}\\
&=& -\theta+\imath\,\frac{1}{4\,\tau^2}\,\theta^2
+\frac{\imath}{2}\,\|\mathbf{v}\|^2+\sigma\,\theta
+\imath\,\frac{\sigma}{4\,\tau^2}\,\theta^2
+\imath\,\frac{\sigma}{2}\,\|\mathbf{v}\|^2
+R_3(\theta,\mathbf{v})      \nonumber\\
&=&\theta\,(\sigma-1)
+\imath\,\left[\frac{\sigma+1}{4\,\tau^2}\,\theta^2
+\frac{\sigma+1}{2}\,\|\mathbf{v}\|^2\right]+R_3(\theta,\mathbf{v})+
\sigma\,R_3(\theta,\mathbf{v}).\nonumber
\end{eqnarray}
Here $(\theta,\mathbf{v})\sim \mathbf{0}$, and the only real
critical point near the origin is $(1,0,\mathbf{0})$.
The Hessian matrix at the critical point is
\begin{equation}
\label{eqn:hessianHFx}
H(F_x)(1,0,\mathbf{0})=
\begin{pmatrix}
0&1   &\mathbf{0}^t 		\\
1& \imath/\tau^2&\mathbf{0}^t                       \\
\mathbf{0}&\mathbf{0}&2\,\imath\,I_{2d-2}
\end{pmatrix},
\end{equation}
with $\det \big(-\imath\, H(F_x)(1,0,\mathbf{0})\big)
=(-1)^d\,2^{2d-2}\cdot (-1)^{d}=2^{2d-2}\neq 0$.
In addition, 
\begin{eqnarray}
\lefteqn{\det\big(s\,I_{2\,d}-\imath\,(1-s)\, H(F_x)(1,0,\mathbf{0})\big)}\\
&=&\begin{vmatrix}
s&-\imath (1-s)  &\mathbf{0}^t 		\\
-\imath (1-s)& s+(1-s)/\tau^2&\mathbf{0}^t                       \\
\mathbf{0}&\mathbf{0}&(2-s)\,I_{2d-2}
\end{vmatrix}\nonumber\\
&=&\left[s^2+\frac{s\,(1-s)}{\tau^2}
+(1-s)^2\right]\cdot (2-s)^{2d-2}>0,\qquad \forall s\in [0,1].\nonumber
\end{eqnarray}
Hence by Theorem 2.3 of \cite{ms} we may apply the
complex version of the stationary phase Lemma, with 
$$\sqrt{\det \big(-\imath\, H(F_x)(1,0,\mathbf{0})\big)}=
2^{d-1}.
$$
Recalling (\ref{eqn:loc vol form}), to leading order in $u$ 
we obtain for (\ref{eqn:inner integrand Iu})
\begin{eqnarray}
\label{eqn:I'stationary}
\lefteqn{I(u,x',x'')}\nonumber\\
&\sim& 
e^{\imath\,u\,\psi_1(x',x'')}(2\,\pi)^d\,\frac{1}{2^{d-1}}\,
u^{-d}\,u\,
s_0^\tau\big(x',y_c(x',x'')\big)\cdot s_0^\tau\big(y_c(x',x''),x''\big)\,u^{2d-2}\,
\frac{2^{d-1}}{\tau}\nonumber\\
&= &e^{\imath\,u\,\psi_1(x',x'')}(2\,\pi)^d\,
s_0^\tau\big(x',y_c(x',x'')\big)\cdot 
s_0^\tau\big(y_c(x',x''),x''\big)\,u^{d-1}\,
\frac{1}{\tau}
\end{eqnarray}
where $\psi_1(x',x'')$ is the critical value of a holomorphic extension
of $\Upsilon^\tau$.
More precisely, let $\tilde{X}^\tau$ denote a complexification of 
the real-analytic manifold $X^\tau$.  
The real-analytic function  
$\Upsilon^\tau(x',x'';\cdot,\cdot):X^\tau\times \mathbb{R}\rightarrow
\mathbb{C}$,
depending on the parameter $(x',x'')\in X^\tau\times X^\tau$,
extends uniquely to a holomorphic function
$\tilde{\Upsilon}^\tau(x',x'';\cdot,\cdot)$ to an open neighbourhood
of $X^\tau\times \mathbb{R}$ in $\tilde{X}^\tau\times \mathbb{C}$.
We have seen that $\Upsilon^\tau(x,x;\cdot,\cdot)$ admits a unique
and non-degenerate crtitical point near $(x,1)$, namely
$(x,1)$. By the theory in \cite{ms}, 
$\tilde{\Upsilon}^\tau(x',x'';\cdot,\cdot)$ 
has a unique critical point 
$\big(y_{cr}(x',x''),\sigma_{cr}(x',x'')\big)\in \tilde{X}^\tau\times 
\mathbb{C}$ 
for $x',\,x''\sim x$ which tends to $(x,1)$ when $x',x''\rightarrow x$.
Then
\begin{equation}
\label{eqn:defn of psi1}
\psi_1(x',x''):=
\Upsilon^\tau\big(x',x'',y_{cr}(x',x''),\sigma_{cr}(x',x'')\big).
\end{equation}

\begin{prop}
\label{prop:psi1=psi}
$\psi_1(x',x'')=\psi(x',x'')$ for all $x',x''\sim x$ in $X^\tau$.
\end{prop}

This is (a special case of) Proposition 4.8 of \cite{bs}.
Nonetheless, we provide the 
proof below for the reader's convenience and 
because the argument appears somewhat more concrete 
in the current real-analytic setting.

Let us premise a few remarks. Let us fix normal Heisenberg coordinates
on $\tilde{M}$
adapted to $X^\tau$ at $x$, and with abuse of notation identify
functions on $M$ with their coordinate representation.
As above, let us identify $\tilde{M}$ (as a differentiable manifold)
with the totally real submanifold $\Delta_{\tilde{M}}:=\mathrm{diag}(\tilde{M})\subset 
\tilde{M}\times \overline{\tilde{M}}$. In local coordinates, we shall
write $Z$ for $x+Z$, where $Z=(z_0,z')$. Thus $Z$ is mapped to
$(Z,Z)\in \Delta_{\tilde{M}}$.
In addition, $\phi^\tau$ may be written (locally near $x$) as a convergent
power series $\phi^\tau (Z,\overline{Z})$; its holomorphic
extention to $\tilde{M}\times \overline{\tilde{M}}$ is then locally
given by $\tilde{\phi}^\tau\left(Z,\overline{W}\right)$. By
(\ref{eqn:psitauextol}),
\begin{equation}
\label{eqn:psi local coordinates}
\psi(x+Z,x+W)=\frac{1}{\imath}\,\tilde{\phi}^\tau\left(Z,\overline{W}\right).
\end{equation}

By the embedding $X^\tau\hookrightarrow \tilde{M}\cong
\Delta_{\tilde{M}}\subset \tilde{M}\times \overline{\tilde{M}}$,
we can locally realize the real-analytic hyprsurface
$X^\tau$ as the manifold of 
$\mathbb{C}^d\times \overline{\mathbb{C}^d}$
$$
{X^\tau}'
:=\left\{(Z,Z)\,:\,\tilde{\phi}^\tau \left(Z,\overline{Z}\right)
=0\right\}.
$$
The complexification $\widetilde{X^\tau}$ is then locally describable
as the holomorphic hypersurface
of $\mathbb{C}^d\times \overline{\mathbb{C}^d}$
$$
{\widetilde{X^\tau}}':=
\left\{(Z,W)\,:\,\tilde{\phi}^\tau \left(Z,\overline{W}\right)
=0\right\}
$$
(to be precise, 
here $(Z,W)$ belongs to a neighbourhood of the diagonal).

If $(Z,W)\in {\widetilde{X^\tau}}'$, the holomorphic tangent
space to ${\widetilde{X^\tau}}'$ is
$$
T^{(1,0)}_{(Z,W)}{\widetilde{X^\tau}}'
=\left\{(\delta Z, \delta W)\,:\,\langle\partial_Z\tilde{\phi}^\tau,
\delta Z\rangle+\langle\overline{\partial}_W\tilde{\phi}^\tau,
\overline{\delta\, W}\rangle =0\right\},
$$
with complex multiplication given by 
$\lambda\cdot (\delta Z,\delta W):=
(\lambda\,\delta\,Z,\overline{\lambda}\,\delta\,W)$.

Let us consider the complex vector sub-bundles
$H',\,H'' \subseteq T^{(1,0)}{\widetilde{X^\tau}}'$ given by
\begin{eqnarray}
\label{eqn:horprime+prime}
H'_{(Z,W)}&:=&\left\{(\delta Z, 0)\,:\,\left\langle
\left.\partial_Z\tilde{\phi}^\tau\right|_{(Z,W)},
\delta Z\right\rangle =0\right\},\\
H''_{(Z,W)}&:=&\left
\{(0, \delta W)\,:\,\left\langle\left.
\overline{\partial}_W\tilde{\phi}^\tau\right|_{(Z,W)},
\overline{\delta\, W}\right\rangle  =0
\right\};\nonumber
\end{eqnarray}
we have emphasized that 
$\partial\tilde{\phi}^\tau$ only involves $Z$-derivatives,
while $\overline{\partial}_W\tilde{\phi}^\tau$ only involves
$W$-derivatives, since $\tilde{\phi}^\tau$ is $(J,-J)$-holomorphic.

Restricted to $X^\tau$, $H'$ (respectively, $H''$) is the vector bundle
of tangent vectors tangent to $X^\tau$
and of type $(1,0)$ (respectively, $(0,1)$). Furthermore, if
$(Z,Z)\in X^\tau$ then $H'_{(Z,Z)}=\overline{H''_{(Z,Z)}}$, and
$H'_{(Z,Z)}$ and $H''_{(Z,Z)}$ are non-singularly paired under the
Levi form along $X^\tau$. Sufficiently close to $X^\tau$ in 
$\tilde{X}^\tau$, therefore, by continuity 
$H'_{(Z,W)}$ and $H''_{(Z,W)}$ are still
non-singularly paired under the
Levi form (locally represented by the matrix 
$\left[\partial ^2\tilde{\phi}^\tau/
\partial Z_i\,\partial\overline{W}_j\right])$.

\begin{proof}
[Proof of Proposition \ref{prop:psi1=psi}]
The real-analytic phase $\Upsilon^\tau (x',x'';\cdot,\cdot):X^\tau\times \mathbb{R}_+
\rightarrow \mathbb{C}$ in (\ref{eqn:faseUsilontau}) may be
locally expressed as follows. 
Let us write $x'=x+Z',\,x''=x+Z'', y=x+U\in X^\tau$, 
corresponding to pairs $(Z',Z'),\,(Z'',Z''), (U,U)$ satisfying
$\tilde{\phi}^\tau(Z',\overline{Z}')
=\tilde{\phi}^\tau(Z'',\overline{Z}'')
=\tilde{\phi}^\tau (U,\overline{U})=0$.
Then by (\ref{eqn:faseUsilontau}) and
(\ref{eqn:psi local coordinates}) we have
\begin{equation}
\label{eqn:faseUsilontau1}
\Upsilon^\tau (x',x'';y,\sigma)=
\frac{1}{\imath}\,\left[ \tilde{\phi}^\tau(Z',\overline{U})
+\sigma\,\tilde{\phi}^\tau(U,\overline{Z}'')    \right].
\end{equation}

Consider the holomorphic extension
$\tilde{\Upsilon}^\tau (x',x'';\cdot,\cdot):\tilde{X}^\tau\times \mathbb{C}
\rightarrow \mathbb{C}$. Let us write the general point 
$\tilde{y}\in \tilde{X}^\tau$ near $x$, by the previous identifications,
as $(x+U,x+W)$, corresponding to a pair $(U,W)$ with
$\tilde{\phi}^\tau\left(U,\overline{W}\right)=0$.
Then
\begin{equation}
\label{eqn:faseUsilontau2}
\tilde{\Upsilon}^\tau (x',x'';\tilde{y},\tilde{\sigma})=
\frac{1}{\imath}\,\left[ \tilde{\phi}^\tau(Z',\overline{W})
+\tilde{\sigma}\,\tilde{\phi}^\tau(U,\overline{Z}'')    \right].
\end{equation}

We have seen that 
$\Upsilon^\tau (x,x;\cdot,\cdot)$
admits a unique (real)
non-degenerate critical point near $(x,1)$, 
namely $(x,1)$ itself.
For all $x',x''\sim x$, therefore,
$\tilde{\Upsilon}^\tau (x',x'';\cdot,\cdot)$
admits a unique critical
point in the complex domain near $(x,1)$,
which will be a real-analytic function of $(x',x'')$.

At such critical point, 
$0=\partial_{\tilde{\sigma}}\tilde{\Upsilon}^\tau (x',x'';\tilde{y},\tilde{\sigma})=-\imath\,\tilde{\phi}^\tau(U,\overline{Z}'')$.
Hence, $(U,Z'')\in {\widetilde{X^\tau}}'$.

Let us consider the subspace $H'_{(U,W)}\subset T^{1,0}_{(U,W)}{\widetilde{X^\tau}}'$. By its definition in (\ref{eqn:horprime+prime}), 
\begin{equation}
\label{eqn:1st vanishing}
\partial_U \left.\tilde{\phi}^\tau\right|_{(U,W)}
=0\quad 
\text{on}\quad
H'_{(U,W)}.
\end{equation}

On the other hand, since $(\tilde{y},\tilde{\sigma})$ 
is a critical point of (\ref{eqn:faseUsilontau2}),
and $\tilde{y}$ corresponds to $(U,W)$ in local coordinates, 
by (\ref{eqn:faseUsilontau2}) we also
have 
\begin{equation}
\label{eqn:2nd vanishing}
\partial_U \left.\tilde{\phi}^\tau\right|_{(U,Z'')}
=0\quad 
\text{on}\quad
H'_{(U,W)}.
\end{equation}

Since 
$(U,W),\, (U,Z'')\in {\widetilde{X^\tau}}'$,
to first order in $Z''-W$, and with some abuse of notation,
we may regard
$(0,Z''-W)$ as an element of $T^{1,0}_{(U,W)}{\widetilde{X^\tau}}'$.
We have
$$
\partial_U \left.\tilde{\phi}^\tau\right|_{(U,Z'')}=
\partial_U \left.\tilde{\phi}^\tau\right|_{(U,W)}
+\left[\frac{\partial^2\tilde{\phi}^\tau }{\partial U_i\,\partial 
\overline{W}_j}\right]\,(\overline{Z}''-\overline{W})+R_2(Z''-W),
$$
where $R_2$ vanishes to second order at the origin. By
(\ref{eqn:1st vanishing}) and (\ref{eqn:2nd vanishing}),
the tangent vector $Z''-W$ is in the kernel of the Levi form at
$(U,W)$. Since the latter is non-degenerate, we conclude
that $W=Z''$.

Since $\tilde{\phi}^\tau(U,\overline{W})=0$ 
and $\tilde{\phi}^\tau(Z',\overline{Z}'')=\imath\,\psi (x',x'')$,
in view of (\ref{eqn:defn of psi1})
the claim follows by replacing $W$ with $Z''$ in 
(\ref{eqn:faseUsilontau2}).

\end{proof}

We can now conclude the proof of Theorem
\ref{thm:leading order term Pi}.
By (\ref{eqn:composition szego}), Proposition \ref{prop:psi1=psi}), and idempotency,
$\Pi^\tau$ is a Fourier integral operator
with phase $\psi^\tau$ and a symbol of order $d-1$
whose leading order term must coincide with the one in 
(\ref{eqn:fioszego}). Since $y_c(x,x)=x$, equating the 
leading order coefficients in 
(\ref{eqn:semiclasstau}) and (\ref{eqn:composition szego}),
we obtain
$$
s^\tau_0(x,x)=
(2\,\pi)^d\,
s_0^\tau\big(x,x\big)^2\,
\frac{1}{\tau}.
$$
The claim follows.
\end{proof}

\subsection{Dynamical Toeplitz operators}
\label{sctn:dyn toepl ops}

As mentioned in the Introduction, the homogenous
geodesic flow is generally not holomorphic for
$J_{ad}$.
Equivalently, the $(1,0)$-component of $\upsilon_{\sqrt{\rho}}$
needn't be holomorphic on
$\tilde{M}\setminus M$. 
Therefore, when
viewed a differential operator,
$\upsilon_{\sqrt{\rho}}^\tau$ does generally 
not preserve the Hardy space
$H(X^\tau)$. A natural replacement is the 
self-adjoint, first order Toeplitz 
operator $\mathfrak{D}_{\sqrt{\rho}}^\tau:=\Pi^\tau\circ D_{\sqrt{\rho}}\circ \Pi^\tau$ (\S \ref{sctn:symbol toeplitz}).
The latter generates the $1$-parameter group of unitary Toeplitz operators
\begin{equation}
\label{eqn:unitary group}
U_{\sqrt{\rho}}^\tau(t):=e^{\imath\,t\,\mathfrak{D}_{\sqrt{\rho}}^\tau}:H(X^\tau)\rightarrow H(X^\tau).
\end{equation}
In view of (\ref{eqn:defn di Dsqrtrho}), 
$$\imath\,t\,\mathfrak{D}_{\sqrt{\rho}}^\tau=
\imath\,t\,\Pi^\tau\circ D_{\sqrt{\rho}}\circ \Pi^\tau
=-t\,\Pi^\tau\circ\upsilon_{\sqrt{\rho}}^\tau
\circ \Pi^\tau.
$$
Hence, heuristically $U_{\sqrt{\rho}}^\tau(t)$ is a Toeplitz
quantization of the geodesic flow at time $-t$.
In the notation of the Introduction, the distributional kernel 
$U_{\sqrt{\rho}}^\tau(t;\cdot,\cdot)\in \mathcal{D}'(X^\tau\times X^\tau)$
of (\ref{eqn:unitary group})
admits the spectral description
\begin{equation}
\label{eqn:spectral dec Dlambda}
U_{\sqrt{\rho}}^\tau(t;x,y)
=\sum_{j=1}^{+\infty}e^{\imath\,t\,\lambda_j}
\cdot\sum_{k=1}^{\ell_j}\,\rho_{j,k}(x)
\cdot \overline{\rho_{j,k}(y)}.
\end{equation}
Arguing, say, as in \S 12 of \cite{gs}, and
using (\ref{eqn:smoothed kernel}) and (\ref{eqn:fourtransform}),
one obtains
%
\begin{eqnarray}
\label{eqn:smoothed kernel deriv}
\Pi^\tau_{\chi,\lambda}(x,y)&=&
\sum_{j=1 }^{+\infty}\,\hat{\chi}(\lambda-\lambda_j)\,\sum_{k=1}^{\ell_j}\,\rho_{j,k}(x)
\cdot \overline{\rho_{j,k}(y)}\\
&=&\frac{1}{\sqrt{2\,\pi}}\,
\int_{-\infty}^{+\infty}e^{-\imath\,\lambda\,t}\,
\chi(t)\,U_{\sqrt{\rho}}^\tau(t;x,y)\,\mathrm{d}t.\nonumber
\end{eqnarray}

It was shown by Zelditch that, up to smoothing Toeplitz operators,
$U_{\sqrt{\rho}}^\tau(t)$ is a \lq dynamical Toeplitz operator\rq\, associated to the geodesic flow at time $-t$, composed with
a suitable pseudodifferential operator
(see e.g. \S 5.3 of \cite{z10}).
To express this precisely, recall that 
$\Gamma^\tau_t:X^\tau\rightarrow X^\tau$ denotes the
geodesic flow along $X^\tau$ (Corollary \ref{cor:integral curve geodlocal}); for $t\in \mathbb{R}$ let us set
$\Pi_t^{\tau}:={\Gamma^\tau_t}^*\circ \Pi^\tau$. 
Thus $\Pi_t^{\tau}$ has Schwartz kernel
$\Pi_t^{\tau}(x,y):=\Pi^\tau\left({\Gamma^\tau_t}(x),y\right)$. 

\begin{rem}
When defining the pull-back under a diffeomorphism, one ought to
distinguish whether $\Pi^\tau$ is referred to as a (generalized)
function, density, or half-density. There is no ambiguity
in the present case, since these are being identified
by means of $\mathrm{vol}^R_{X^\tau}$, which is invariant under
$\Pi_t^{\tau}$.
\end{rem}

\begin{thm} (Zelditch)
\label{thm:dyn Usqrtrho}
There exist
a zeroth order polyhomogeneous complete classical symbol of the form
\begin{equation}
\label{eqn:symbol for Usqrtrho}
\sigma^\tau_t(x,r)\sim \sum_{j=0}^{+\infty}
\sigma^\tau_{t,j}(x)\,r^{-j}\quad (x\in X^\tau,\,r>0),
\end{equation}
and a zeroth order pseudodifferential operator 
$P^\tau_t
\sim  \sigma^\tau_t(x,D_{\sqrt{\rho}}^\tau)$
such that
\begin{equation}
\label{eqn:key dynamical approximation}
U_{\sqrt{\rho}}^\tau(t)\simeq \Pi^\tau\circ P^\tau_t
\circ \Pi_{-t}^{\tau},
\end{equation}
where $\simeq$ means \lq equal modulo smoothing Toeplitz operators\rq.
Furthermore, the principal symbol $\sigma^\tau_{t,0}(x)$
is the inverse of the 
$L^2$-pairing of two normalized Gaussian functions, 
related to complex structures $J_x$ and ${J_t}_x$, 
respectively, where
$J_t$ is the push-forward of $J$ under the flow of
$\upsilon_{\sqrt{\rho}}^\tau$
at time $t$;
in particular, $\sigma^\tau_{0,0}(x)=1$.
\end{thm}

Up to smoothing terms, the Schwartz kernel of $\Pi_{-t}^{\tau}$
has the form
\begin{equation}
\label{eqn:Ptaut}
\Pi_{-t}^{\tau}(x_1,x_2)\sim\int_0^{+\infty}
e^{\imath\,u\,\psi^\tau\left(\Gamma^\tau_{-t}(x_1),x_2  \right)}\,
s^\tau\left(\Gamma^\tau_{-t}(x_1),x_2,u\right)\,\mathrm{d}u,
\end{equation}
where $\psi^\tau$ and $s^\tau$ are as in (\ref{eqn:fioszego}).
Using classical results on the composition of pseudodifferential and
Fourier integral operators (\cite{shu}, \cite{t}), 
we reach the following conclusion.

\begin{lem}
\label{lem:calPtaut}
Up to smoothing terms, the Schwartz kernel 
of the composition
$\mathcal{P}^\tau_t:=
P^\tau_t\circ \Pi_{-t}^{\tau}$ has the form
\begin{equation}
\mathcal{P}^\tau_t(x_1,x_2)\sim\int_0^{+\infty}
e^{\imath\,u\,\psi^\tau\left(\Gamma^\tau_{-t}(x_1),x_2  \right)}\,
r^\tau_t\left(x_1,x_2,u\right)\,\mathrm{d}u,
\end{equation}
where
$$
r^\tau_t\left(x_1,x_2,u\right)\sim 
\sum_{j\ge 0}\,u^{d-1-j}\,r^\tau_{t\,j}(x_1,x_2),\quad 
r^\tau_{t\,0}(x_1,x_2)=
\sigma^\tau_{t,0}\left(x_1\right)\cdot
s^\tau_0\left(\Gamma^\tau_{-t}(x_1),x_2\right).
$$
\end{lem}

By (\ref{eqn:key dynamical approximation}),
we have
\begin{equation}
\label{eqn:key dyn approx P}
U_{\sqrt{\rho}}^\tau(t)\simeq \Pi^\tau\circ \mathcal{P}^\tau_t.
\end{equation}

\section{Near-diagonal asymptotics for $\Pi^\tau_{\chi,\lambda}$}

Before delving into the proof of Theorem \ref{thm:absolute main thm},
let us premise some notation.

The choice of a normal Heisenberg local chart $\varphi^\tau$ for
$X^\tau$ at $x$ determines an isomorphism
$T_xX^\tau\cong \mathbb{R}\times \mathbb{R}^{2d-2}$;
a general $\upsilon\in T_xX^\tau$ will be written accordingly
as a pair $\upsilon=(\theta,\mathbf{u})$.
The subspaces $\mathbb{R}\times \{\mathbf{0}\}$ and
$\{0\}\times \mathbb{R}^{2d-2}$ correspond, respectively, to 
$\mathcal{T}^\tau(x)$ and
$\mathcal{H}^\tau(x)$ (see (\ref{eqn:directsum tXtau})).
By (\ref{eqn:sq norm tildekappaV}), the 
isomorphism $\mathbb{C}^{d-1}\cong \{0\}\times \mathbb{R}^{2d-2}
\rightarrow \mathcal{H}^\tau(x)$ is unitary, when $\mathcal{H}^\tau(x)$
is endowed with the Hermitian structure associated to 
$\omega_x=\frac{1}{2}\,\Omega_x$.

With the notation of Corollary \ref{cor:derivata upsilon theta}, 
let $(a_x,\mathbf{A}_x)\in \mathbb{R}\times\mathbb{R}^{2d-2}$ be
defined by
\begin{equation}
\label{eqn:FexpansionaA}
\tau\,f(0,\vartheta,\mathbf{u})=\frac{1}{2\,\tau^2}\,
a_x\,\vartheta+\langle \mathbf{A}_x,\mathbf{u}\rangle
+F_2^\tau(\vartheta,\mathbf{u}),
\end{equation}
where $F_2^\tau$ vanishes to second order at
the origin. We may then reformulate the conclusion of Corollary 
\ref{cor:integral curve geodlocal} writing
\begin{eqnarray}
\label{eqn:geodesflow rescaled}
\Gamma^\tau_t(y)&=&x+\left(\vartheta -\tau\,t+
t\,\left(\frac{1}{2\,\tau^2}\,
a_x\,\vartheta+\langle \mathbf{A}_x,\mathbf{u}\rangle \right)+
R_3(\tau\,t,\vartheta,\mathbf{u}),\right.\nonumber\\
&&\mathbf{u}
+ \mathbf{R}_2(\tau\,t,\vartheta,\mathbf{u})\Big).
\end{eqnarray}
We shall verify \textit{a posteriori} that
$(a_x,\mathbf{A}_x)=(0,\mathbf{0})$.

\begin{defn}
\label{defn:newPsi2inv}
With $(a_x,\mathbf{A}_x)$ as in (\ref{eqn:FexpansionaA}),
let us set
\begin{equation*}
\mathbf{a}_x(\theta_1,\theta_2):=\frac{\theta_1-\theta_2}{\tau}\,\mathbf{A}_x.
\end{equation*}
We further define 
$$\Psi_2:T_xX^\tau
\times T_xX^\tau\rightarrow
\mathbb{C}
$$
by setting, given $\upsilon_j=(\theta_j,\mathbf{v}_j)$,
\begin{eqnarray}
\label{eqn:defn di PSI0}
\Psi_2(\upsilon_1,\upsilon_2)&:=&
-\imath\,\omega_x\left(
\mathbf{v}_1,\mathbf{v}_2\right)-\frac{\imath}{2}\,\omega_x\left(
J_x\big(\mathbf{a}_x(\theta_1,\theta_2)\big),\mathbf{v}_1+\mathbf{v}_2\right)\\
&&-\frac{1}{4}\,\|\mathbf{v}_1-\mathbf{v}_2\|^2
-\frac{1}{4}\,\left\|\mathbf{v}_1-\mathbf{v}_2
+J_x\big(\mathbf{a}_x(\theta_1,\theta_2)\big)\right\|^2.
\nonumber
\end{eqnarray}

\end{defn}


\begin{proof}
[Proof of Theorem \ref{thm:absolute main thm}]
By (\ref{eqn:smoothed kernel deriv}) and 
(\ref{eqn:key dynamical approximation}), we can rewrite (\ref{eqn:smoothed kernel}) in the
following form:

\begin{eqnarray}
\label{eqn:smoothed kernel exp}
\lefteqn{\Pi^\tau_{\chi,\lambda}(x_1,x_2)}\\
&\sim&
\frac{1}{\sqrt{2\,\pi}}\,\int_{X^\tau}
\int_{-\epsilon}^{+\epsilon}\left[e^{-\imath\,\lambda\,t}\,
\chi(t)\,\Pi^\tau (x_1,y)\,
\left[P^\tau_t\circ \Pi_{-t}^{\tau}\right](y,x_2)\right]\,\mathrm{d}t\,\mathrm{d}V_{X^\tau}(y)\nonumber,
\end{eqnarray}
where $\sim$ stands for \lq has the same asymptotics 
for $\lambda\rightarrow +\infty$ as\rq.

%


The wave front set of $\Pi^\tau $ is given by (\ref{eqn:sigmatauantidiag}).
Therefore, given that the geodesic flow preserves $\alpha$
(Lemma \ref{lem:primitive geod inv}),
the wave front of
$P^\tau_t\circ \Pi_{-t}^{\tau}$ is
\begin{eqnarray}
\label{eqn:wavefront sigmataut}
\mathrm{WF}\left( P^\tau_t\circ \Pi_{-t}^{\tau} \right)
&=&\mathrm{WF}\left( \Pi_{-t}^{\tau} \right)
=\left(\Gamma_{-t}^{\tau}\times \mathrm{id}_{X^\tau}\right)^*
\big(\mathrm{WF}(\Pi^\tau)   \big)\nonumber\\
&=&
\left\{\left(\Gamma^\tau_{t}(x),
r\alpha^\tau_{\Gamma^\tau_{t}(x)},
x,-r\,\alpha^\tau_x\right)\,:\,
x\in X^\tau,\,r>0\right\}.\nonumber
\end{eqnarray}
Hence the singular supports of $\Pi^\tau$ 
and $P^\tau_t\circ \Pi_{-t}^{\tau}$ are, respectively,
the diagonal and the graph of $\Gamma^\tau_{-t}$
in $X^\tau\times X^\tau$.

Let $c>0$ be such that 
\begin{equation}
\label{eqn:defn of c}
\mathrm{dist}_{X^\tau}\left(\Gamma^\tau_{-t}(x),x\right)
\le c\,|t|,\quad
\forall\, x\in X^\tau,\,t\in \mathbb{R}.
\end{equation}

\begin{lem}
\label{eqn:distance estimate}
Provided $\epsilon>0$ is sufficiently small,
if $\mathrm{dist}_{X^\tau}(x_1,x_2)\ge 3\,c\,\epsilon$
then 
$$
\max \left\{\mathrm{dist}_{X^\tau}\left(x_1,y\right),
\mathrm{dist}_{X^\tau}\left(\Gamma^\tau_{-t}(y),x_2\right)\right\}
\ge c\,\epsilon, \quad \forall \,y\in X^\tau,\,
t\in \mathrm{supp}(\chi).
$$
\end{lem}

\begin{proof}
For any $y\in X^\tau$ and $t\in (-\epsilon,\epsilon)$
\begin{eqnarray}
\label{eqn:distance x12}
3\,c\,\delta&\le& \mathrm{dist}_{X^\tau}\left(x_1,x_2\right)  \\
&\le &\mathrm{dist}_{X^\tau}\left(x_1,y\right)
+\mathrm{dist}_{X^\tau}\left(y,\Gamma^\tau_{-t}(y)\right)+
\mathrm{dist}_{X^\tau}\left(\Gamma^\tau_{-t}(y),x_2\right)
\nonumber\\
&\le&\mathrm{dist}_{X^\tau}\left(x_1,y\right)
+c\,\epsilon +
\mathrm{dist}_{X^\tau}\left(\Gamma^\tau_{-t}(y),x_2\right).
\nonumber
\end{eqnarray}
Hence, 
$$
\mathrm{dist}_{X^\tau}\left(x_1,y\right)+
\mathrm{dist}_{X^\tau}\left(\Gamma^\tau_{-t}(y),x_2\right)\ge 
2\,c\,\epsilon\quad \forall\,t\in \mathrm{supp}(\chi).
$$
\end{proof}

\begin{lem}
\label{lem:rapid decrease distance}
For any $\epsilon\sim 0^+$, if 
$\mathrm{supp}(\chi)\subset [-\epsilon,\epsilon]$ and 
$\mathrm{dist}_{X^\tau}(x_1,x_2)\ge 3\,c\,\epsilon$,
then $\Pi^\tau_{\chi,\lambda}(x_1,x_2)=O\left(\lambda^{-\infty}\right)$.
\end{lem}

\begin{proof}
Let us define
\begin{eqnarray*}
U_1&:=&\left\{y\in X^\tau\,:\,
\mathrm{dist}_{X^\tau}\left(x_1,y\right)>\frac{c\,\epsilon}{2}\right\},
\\
U_2&:=&\left\{y\in X^\tau\,:\,
\mathrm{dist}_{X^\tau}\left(\Gamma^\tau_{-t}(y),x_2\right)>\frac{c\,\epsilon}{2},\,\forall\,t\in \mathrm{supp}(\chi)\right\}.
\end{eqnarray*}
By Lemma \ref{eqn:distance estimate}, 
$\mathcal{U}:=\{U_1,\,U_2\}$ is an open cover of $X^\tau$.
Let $\gamma_1+\gamma_2=1$ be a partition of unity of $X^\tau$ subordinate to $\mathcal{U}$. Then
$$
\Pi^\tau_{\chi,\lambda}(x_1,x_2)\sim
\Pi^\tau_{\chi,\lambda}(x_1,x_2)_1+
\Pi^\tau_{\chi,\lambda}(x_1,x_2)_2,
$$
where
\begin{eqnarray}
\label{eqn:smoothed kernel exp j}
\lefteqn{\Pi^\tau_{\chi,\lambda}(x_1,x_2)_j}\\
&:=&
\frac{1}{\sqrt{2\,\pi}}\,
\int_{-\infty}^{+\infty}e^{-\imath\,\lambda\,t}\,
\chi(t)\,\left[\int_{U_j}\gamma_j(y)\,\Pi^\tau (x_1,y)\,
\left[P^\tau_t\circ \Pi_{-t}^{\tau}\right](y,x_2)\,\mathrm{d}V_{X^\tau}(y)\right]\,\mathrm{d}t\nonumber.
\end{eqnarray}
By definition of $U_1$, the function $y\mapsto\gamma_1(y)\cdot \Pi^\tau (x_1,y)$ is 
$\mathcal{C}^\infty$ (and depends smoothly on $x_1$); therefore,
for $j=1$
the inner integral in (\ref{eqn:smoothed kernel exp j}) is 
the distributional kernel of the composition of a smoothing kernel
with $P^\tau_t$, whence it is itself a smoothing kernel.
In other words, it
a smooth
function of $(x_1,x_2,t)$, uniformly so under 
the assumption. Therefore,
$\Pi^\tau_{\chi,\lambda}(x_1,x_2)_1$ is the
evaluation at $\lambda$ of the Fourier transform
of a compactly supported $\mathcal{C}^\infty$ function of $t$,
thus it is $O\left(\lambda^{-\infty}\right)$.

Similarly, by definition of $U_2$, the function 
$y\mapsto \gamma_2(y)\cdot 
\left[P^\tau_t\circ \Pi_{-t}^{\tau}\right](y,x_2)$ is 
$\mathcal{C}^\infty$ (and depends smoothly on $x_2$),
whence
$\Pi^\tau_{\chi,\lambda}(x_1,x_2)_2=
O\left(\lambda^{-\infty}\right)$ by a similar argument.

\end{proof}

Assume from now on $\mathrm{dist}_{X^\tau}(x_1,x_2)< 3\,c\,\epsilon$,
with $c$ as in (\ref{eqn:defn of c}).
Let us define
\begin{eqnarray}
\label{eqn:V12}
U'_1&:=&\left\{y\in X^\tau\,:\,\mathrm{dist}_X(x_1,y)< 2\,\epsilon
\right\},\\
U'_2&:=&\left\{y\in X^\tau\,:\,\mathrm{dist}_X(x_1,y)>\epsilon\right\}.\nonumber
\end{eqnarray}
Then $\mathcal{U}':=\{U_1',U_2'\}$ is an open cover of $X^\tau$;
let $\gamma_1'+\gamma_2'=1$ be a partion of unity subordinate to
$\mathcal{U}'$. Thus 
$$
\Pi^\tau_{\chi,\lambda}(x_1,x_2)\sim
\Pi^\tau_{\chi,\lambda}(x_1,x_2)'_1+
\Pi^\tau_{\chi,\lambda}(x_1,x_2)'_2,
$$
where $\Pi^\tau_{\chi,\lambda}(x_1,x_2)'_j$ is defined as
in (\ref{eqn:smoothed kernel exp j}), with $\gamma_j'$
in place of $\gamma_j$.
For $y\in U_2'$, the function $y\mapsto \Pi^\tau(x_1,y)$ is
$\mathcal{C}^\infty$, and an adaptation of the
previous argument implies that
$\Pi^\tau_{\chi,\lambda}(x_1,x_2)'_2=O\left(\lambda^{-\infty}  \right)$. Thus
\begin{equation}
\label{eqn:redy}
\Pi^\tau_{\chi,\lambda}(x_1,x_2)\sim
\Pi^\tau_{\chi,\lambda}(x_1,x_2)'_1,
\end{equation}
where integration in $y$ is now over a small neighborhood
of $x_1$.

We may assume that $x_1,x_2$ and every $y$ in the support of
the integrand in
$\Pi^\tau_{\chi,\lambda}(x_1,x_2)'_1$
belong to 
a Heisenberg coordinate neighborhood
centered at some $x\in X^\tau$. 
Without altering the asymptotics, 
$\Pi^\tau$ may be represented 
as a Fourier integral operator (\ref{eqn:fioszego}),
and apply  Lemma
\ref{lem:calPtaut} and (\ref{eqn:key dyn approx P}).
Then (\ref{eqn:smoothed kernel exp}) may be rewritten
\begin{eqnarray}
\label{eqn:smoothed kernel reexp}
\lefteqn{
\Pi^\tau_{\chi,\lambda}(x_1,x_2)
}\\
&\sim&
\frac{1}{\sqrt{2\,\pi}}\,\int_0^{+\infty}\,
\mathrm{d}u\,\int_0^{+\infty}\,\mathrm{d}v\,
\int_{X^\tau}\,\mathrm{d}V_{X^\tau}(y)
\int_{-\epsilon}^{+\epsilon}\,\mathrm{d}t\nonumber\\
&&\left[e^{\imath\,[u\,\psi^\tau(x_1,y)
+v\,\psi^\tau (\Gamma^\tau_{-t}(y),x_2  )
-\lambda\,t]}\,\gamma_1'(y)\,
\chi(t)\,s^\tau(x_1,y,u)\,
r^\tau_t\left(y,x_2,v\right)\right]\nonumber.
\end{eqnarray}
With the rescaling $u\rightsquigarrow \lambda\,u$
and $v\rightsquigarrow \lambda\,v$, (\ref{eqn:smoothed kernel reexp})
becomes
\begin{eqnarray}
\label{eqn:smoothed kernel reexp1}
\lefteqn{\Pi^\tau_{\chi,\lambda}(x_1,x_2)}\\
&\sim&
\frac{\lambda^2}{\sqrt{2\,\pi}}\,\int_0^{+\infty}\,
\mathrm{d}u\,\int_0^{+\infty}\,\mathrm{d}v\,
\int_{X^\tau}\,\mathrm{d}V_{X^\tau}(y)
\int_{-\epsilon}^{+\epsilon}\,\mathrm{d}t
\nonumber\\
&&
\left[e^{\imath\,\lambda\,\Psi^\tau(x_1,x_2;u,v,t,y)}\,
A_\lambda^\tau(x_1,x_2;u,v,t,y)
\right]\nonumber,
\end{eqnarray}
where
\begin{eqnarray}
\label{eqn:Psilambdaprimafase}
\Psi^\tau(x_1,x_2;u,v,t,y)&:=&
u\,\psi^\tau(x_1,y)
+v\,\psi^\tau (\Gamma^\tau_{-t}(y),x_2  )
-t,\\
\label{eqn:Alambda primaampiezza}
A_\lambda^\tau(x_1,x_2;u,v,t,y)&:=&
\chi(t)\,\gamma_1'(y)\,s^\tau(x_1,y,\lambda\,u)\,
r^\tau_t\left(\Gamma^\tau_{-t}(y),x_2,\lambda\,v\right).
\end{eqnarray}

We shall let $y=x+(\theta,\mathbf{u})$ in Heisenberg
local coordinates on $X^\tau$ at $x$;
then $\mathrm{d}V_{X^\tau}(y)$ by
$\mathcal{V}(\theta,\mathbf{u})\,\mathrm{d}\theta\,\mathrm{d}\mathbf{u}$, and by (\ref{eqn:loc vol form})
\begin{equation}
\label{eqn:vol form at 0}
\mathcal{V}(0,\mathbf{0})
=\frac{2^{d-1}}{\tau}.
\end{equation} 
Integration in $(\theta,\mathbf{u},t)$ is  
compactly supported near the origin.
By (\ref{eqn:Psilambdaprimafase}), 
Proposition \ref{prop:loc coord expr}, Corollaries  \ref{cor:integral curve geodlocal} and \ref{cor:diffpsi diagonal}
\begin{equation}
\label{eqn:deryPsi}
\partial_\theta\Psi^\tau=v-u+O(\epsilon),\quad
\partial_t\Psi^\tau=\tau\,v-1+O(\epsilon).
\end{equation}
It follows by a standard \lq partial integration\rq\, argument
in the compactly supported variables $(\theta,t)$
that the contribution of the locus where 
$\|(u,v)\|\gg 0$ contributes negligibly to the
asymtptotics. Similarly, if $\epsilon\ll 1$,
integration by parts in $t$ shows that the contribution
of the locus where $v\ll 1/\tau$ is also negligible. 
Finally, integration by parts in $\theta$ yields a similar
conclusion for $u$. We thus obtain the following reduction.

\begin{lem}
\label{lem:compact support uv}
There exists $D\gg 0$ such that the following holds.
Let $h\in \mathcal{C}^\infty_0\big((1/(2D),2\,D)\big)$ 
be such that
$h\equiv 1$ on $(1/D,D)$. Then 
only a negligible contribution to the asymptotics of
(\ref{eqn:smoothed kernel reexp1}) is lost, 
if the integrand is multiplied by
$h(u)\cdot h(v)$. Hence, integration in $(u,v)$ may
be assumed to be compactly supported in $(1/(2D),2\,D)^2$.
\end{lem}


\begin{proof}[Proof of Statement 1) of Theorem 
\ref{thm:absolute main thm}]
There exist constants
$0<a\le A$ such that
$$
a\,\mathrm{dist}_{X^\tau}\left(\Gamma^\tau_{-t}(y),x  \right)\le 
\mathrm{dist}_{X^\tau}\left(y,\Gamma^\tau_{t}(x)  \right)\le 
A\,\mathrm{dist}_{X^\tau}\left(\Gamma^\tau_{-t}(y),x  \right)
$$
for every $t\in (-\epsilon,\epsilon)$ and $y,\,x\in X^\tau$.

Assume $\mathrm{dist}_{X^\tau}\left(x_1,x_2^\chi\right)
\ge C\,\lambda^{\delta-1/2}$, and define
\begin{eqnarray*}
U^\lambda_1&:=&\left\{y\in X^\tau \,:\,\mathrm{dist}_{X^\tau}(y,x_1)
<(C/2)\,\lambda^{\delta-1/2}      \right\},\\
U^\lambda_2&:=&\left\{y\in X^\tau \,:\,\mathrm{dist}_{X^\tau}(y,x_1)
>(C/3)\,\lambda^{\delta-1/2}      \right\}.
\end{eqnarray*}
Then $\mathcal{U}^\lambda:=\{U^\lambda_1,\,U^\lambda_2\}$
is an open cover of $X^\tau$. 

If $y\in U^\lambda_1$, for any $t\in 
(-\epsilon,\epsilon)$ we have
\begin{eqnarray*}
C\,\lambda^{\delta-1/2}&\le&\mathrm{dist}_{X^\tau}\left(x_1,
,\Gamma^\tau_{t}(x_2)\right)\le \mathrm{dist}_{X^\tau}(x_1,y)
+\mathrm{dist}_{X^\tau}\left(y,\Gamma^\tau_{t}(x_2)  \right)
\nonumber\\
&<&(C/2)\,\lambda^{\delta-1/2}
+\mathrm{dist}_{X^\tau}\left(y,\Gamma^\tau_{t}(x_2)  \right)\nonumber\\
&\Rightarrow&\mathrm{dist}_{X^\tau}\left(\Gamma^\tau_{-t}(y),x_2  \right)
\ge \frac{1}{A}\,
\mathrm{dist}_{X^\tau}\left(y,\Gamma^\tau_{t}(x_2)  \right)\ge 
\frac{C}{2\,A}\,\lambda^{\delta-1/2}.
\end{eqnarray*}
By Corollary \ref{cor:bound psi distance}, in the
same range 
$$
\left|\frac{\partial\Psi^\tau}{\partial v}\right|=
\left|\psi^\tau\left(\Gamma^\tau_{-t}(y),x_2  \right)    \right|
\ge D^\tau\,\left(\frac{C}{2\,A}\,\right)^2\lambda^{2\,\delta-1}.
$$
Integrating by parts in $v$, we obtain that the contribution
of $U_1^\lambda$ to the asymptotics of (\ref{eqn:smoothed kernel reexp1}) is $O\left(\lambda^{-\infty}\right)$.

On the other hand, if $y\in U^\lambda_2$ then
\begin{eqnarray*}
\left|\frac{\partial\Psi^\tau}{\partial u}\right|=
\left|\psi^\tau\left(y,x_1\right)    \right|
\ge D^\tau\,
\mathrm{dist}_{X^\tau}(y,x_1)^2
>\frac{D^\tau\,C^2}{9}\,\lambda^{2\,\delta-1}.
\end{eqnarray*}
Integrating by parts in $u$, we obtain that the contribution
of $U_2^\lambda$ is also $O\left(\lambda^{-\infty}\right)$.
\end{proof}

We focus on statement 2.
Let us set
\begin{equation}
\label{eqn:xjlambda}
x_{j\lambda}:=x+\left( \frac{\theta_j}{\sqrt{\lambda}},\frac{\mathbf{v}_j}{\sqrt{\lambda}} \right).
\end{equation}
%
Under the assumptions, 
\begin{equation}
\label{eqn:distance xxjlambdaest}
\mathrm{dist}_{X^\tau}(x,x_{j\lambda})\le 
2\,C\,\lambda^{\delta-\frac{1}{2}}.
\end{equation}


Let $h^\tau$ be a Riemannian metric on $X^\tau$ that in a sufficiently 
small neighborhood of $x$ is given in Heisenberg local coordinates by
$$
h^\tau=\frac{1}{\tau^2}(\mathrm{d}\theta)^2+
\mathrm{d}z'\otimes \mathrm{d}\overline{z}'+
\mathrm{d}\overline{z}'\otimes \mathrm{d}z'.
$$
Thus $h^\tau_x=\hat{\kappa}^\tau_x$ by
Corollary \ref{cor:heis coord norms}.
Let $\widetilde{\mathrm{dist}}_{X^\tau}:X^\tau\times X^\tau\rightarrow
\mathbb{R}$ be the Riemannian distance of
$h^\tau$. By the latter remark,
$$\widetilde{\mathrm{dist}}_{X^\tau}(x,y)/\mathrm{dist}_{X^\tau}(x,y)
\rightarrow 1 \quad \text{for}\quad y\rightarrow x.
$$

For $r>0$, let 
$B_x(r)\subseteq X^\tau$ be the open ball centered 
at $x$ for $h^\tau$.
Let $r$ be small enough that $B_x(r)\subset U^\tau$,
and consider the open cover of $X^\tau$ $$\mathcal{B}:=\left\{B_x(r), \overline{B_x(r/2)}^c\right\}.$$
Let
$(b_1,b_2)$ be a partition of unity
subordinate to $\mathcal{B}$. 
For some fixed $R>0$ and any $\lambda\gg 1$, let 
$b_j^\lambda\in \mathcal{C}^\infty (X^\tau)$ be defined in Heisenberg
local coordinates by
$$
b_j^\lambda\left((\theta_0,z')\right):=
b_j\left(\frac{r}{12\,C}\,
\lambda^{1/2-\delta} \,(\theta_0,z')  \right).$$
Thus $\left\{b_1^\lambda,b_2^\lambda\right\}$ is a partition of unity subordinate to the rescaled open cover
$\mathcal{B}^\lambda:=
\left\{
B_x\left(12\,C\,
\lambda^{\delta-1/2} 
\right), \overline{
B_x\left(6\,C\, 
\lambda^{\delta-1/2}
\right)
}
^c\right\}$.
For $y\in \mathrm{supp}\left(b_2^\lambda\right)$ and
$\lambda\gg 0$, 
\begin{equation}
\label{eqn:comp distances tau}
\mathrm{dist}_{X^\tau}\left(x, y\right)\ge
\frac{1}{2}\,\widetilde{\mathrm{dist}}_{X^\tau}\left(x, y\right)
\ge 3\,C\,\lambda^{\delta-1/2}.
\end{equation}

The asymptotics of $\Pi^\tau_{\chi,\lambda}(x_{1\lambda},x_{2\lambda})$
are given by 
(\ref{eqn:smoothed kernel reexp1}) with $(x_1,x_2)$ replaced by
$(x_{1\lambda},x_{2\lambda})$.

\begin{lem}
\label{lem:reduction in y}
%
Only a negligible contribution to the asymptotics of 
$\Pi^\tau_{\chi,\lambda}(x_{1\lambda},x_{2\lambda})$ is lost, if the integrand in (\ref{eqn:smoothed kernel reexp1}) is 
multiplied by $b_1^\lambda (y)$.
\end{lem}


\begin{proof}
For $y\in \mathrm{supp}\left(b_2^\lambda\right)$ and
$\lambda\gg 0$, 
\begin{equation}
\label{eqn:comp distances tau1}
\mathrm{dist}_{X^\tau}\left(x, y\right)\ge
\frac{1}{2}\,\widetilde{\mathrm{dist}}_{X^\tau}\left(x, y\right)
\ge 3\,C\,\lambda^{\delta-1/2}.
\end{equation}
%
Hence where $b_1^\lambda(y)= 0$ in view of (\ref{eqn:distance xxjlambdaest})
\begin{eqnarray*}
\mathrm{dist}_{X^\tau}\left(x_{1\lambda}, y\right)&\ge& 
\mathrm{dist}_{X^\tau}\left(x, y\right)
-\mathrm{dist}_{X^\tau}\left(x, x_{1\lambda}\right)\\
&\ge& 
3\,C\,\lambda^{\delta-1/2}-2\,C\,\lambda^{\delta-1/2}
=C\,\lambda^{\delta-1/2}.
\end{eqnarray*}
Recalling (\ref{eqn:Psilambdaprimafase}), with
$x_j$ replaced by $x_{j\lambda}$,
the claim follows arguing as in the proof of 
statement 1. (integrate by parts in $u$).
\end{proof}

We conclude from (\ref{eqn:smoothed kernel reexp1}) and
Lemma \ref{lem:reduction in y} that
\begin{eqnarray}
\label{eqn:smoothed kernel reexp2}
\lefteqn{\Pi^\tau_{\chi,\lambda}(x_{1\lambda},x_{2\lambda})}\\
&\sim&
\frac{\lambda^2}{\sqrt{2\,\pi}}\,\int_{1/(2\,D)}^{2\,D}\,
\mathrm{d}u\,\int_{1/(2\,D)}^{2\,D}\,\mathrm{d}v\,
\int_{-\infty}^{+\infty}\,\mathrm{d}\theta\,\int_{\mathbb{R}^{2\,d}}\,\mathrm{d}\mathbf{u}\,
\int_{-\epsilon}^{+\epsilon}\,\mathrm{d}t\nonumber\\
&&\left[e^{\imath\,\lambda\,\Psi^\tau\big(x_{1\lambda},x_{2\lambda};u,v,t,y(\theta,\mathbf{u})\big)}\,
b_1^\lambda \big(y(\theta,\mathbf{u})\big)\,
A_\lambda^\tau\big(x_{1\lambda},x_{2\lambda};u,v,t,y(\theta,\mathbf{u})\big)\,\mathcal{V}(\theta,\mathbf{u})
\right].\nonumber
\end{eqnarray}
On the domain of integration, $(\theta/\tau,\mathbf{u})$ ranges in
a shrinking ball of radius $O\left(\lambda^{\delta-1/2}\right)$
centered at the origin in $\mathbb{R}^{2d+1}$.

Next we show that integration in $\mathrm{d}t$ may be localized to
a shrinking neighbourhood of the origin.

Let us fix a constant $C_1>0$
such that for all $\lambda\gg 0$
and $y=y(\theta,\mathbf{u})$ on the support of $b_1^\lambda \big(y(\theta,\mathbf{u})\big)$ 
\begin{eqnarray}
\label{eqn:estimateC1xlambda}
\mathrm{dist}_{X^\tau}\left(x_{2\lambda},y\right )\le
\mathrm{dist}_{X_\tau}(x_{2\lambda},x)+\mathrm{dist}_{X^\tau}
(x,y)\le C_1\,\lambda^{\delta-\frac{1}{2}}.          
\end{eqnarray}
Let $\beta\in \mathcal{C}^\infty _0(\mathbb{R})$ be such that
$\beta(t)\equiv 1$ on $[-1,1]$, and define
$\beta^\lambda(t):=
\beta\left(\frac{1}{3\,C_1}\,\lambda^{\frac{1}{2}-\delta}\,t  \right)$.
Then $1-\beta^\lambda (t)=0$
if $|t|\le 3\,C_1\,\lambda^{\delta-\frac{1}{2}}$.

We have
 $$\Pi^\tau_{\chi,\lambda}(x_{1\lambda},x_{2\lambda})\sim
 \Pi^\tau_{\chi,\lambda}(x_{1\lambda},x_{2\lambda})'+
\Pi^\tau_{\chi,\lambda}(x_{1\lambda},x_{2\lambda})'',$$
where $\Pi^\tau_{\chi,\lambda}(x_{1\lambda},x_{2\lambda})'$
and $\Pi^\tau_{\chi,\lambda}(x_{1\lambda},x_{2\lambda})''$ are
as in (\ref{eqn:smoothed kernel reexp2}), but with the 
integrand multiplied, respectively, by 
$\beta^\lambda (t)$ and $1-\beta^\lambda (t)$.

\begin{lem}
\label{eqn:scaling in t}
$\Pi^\tau_{\chi,\lambda}(x_{1\lambda},x_{2\lambda})''=
O\left(\lambda^{-\infty}\right)$.
\end{lem}

\begin{proof}
Where $1-\beta^\lambda (t)\neq 0$, 
we have $|t|> 3\,C_1\,\lambda^{\delta-\frac{1}{2}}$.
In view of (\ref{eqn:normsT}), if
$\epsilon $ is small enough and 
and $|\epsilon|>|t|> 3\,C_1\,\lambda^{\delta-\frac{1}{2}}$ then
$$
\mathrm{dist}_{X^\tau}\left(\Gamma^\tau_{-t}(y),y\right )
\ge \frac{t}{2}\ge \frac{3}{2}\,C_1\,\lambda^{\delta-\frac{1}{2}}.
$$
Therefore, by (\ref{eqn:estimateC1xlambda}) on the support of the integrand of
$\Pi^\tau_{\chi,\lambda}(x_{1\lambda},x_{2\lambda})''$ we have
\begin{eqnarray*}
\mathrm{dist}_{X^\tau}\left(\Gamma^\tau_{-t}(y),x_{2\lambda} \right )
&\ge&\mathrm{dist}_{X^\tau}\left(\Gamma^\tau_{-t}(y),y\right )-
\mathrm{dist}_{X^\tau}\left(x_{2\lambda},y\right )\ge\frac{C_1}{2}\,\lambda^{\delta-\frac{1}{2}}.
\end{eqnarray*}
The claim follows again by the argument in the proof of statement 1,
by iterated integration by parts in $v$.
\end{proof}

With the rescalings 
$\theta\rightsquigarrow \theta/\sqrt{\lambda}$,
$\mathbf{u}\rightsquigarrow \mathbf{u}/\sqrt{\lambda}$,
$t\rightsquigarrow t/\sqrt{\lambda}$,
(\ref{eqn:smoothed kernel reexp2}) becomes

\begin{eqnarray}
\label{eqn:smoothed kernel reexp3}
\lefteqn{\Pi^\tau_{\chi,\lambda}(x_{1\lambda},x_{2\lambda})}\\
&\sim&
\frac{\lambda^{2-d}}{\sqrt{2\,\pi}}\,\int_{1/(2\,D)}^{2\,D}\,\,
\mathrm{d}u\,\int_{1/(2\,D)}^{2\,D}\,\,\mathrm{d}v\,
\int_{-\infty}^{+\infty}\,\mathrm{d}\theta\,\int_{\mathbb{R}^{2\,d}}\,\mathrm{d}\mathbf{u}\,
\int_{-\infty}^{+\infty}\,\mathrm{d}t\nonumber\\
&&\left[e^{\imath\,\lambda\,
\Psi^\tau_\lambda\big(x_{1\lambda},x_{2\lambda};u,v,t,\theta,\mathbf{u}\big)}\,
B_\lambda^\tau\big(x_{1\lambda},x_{2\lambda};u,v,t,\theta,\mathbf{u}
\big)
\right],\nonumber
\end{eqnarray}
where, in view of (\ref{eqn:Psilambdaprimafase}),
\begin{eqnarray}
\label{eqn:Psilambdaresc}
\lefteqn{\Psi^\tau_\lambda\big(x_{1\lambda},x_{2\lambda};u,v,t,\theta,\mathbf{u}\big)}\\
&:=&\Psi^\tau\left(x_{1\lambda},x_{2\lambda};
u,v,\frac{t}{\sqrt{\lambda}},
y\left(\frac{1}{\sqrt{\lambda}}\,(\theta,\mathbf{u})\right)\right),
\nonumber\\
&=&
u\,\psi^\tau\left(x_{1\lambda},
y\left(\frac{1}{\sqrt{\lambda}}\,(\theta,\mathbf{u})\right)\right)
+v\,\psi^\tau \left(
\Gamma^\tau_{-t/\sqrt{\lambda}}
\left(
y\left(\frac{1}{\sqrt{\lambda}}\,(\theta,\mathbf{u})
\right)
\right),x_{2\lambda} \right )\nonumber\\
&&
-\frac{t}{\sqrt{\lambda}},\nonumber
\\
\lefteqn{B_\lambda^\tau\big(x_{1\lambda},x_{2\lambda};u,v,t,\theta,\mathbf{u})
\big)}
\\
&:=&b_1^\lambda \left(y\left(\frac{1}{\sqrt{\lambda}}\,(\theta,\mathbf{u}\right)\right)\,
A_\lambda^\tau\left(x_{1\lambda},x_{2\lambda};
u,v,\frac{t}{\sqrt{\lambda}},
y\left(\frac{1}{\sqrt{\lambda}}\,(\theta,\mathbf{u})\right)\right)\nonumber\\
&&\cdot 
\mathcal{V}\left(\frac{1}{\sqrt{\lambda}}\,(\theta,\mathbf{u})\right).\nonumber
\end{eqnarray}
Integration in $(t,\theta,\mathbf{v})$ in (\ref{eqn:smoothed kernel reexp3}) 
is  over an expanding ball centered
at the origin and radius $O\left(\lambda^\delta\right)$.

We now make explicit the dependence of $\Psi^\tau_\lambda$
in (\ref{eqn:Psilambdaresc}) on the rescaled
variables.
By Proposition \ref{prop:approx expr psitau},
\begin{eqnarray}
\label{eqn:1stpsitau}
\lefteqn{
\imath\,\psi^\tau\left(x_{1\lambda},y\left(\frac{1}{\sqrt{\lambda}}\,(\theta,\mathbf{u})
\right)\right)}\\
&=&\imath\,\frac{\theta_1-\theta}{\sqrt{\lambda}}
-\frac{1}{4\,\tau^2\,\lambda}\,\left(\theta_1-\theta\right)^2
+\frac{\psi_2^{\omega_x}\left(\mathbf{v}_1,\mathbf{u}\right)}{\lambda}
+R_3\left(
\frac{\theta_1}{\sqrt{\lambda}},\frac{\mathbf{v}_1}{\sqrt{\lambda}},
\frac{\theta}{\sqrt{\lambda}},\frac{\mathbf{u}}{\sqrt{\lambda}}   \right).
\nonumber
\end{eqnarray}
Furthermore, by Corollary \ref{cor:integral curve geodlocal}
and (\ref{eqn:geodesflow rescaled}),
\begin{eqnarray}
\label{eqn:Gammatuay}
\lefteqn{\Gamma^\tau_{-t/\sqrt{\lambda}}
\left(
y\left(\frac{1}{\sqrt{\lambda}}\,(\theta,\mathbf{u})
\right)
\right)}\\
&=&x+\left(\frac{1}{\sqrt{\lambda}}\left(\theta +\tau\,t\right)
-\frac{t}{\lambda}\,\left(\frac{1}{2\,\tau^2}\,
a_x\,\theta+\langle \mathbf{A}_x,\mathbf{u}\rangle  \right)+
R_3\left(\frac{\tau\,t}{\sqrt{\lambda}},\frac{\theta}{\sqrt{\lambda}},\frac{\mathbf{v}}{\sqrt{\lambda}}\right),\right.\nonumber\\
&&\left.\frac{\mathbf{u}}{\sqrt{\lambda}} 
+ \mathbf{R}_2\left(\frac{\tau\,t}{\sqrt{\lambda}},
\frac{\theta}{\sqrt{\lambda}},\frac{\mathbf{u}}{\sqrt{\lambda}}\right)\right).
\nonumber
\end{eqnarray}
Therefore, again by Proposition \ref{prop:approx expr psitau},
\begin{eqnarray}
\label{eqn:2ndpsitau}
\lefteqn{
\imath\,\psi^\tau \left(\Gamma^\tau_{-t/\sqrt{\lambda}}
\left(
y\left(\frac{1}{\sqrt{\lambda}}\,(\theta,\mathbf{u})
\right)
\right),x_{2\lambda} \right )}\\
&=&\frac{\imath}{\sqrt{\lambda}}\,(\theta+\tau\,t-\theta_2)
-\frac{\imath}{\lambda}\,t\,
\left(\frac{1}{2\,\tau^2}\,
a_x\,\theta+\langle \mathbf{A}_x,\mathbf{u}\rangle  \right)
\nonumber\\
&&-\frac{1}{4\,\tau^2\,\lambda}\,(\theta+\tau\,t-\theta_2)^2
+\frac{\psi_2^{\omega_x}\left(\mathbf{u},\mathbf{v}_2\right)}{\lambda}
+R_3\left(\frac{\tau\,t}{\sqrt{\lambda}},
\frac{\theta_2}{\sqrt{\lambda}},\frac{\mathbf{v}_2}{\sqrt{\lambda}},
\frac{\theta}{\sqrt{\lambda}},\frac{\mathbf{u}}{\sqrt{\lambda}}   \right).
\nonumber
\end{eqnarray}
We finally obtain
\begin{eqnarray}
\label{eqn:Phase expansion lambda}
\lefteqn{
\imath\,\lambda\,
\Psi^\tau_\lambda\big(x_{1\lambda},x_{2\lambda};u,v,t,\theta,\mathbf{u}\big)
}\\
&=&\imath\,\sqrt{\lambda}\,\Upsilon^\tau (t,v,\theta,u)
+S(\mathbf{u},t,v,\theta,u)+\lambda\,R_3\left(
\frac{\tau\,t}{\sqrt{\lambda}},
\frac{\theta_j}{\sqrt{\lambda}},\frac{\mathbf{v}_j}{\sqrt{\lambda}},
\frac{\theta}{\sqrt{\lambda}},\frac{\mathbf{u}}{\sqrt{\lambda}}   \right),
\nonumber
\end{eqnarray}
where
\begin{eqnarray}
\label{eqn:Upsilondfn}
\Upsilon^\tau (t,v,\theta,u)
&=&\Upsilon^\tau (\theta_j,\mathbf{v}_j; t,v,\theta,u) \\  
&:=&u\,(\theta_1-\theta)+v\,(\theta+\tau\,t-\theta_2)-t,
\nonumber
\end{eqnarray}
\begin{eqnarray}
\label{defn:S12}
\lefteqn{S(\mathbf{u},t,v,\theta,u)=S_{\theta_j,\mathbf{v}_j}(\mathbf{u},t,v,\theta,u)}\\
&:=&
-\imath\,v\,t\,\left(\frac{1}{2\,\tau^2}\,
a_x\,\,\theta+\langle \mathbf{A}_x,\mathbf{u}\rangle  \right)
-\frac{1}{4\,\tau^2}\,u\,\left(\theta_1-\theta\right)^2
-\frac{1}{4\,\tau^2}\,v\,(\theta+\tau\,t-\theta_2)^2\nonumber\\
&& +
u\,\psi_2^{\omega_x}\left(\mathbf{v}_1,\mathbf{u}\right)
+v\,\psi_2^{\omega_x}\left(\mathbf{u},\mathbf{v}_2\right),
\nonumber
\end{eqnarray}
and $R_3$ vanishes to third order
at the origin. Notice that $S(\mathbf{u},t,v,\theta,u)$ is homogenous of degree $2$ in the rescaled variables.

If we set $\theta':=\theta-\theta_1$, 
$t':=\theta+\tau\,t-\theta_2$, and
$\mathbf{u}':=\mathbf{u}-\mathbf{v}_1$, then 
\begin{equation}
\label{eqn:key real part estimate}
\Re\big(S(\mathbf{u},t,v,\theta,u)  \big)\le 
-b_0\,\|(t',\theta',\mathbf{u}')\|^2
\end{equation}
for some $b_0>0$.

In light of (\ref{eqn:Phase expansion lambda}),
we can rewrite (\ref{eqn:smoothed kernel reexp3}) as follows:
\begin{eqnarray}
\label{eqn:smoothed kernel reexp40}
\Pi^\tau_{\chi,\lambda}(x_{1\lambda},x_{2\lambda})
\sim
\frac{\lambda^{2-d}}{\sqrt{2\,\pi}}\,
\int_{\mathbb{R}^{2\,d-2}}\,I_\lambda(\mathbf{u})\,\mathrm{d}\mathbf{u}
\end{eqnarray}
where 
\begin{eqnarray}
\label{eqn:defnIlambda}
I_\lambda(\mathbf{u})&:=&\int_{-\infty}^{+\infty}\,\mathrm{d}t\,
\int_{1/(2\,D)}^{2\,D}\,
\mathrm{d}v\,\int_{-\infty}^{+\infty}\,\mathrm{d}\theta\,
\int_{1/(2\,D)}^{2\,D}\,\,\mathrm{d}u
\\
&&\left[e^{\imath\,\sqrt{\lambda}\,\Upsilon^\tau (t,v,\theta,u)}\,
\tilde{B}_\lambda^\tau\big(x_{1\lambda},x_{2\lambda};\mathbf{u},t,v,\theta,u
\big)
\right]\nonumber
\end{eqnarray}
and, using (\ref{eqn:semiclasstau}),
(\ref{eqn:Alambda primaampiezza}), (\ref{eqn:vol form at 0}), 
and (\ref{eqn:Phase expansion lambda})
\begin{eqnarray}
\label{eqn:tildeB0}
\lefteqn{\tilde{B}_\lambda^\tau\big(x_{1\lambda},x_{2\lambda};\mathbf{u},t,v,\theta,u
\big)}\\
&:=&e^{S(\mathbf{u},t,v,\theta,u)+\lambda\,R_3\left(
\frac{\tau\,t}{\sqrt{\lambda}},
\frac{\theta_j}{\sqrt{\lambda}},\frac{\mathbf{v}_j}{\sqrt{\lambda}},
\frac{\theta}{\sqrt{\lambda}},\frac{\mathbf{u}}{\sqrt{\lambda}}   \right)}\,B_\lambda^\tau\big(x_{1\lambda},x_{2\lambda};\mathbf{u},t,v,\theta,u
\big).\nonumber
\end{eqnarray}

For every $N=1,2,\ldots$,
Taylor expansion in the rescaled variables yields 
for the third order remaninder $R_3$ in the exponent in (\ref{eqn:tildeB0})
\begin{eqnarray}
\label{eqn:R3expanded}
\lefteqn{\lambda\,R_3\left(
\frac{\tau\,t}{\sqrt{\lambda}},
\frac{\theta_j}{\sqrt{\lambda}},\frac{\mathbf{v}_j}{\sqrt{\lambda}},
\frac{\theta}{\sqrt{\lambda}},\frac{\mathbf{u}}{\sqrt{\lambda}}   \right)  }\\
&=&\sum_{s=0}^N\,\lambda^{1-\frac{3+s}{2}}\, P_{3+s}(\theta_j,\mathbf{v}_j,
\theta,\mathbf{u})+\lambda\,R_{4+N}\left(
\frac{\tau\,t}{\sqrt{\lambda}},
\frac{\theta_j}{\sqrt{\lambda}},\frac{\mathbf{v}_j}{\sqrt{\lambda}},
\frac{\theta}{\sqrt{\lambda}},\frac{\mathbf{u}}{\sqrt{\lambda}}  \right)  ,
\nonumber
\end{eqnarray}
where $P_l$ denotes a homogeneous polynomial of degree $l$ in
the argument.
On the domain of integration, the latter summand is bounded
in absolute value by $C_N\,\lambda^{1-\left(\frac{1}{2}-\delta\right)\,(4+N)}$
for some constant $C_N>0$. Passing to the exponential, we therefore
obtain an asymptotic expansion
\begin{eqnarray}
\label{eqn:taylor expR3}
\lefteqn{
e^{\lambda\,R_3\left(
\frac{\tau\,t}{\sqrt{\lambda}},
\frac{\theta_j}{\sqrt{\lambda}},\frac{\mathbf{v}_j}{\sqrt{\lambda}},
\frac{\theta}{\sqrt{\lambda}},\frac{\mathbf{u}}{\sqrt{\lambda}}   \right)
}=\sum_{k=0}^{+\infty}\frac{\lambda^k}{k!}\,
R_3\left(
\frac{\tau\,t}{\sqrt{\lambda}},
\frac{\theta_j}{\sqrt{\lambda}},\frac{\mathbf{v}_j}{\sqrt{\lambda}},
\frac{\theta}{\sqrt{\lambda}},\frac{\mathbf{u}}{\sqrt{\lambda}}   \right)^k
}\\
&\sim&\sum_{k=0}^{+\infty}\frac{1}{k!}\,
\left[ \sum_{s=0}^N\,\lambda^{1-\frac{3+s}{2}}\, P_{3+s}(\theta_j,\mathbf{v}_j,
\theta,\mathbf{u})\right]^k=
\sum_{k=0}^{+\infty}\,\lambda^{-\frac{k}{2}}\,Q_k
(\theta_j,\mathbf{v}_j,
\theta,\mathbf{u}),
\nonumber
\end{eqnarray}
where $Q_k$ is a polynomial of degree $\le 3\,k$.

On the other hand, the asymptotic expansion
(\ref{eqn:semiclasstau}) yields
\begin{eqnarray}
\label{eqn:semiclassstaulambda}
\lefteqn{
s^\tau(x_{1\lambda},y,\lambda\,u)\sim
\sum_{k\ge 0}\,(\lambda\,
u)^{d-1-k}\,s^\tau_j(x_{1\lambda},y)}\\
&=&(\lambda\,u)^{d-1}\,s^\tau_0(x,x)
+\sum_{k+l\ge 1}\lambda^{d-1-k-\frac{l}{2}}\,
u^{d-1-k}\,S^\tau_{k,l}(\theta_1,\mathbf{v}_1,
\theta,\mathbf{u}),
\nonumber
\end{eqnarray}
where $S^\tau_{k,l}$ is homogenous of degree $l$.
Similarly, the asymptotic expansion in Lemma \ref{lem:calPtaut}
yields
\begin{eqnarray}
\label{eqn:semiclassrtaulambda}
\lefteqn{
r^\tau_t\left(y,x_{2\,\lambda},\,\lambda\,v\right)\sim
\sum_{k\ge 0}\,(\lambda\,
u)^{d-1-k}\,r^\tau_{t\,j}(y,x_{2\,\lambda})}\\
&=&(\lambda\,u)^{d-1}\,s^\tau_{0}(x,x)
+\sum_{k+l\ge 1}\lambda^{d-1-k-\frac{l}{2}}\,
u^{d-1-k}\,R^\tau_{k,l}(\theta_1,\mathbf{v}_1,
\theta,\mathbf{u}),
\nonumber
\end{eqnarray}
where again $R^\tau_{k,l}$ is homogeneous of degree $l$,
and we have used that $r^\tau_{0,\,0}(x,x)=s^\tau_0(x,x)$.

Multiplying (\ref{eqn:taylor expR3}), (\ref{eqn:semiclassstaulambda}),
(\ref{eqn:semiclassrtaulambda}), and the Taylor expansion
of $\chi$ and $\mathcal{V}$ at the origin, we obtain an asymptotic
expansion 
\begin{eqnarray}
\label{eqn:tildeBexpanded}
\lefteqn{\tilde{B}_\lambda^\tau
\big(x_{1\lambda},x_{2\lambda};\mathbf{u},t,v,\theta,u
\big)}\\
&\sim&\lambda^{2d-2}\,\frac{2^{d-1}}{\tau}\, \chi (0)\,
e^{S(\mathbf{u},t,v,\theta,u)}\,
u^{d-1}\,v^{d-1}\,s^\tau_0(x,x)^2\cdot
\beta\left(\lambda^{-\delta}\,\left( \mathbf{u},
t,\theta  \right)  \right)\nonumber\\
&&\cdot \left[1+\sum_{k\ge 1}\lambda^{-k/2}\,
B_k(x;\mathbf{u},\theta_j,\mathbf{v}_j,t,v,\theta,u)     \right],\nonumber
\end{eqnarray}
where $B_k(x;\cdot)$ is a polynomial of degree $\le 3\,k$
in $(\theta_j,\mathbf{v}_j,\mathbf{u},t,\theta)$, while 
$\beta$ is compactly supported and identically equal to $1$ in a
suitable neighborhood of the
origin. The latter is indeed an asymptotic expansion for
$\delta\in (0,1/6)$.
Furthermore, fractional powers of $\lambda$ arise
from Taylor expansion in $(\theta_j,\mathbf{v}_j,\mathbf{u})$, while 
the asymptotic expansion for the amplitude in the Szeg\"{o} kernel
parametrix is by descending integer powers. Hence $P_j$ will be even
for $j$ even (corresponding to integer powers of $\lambda$), 
and odd for $j$ odd (corresponding to fractional powers).

Inserting (\ref{eqn:tildeBexpanded}) in (\ref{eqn:defnIlambda}),
we obtain an asymptotic expansion for the integrand which,
in view of (\ref{eqn:key real part estimate}) or the previous 
remark on the domain of integration, can be integrated term by term.
Each term
is an oscillatory integral with phase (\ref{eqn:Upsilondfn})
in the parameters $(t,v,\theta,u)$, and depending
parametrically on the other parameters.

Now we remark that the asymptotics of (\ref{eqn:smoothed kernel reexp40}) are
unaltered, if integration in $(t,\theta)$ is restricted to a suitable
compact set. In fact,
since $\partial_u\Upsilon^\tau=\theta_1-\theta$, 
$\partial_v\Upsilon^\tau
=\theta+\tau\,t-\theta_2$, integration by parts in
$(u,v)$ implies the following.

\begin{lem}
\label{lem:compact integr theta t} 
Only a negligible contribution to the asymptotics
of (\ref{eqn:smoothed kernel reexp3}) is lost, if the amplitude
(\ref{eqn:tildeB0}) is multiplied by a compactly supported cut-off
function in $(\theta, t)$, identically equal to $1$ near 
$\big(\theta_1,(\theta_2-\theta_1)/\tau\big)$.
\end{lem}

We shall leave the latter cut-off implicit in the following, and 
simply assume henceforth that integration in $(\theta,t)$ is over
a compact neighborhood of $\big(\theta_1,(\theta_2-\theta_1)/\tau\big)$.

The proof of the following is omitted.

\begin{lem}
\label{lem:critical point non deg}
$\Upsilon^\tau$ in (\ref{eqn:Upsilondfn}) has a unique
stationary point $P_s=(t_s,v_s,\theta_s,u_s)$, given by
$$
P_s^\tau=\left(
\frac{\theta_2-\theta_1}{\tau},\,\frac{1}{\tau},\,\theta_1,\,\frac{1}{\tau}
\right).
$$
The critical point is non-degenerate. The 
Hessian matrix and its inverse at the critical point are
$$
H(\Upsilon^\tau)_{P_s^\tau}=
\begin{pmatrix}
0&\tau&0&0\\
\tau& 0&1&0   \\
0 &1 &0&-1  \\
0&0&-1&0
\end{pmatrix},\quad
H(\Upsilon^\tau)_{P_s^\tau}^{-1}=
\begin{pmatrix}
0&1/\tau&0&1/\tau\\
1/\tau& 0&0&0   \\
0 &0 &0&-1  \\
1/\tau&0&-1&0
\end{pmatrix}.
$$
The Hessian determinant, the Hessian signature, and
the critical value are 
$$\det H(\Upsilon^\tau)_{P_s^\tau}=\tau^2,\quad
\mathrm{sgn}H(\Upsilon^\tau)_{P_s^\tau}= 0,\quad
\Upsilon^\tau (P_s^\tau)=\frac{\theta_1-\theta_2}{\tau}.
$$
The third order remainder at $P_s$ is zero.
\end{lem}

Let us set 
\begin{equation}
\label{eqn:forma di LLPFS}
L:=\frac{1}{\tau}\,\left( \frac{\partial^2}{\partial t\partial u}
 +\frac{\partial^2}{\partial t\partial v}  \right)
 -\frac{\partial^2}{\partial\theta\,\partial u},
\end{equation}
and apply the Lemma of Stationary Phase to (\ref{eqn:defnIlambda})
in correspondence to the $k$-th summand 
in (\ref{eqn:tildeBexpanded}); we
obtain an asymptotic expansion whose $r$-th summand 
(for $r=0,1,2,\ldots$)
is a multiple of
\begin{eqnarray}
\label{eqn:jlthsummand}
\lefteqn{ (4\,\pi^2/\tau)\cdot e^{\imath\,\sqrt{\lambda}\,\frac{\theta_1-\theta_2}{\tau}}
\cdot \lambda^{2d-3-(r+k)/2}}\\
&&\cdot
\left.L^r\left(
e^{S(\mathbf{u},t,v,\theta,u)}\,
u^{d-1}\,v^{d-1}\,P_k(x;\theta_j,\mathbf{v}_j,\mathbf{u},t,v,\theta,u)\right)\right|_{P_s},\nonumber
\end{eqnarray}
where the dependence on $(\theta_j,\mathbf{v}_j)$ of the exponent
is left implicit for
brevity.
In view of (\ref{defn:S12}), with the notation of Definition
\ref{defn:newPsi2inv}
the value of $S(\mathbf{u},t,v,\theta,u)$ at the critical point 
is
\begin{eqnarray}
\label{eqn:Sc}
S_c(\theta_j,\mathbf{v}_j,\mathbf{u})&:=&S_{\theta_j,\mathbf{v}_j}(\mathbf{u},t_s,v_s,\theta_s,u_s)\\
&=&
\imath\cdot
\left[
\frac{(\theta_1-\theta_2)\,\theta_1}{2\,\tau^4}\,a_x+
\frac{\theta_1-\theta_2}{\tau^2}\cdot\langle 
\mathbf{A}_x,\mathbf{u}\rangle \right]
\nonumber\\
&& +
\frac{1}{\tau}\,\psi_2^{\omega_x}\left(\mathbf{v}_1,\mathbf{u}\right)
+\frac{1}{\tau}\,\psi_2^{\omega_x}\left(\mathbf{u},\mathbf{v}_2\right)
\nonumber\\
&=&
\imath\,\frac{(\theta_1-\theta_2)\,\theta_1}{2\,\tau^4}\,a_x
-\frac{\imath}{\tau}\cdot\omega_x\big(J_x\big(\mathbf{a}_x(\theta_1,\theta_2)\big),\mathbf{u}\big)  
\nonumber\\
&& +
\frac{1}{\tau}\,\psi_2^{\omega_x}\left(\mathbf{v}_1,\mathbf{u}\right)
+\frac{1}{\tau}\,\psi_2^{\omega_x}\left(\mathbf{u},\mathbf{v}_2\right).
\nonumber
\end{eqnarray}
%
Thus
\begin{eqnarray}
\label{eqn:Sc1}
\lefteqn{S_c(\theta_j,\mathbf{v}_j,\mathbf{u})}\\
&=&
\imath\,\frac{(\theta_1-\theta_2)\,\theta_1}{2\,\tau^4}\,a_x
\nonumber
\\
&&
+\frac{1}{\tau}\left[ -\imath\,\omega_x\big(
\mathbf{v}_1-\mathbf{v}_2+
J_x\big(\mathbf{a}_x(\theta_1,\theta_2)\big),\mathbf{u}\big) -\frac{1}{2}\,\left(\|\mathbf{v}_1\|^2+\|\mathbf{v}_2\|^2  \right) \right.\nonumber\\
&&\left. -\|\mathbf{u}\|^2 +\langle\mathbf{u},\mathbf{v}_1+\mathbf{v}_2
\rangle   \right]\nonumber\\
&=&
\imath\,\frac{(\theta_1-\theta_2)\,\theta_1}{2\,\tau^4}\,a_x
\nonumber
\\
&&
+
\frac{1}{\tau}\left[ 
 -\frac{1}{2}\,\left(\|\mathbf{v}_1\|^2+\|\mathbf{v}_2\|^2  \right)+\frac{1}{4}\,\|\mathbf{v}_1+\mathbf{v}_2\|^2\right.\nonumber\\
&&\left. -\imath\,\omega_x\big(
\mathbf{v}_1-\mathbf{v}_2+
J_x\big(\mathbf{a}_x(\theta_1,\theta_2)\big),\mathbf{u}\big)  -\left\|\mathbf{u}-\frac{1}{2}\,(\mathbf{v}_1+\mathbf{v}_2)
\right\|^2\right].\nonumber
\end{eqnarray}
In particular, $S_c(\theta_j,\mathbf{v}_j,\mathbf{u})$ is homogenous of degree $2$.

It follows from (\ref{defn:S12}) 
and (\ref{eqn:forma di LLPFS})  that
(\ref{eqn:jlthsummand}) is a linear combination 
of terms of the form
$\lambda^{2d-3-k/2}\,P_k(x;
\theta_j,\mathbf{v}_j,\mathbf{u})\,e^{S_c(\theta_j,\mathbf{v}_j,\mathbf{u})}$, where $P_k(x;\cdot)$ is a polynomial of degree 
$\le 3\,k$.
Furthermore, using that $S(\mathbf{u},t,v,\theta,u)$ is homogenous of degree $2$ in the rescaled variables, the explicit expression
(\ref{eqn:forma di LLPFS}) of $L$, and the linear dependence of 
$\theta_s$ and $t_s$ on $(\theta_1,\theta_2)$, 
one verifies that $P_k(x;\cdot)$ has parity $k$.

Putting all these asymptotic expansions together, we obtain an
asymptotic expansion for $I_\lambda(\mathbf{u})$ of the form
\begin{eqnarray}
\label{eqn:Ilambdavexpanded}
I_\lambda(\mathbf{u})&\sim&
e^{\imath\,\sqrt{\lambda}\,\frac{\theta_1-\theta_2}{\tau}}
\cdot\frac{\lambda^{2d-3}}{(2\,\pi^2\,\tau^2)^{d-1}}
\cdot \chi (0)\,
e^{S_c(\theta_j,\mathbf{v}_j,\mathbf{u})}\nonumber\\
&&\cdot \beta_1\left(\lambda^{-\delta}\, \mathbf{u},
 \right)\left[1+\sum_{k\ge 1}\lambda^{-k/2}\,F_k(x;
\theta_j,\mathbf{v}_j,\mathbf{u})\right]
\end{eqnarray}
where $F_k(x;\cdot)$ is a polynomial of degree $\le 3\,k$
and parity $k$, and $\beta_1$
is an appropriate cut-off function identically equal to one near the origin.

The asymptotic expansion
(\ref{eqn:Ilambdavexpanded}) 
may be integrated term by term.
In view of the rapidly decreasing exponential
$e^{S_c(\theta_j,\mathbf{v}_j,\mathbf{u})}$, we obtain 
\begin{eqnarray}
\label{eqn:smoothed kernel reexp4}
\lefteqn{
\Pi^\tau_{\chi,\lambda}(x_{1\lambda},x_{2\lambda})
=\frac{\lambda^{2-d}}{\sqrt{2\,\pi}}\,
\int_{\mathbb{R}^{2\,d-2}}\,I_\lambda(\mathbf{u})\,\mathrm{d}\mathbf{u}}\\
&\sim&e^{\imath\,\sqrt{\lambda}\,\frac{\theta_1-\theta_2}{\tau}}
\cdot
\frac{\lambda^{d-1}}{\sqrt{2\,\pi}}\cdot 
\frac{1}{(\sqrt{2}\,\pi\,\tau)^{2\,(d-1)}}
\cdot \chi (0)\,
\int_{\mathbb{R}^{2\,d-2}}\,e^{S_c(\theta_j,\mathbf{v}_j,\mathbf{u})}
\mathrm{d}\mathbf{u}
\nonumber\\
&&\cdot \left[1+\sum_{k\ge 1}\lambda^{-k/2}\,
\int_{\mathbb{R}^{2\,d-2}}\,e^{S_c(\theta_j,\mathbf{v}_j,\mathbf{u})}
F_k(x;
\theta_j,\mathbf{v}_j,\mathbf{u})\,
\mathrm{d}\mathbf{u}
\right].\nonumber
\end{eqnarray}

We compute the leading order term using (\ref{eqn:Sc1}). With the
change of variable 
$$
\mathbf{w}:=\mathbf{u}-\frac{1}{2}\,(\mathbf{v}_1+\mathbf{v}_2),
$$
we have
\begin{eqnarray}
\label{eqn:Sc2}
\lefteqn{S_c(\theta_j,\mathbf{v}_j,\mathbf{u})}\\
&=&
\imath\,\frac{(\theta_1-\theta_2)\,\theta_1}{2\,\tau^4}\,a_x
\nonumber
\\
&&
+
\frac{1}{\tau}\left[ -\imath\,\omega_x\left(
\mathbf{v}_1-\mathbf{v}_2+
J_x(\mathbf{a}),\frac{1}{2}\,(\mathbf{v}_1+\mathbf{v}_2)\right) -\frac{1}{2}\,\left(\|\mathbf{v}_1\|^2+\|\mathbf{v}_2\|^2  \right) +\frac{1}{4}\,\|\mathbf{v}_1+\mathbf{v}_2\|^2\right.\nonumber\\
&&\left. -\imath\,\omega_x\left(
\mathbf{v}_1-\mathbf{v}_2+
J_x\big(\mathbf{a}_x(\theta_1,\theta_2)\big),\mathbf{w}\right) -\left\|\mathbf{w}\right\|^2\right]\nonumber\\
&=&
\imath\,\frac{(\theta_1-\theta_2)\,\theta_1}{2\,\tau^4}\,a_x
\nonumber
\\
&&
+
\frac{1}{\tau}\left[ -\imath\,\omega_x\left(
\mathbf{v}_1,\mathbf{v}_2\right)-\imath\,\omega_x\left(
J_x\big(\mathbf{a}_x(\theta_1,\theta_2)\big),\frac{1}{2}\,(\mathbf{v}_1+\mathbf{v}_2)\right) 
-\frac{1}{2}\,\left(\|\mathbf{v}_1\|^2+\|\mathbf{v}_2\|^2  \right) 
\right.\nonumber\\
&&\left. +\frac{1}{4}\,\|\mathbf{v}_1+\mathbf{v}_2\|^2-\imath\,\omega_x\left(
\mathbf{v}_1-\mathbf{v}_2+
J_x\big(\mathbf{a}_x(\theta_1,\theta_2)\big),\mathbf{w}\right) -\left\|\mathbf{w}\right\|^2\right].\nonumber
\end{eqnarray}

Let us set $\mathbf{v}_j':=\frac{1}{\sqrt{\tau}}\,
\mathbf{v}_j$, $\mathbf{a}':=\frac{1}{\sqrt{\tau}}\,
\mathbf{a}_x(\theta_1,\theta_2)$. With the 
further change of variable 
$\mathbf{w}=\frac{1}{\sqrt{2}}\,\mathbf{r}$, we obtain
\begin{eqnarray}
\label{eqn:gaussian int 1}
\lefteqn{\int_{\mathbb{R}^{2\,d-2}}\,
e^{\frac{1}{\tau}\,\big[-\imath\,\omega_x\left(
\mathbf{v}_1-\mathbf{v}_2+
J_x\big(\mathbf{a}_x(\theta_1,\theta_2)\big),\mathbf{w}\right) -\left\|\mathbf{w}\right\|^2\big]}
\mathrm{d}\mathbf{w}}\\
&=&\frac{1}{2^{d-1}}\,\int_{\mathbb{R}^{2\,d-2}}\,
e^{\frac{1}{\tau}\,\left[-\imath\,\omega_x\left(\frac{1}{\sqrt{2}}\big(
\mathbf{v}_1-\mathbf{v}_2+
J_x\big(\mathbf{a}_x(\theta_1,\theta_2)\big)\big),\mathbf{r}\right) -\frac{1}{2}\,\left\|\mathbf{r}\right\|^2\right]}
\mathrm{d}\mathbf{r}\nonumber\\
&=&\frac{1}{2^{d-1}}\,\int_{\mathbb{R}^{2\,d-2}}\,
e^{-\imath\,\omega_x\left(\frac{1}{\sqrt{2\,\tau}}\big(
\mathbf{v}_1'-\mathbf{v}_2'+
J_x(\mathbf{a}')\big),\frac{\mathbf{r}}{\sqrt{\tau}}\right) -\frac{1}{2}\,\left\|\frac{\mathbf{r}}{\sqrt{\tau}}\right\|^2}
\mathrm{d}\mathbf{r}\nonumber\\
&=&\left(\frac{\tau}{2}\right)^{d-1}\,\int_{\mathbb{R}^{2\,d-2}}\,
e^{-\imath\,\omega_x\left(\frac{1}{\sqrt{2\,\tau}}\big(
\mathbf{v}_1'-\mathbf{v}_2'+
J_x(\mathbf{a}')\big),\mathbf{s}\right) 
-\frac{1}{2}\,\left\|\mathbf{s}\right\|^2}
\mathrm{d}\mathbf{s}\nonumber\\
&=&\left(\frac{\tau}{2}\right)^{d-1}\,(2\,\pi)^{d-1}
e^{
-\frac{1}{4\,\tau}\,\left\|\mathbf{v}_1-\mathbf{v}_2+
J_x\big(\mathbf{a}_x(\theta_1,\theta_2)\big)\right\|^2}\nonumber\\
&=&\left(\tau\,\pi\right)^{d-1}\cdot 
e^{
-\frac{1}{4\,\tau}\,\left\|\mathbf{v}_1-\mathbf{v}_2\right\|^2
-\frac{1}{4\,\tau}\,\left\| J_x\big(\mathbf{a}_x(\theta_1,\theta_2)\big)\right\|^2
-\frac{1}{2\,\tau}\,\left\langle\mathbf{v}_1-\mathbf{v}_2,
J_x\big(\mathbf{a}_x(\theta_1,\theta_2)\big)\right\rangle }.\nonumber
\end{eqnarray}

We have
\begin{eqnarray}
\lefteqn{-\frac{1}{2\,\tau}\,
\left(\|\mathbf{v}_1\|^2+\|\mathbf{v}_2\|^2  \right) 
+\frac{1}{4\,\tau}\,\|\mathbf{v}_1+\mathbf{v}_2\|^2
-\frac{1}{4\,\tau}\,\left\|\mathbf{v}_1-\mathbf{v}_2\right\|^2
}\\
&=&-\frac{1}{2\,\tau}\,\left(\|\mathbf{v}_1\|^2+\|\mathbf{v}_2\|^2\right)
+\frac{1}{\tau}\,\langle\mathbf{v}_1,\mathbf{v}_2\rangle=
-\frac{1}{2\,\tau}\,\|\mathbf{v}_1-\mathbf{v}_2\|^2.\nonumber
\end{eqnarray}

Hence
\begin{eqnarray}
\label{eqn:integrale gaussiano 1}
\lefteqn{\int_{\mathbb{R}^{2\,d-2}}\,e^{S_c(\theta_j,\mathbf{v}_j,\mathbf{u})}\,
\mathrm{d}\mathbf{u}}\\
&=&e^{\imath\,\frac{(\theta_1-\theta_2)\,\theta_1}{2\,\tau^4}\,a_x}
\cdot\left(\tau\,\pi\right)^{d-1}\,
e^{\frac{1}{\tau}\,\left[ -\imath\,\omega_x\left(
\mathbf{v}_1,\mathbf{v}_2\right)-\frac{1}{2}\,\|\mathbf{v}_1-\mathbf{v}_2\|^2-\frac{\imath}{2}\,\omega_x\left(
J_x(\mathbf{a}),\mathbf{v}_1+\mathbf{v}_2\right)   \right]}
\nonumber\\
&&\cdot e^{-\frac{1}{4}\,\left\| J_x(\mathbf{a})\right\|^2
-\frac{1}{2}\,\langle\mathbf{v}_1-\mathbf{v}_2,J_x(\mathbf{a})\rangle}
\nonumber\\
&=&e^{\imath\,\frac{(\theta_1-\theta_2)\,\theta_1}{2\,\tau^4}\,a_x}
\cdot\left(\tau\,\pi\right)^{d-1}\nonumber\\
&&\cdot e^{\frac{1}{\tau}\,\left[ -\imath\,\omega_x\left(
\mathbf{v}_1,\mathbf{v}_2\right)
-\frac{1}{4}\,\|\mathbf{v}_1-\mathbf{v}_2\|^2
-\frac{1}{4}\,\|\mathbf{v}_1-\mathbf{v}_2+J_x(\mathbf{a})\|^2
-\frac{\imath}{2}\,\omega_x\left(
J_x(\mathbf{a}),\mathbf{v}_1+\mathbf{v}_2\right)   \right]}.
\nonumber
\end{eqnarray}

Recalling Definition \ref{defn:newPsi2inv},
we conclude
\begin{equation}
\label{eqn:leading integral}
\int_{\mathbb{R}^{2\,d-2}}\,e^{S_c(\theta_j,\mathbf{v}_j,\mathbf{u})}\,
\mathrm{d}\mathbf{u}=
e^{\imath\,\frac{(\theta_1-\theta_2)\,\theta_1}{2\,\tau^4}\,a_x}
\cdot\left(\tau\,\pi\right)^{d-1}\,e^{\frac{1}{\tau}\,\Psi_2\big((\theta_1,\mathbf{v}_1),(\theta_2,\mathbf{v}_2)\big)}.
\end{equation}
Hence the leading order term in the asymptotic expansion
(\ref{eqn:smoothed kernel reexp4}) is
\begin{eqnarray}
\label{eqn:leading order term}
\frac{1}{\sqrt{2\,\pi}}\cdot\left(\frac{\lambda}{2\,\pi\tau}\right)^{d-1}\,
e^{\imath\,\sqrt{\lambda}\,\frac{\theta_1-\theta_2}{\tau}}
\,
e^{\imath\,\frac{(\theta_1-\theta_2)\,\theta_1}{2\,\tau^4}\,a_x}
\,\chi (0)\,
e^{\frac{1}{\tau}\,\Psi_2\big((\theta_1,\mathbf{v}_1),(\theta_2,\mathbf{v}_2)\big)}.
\end{eqnarray}

\begin{lem}
\label{lem:axvanishes}
$a_x=0$.
\end{lem}

\begin{proof}
Let us consider the special case where 
$\mathbf{v}_1=\mathbf{v}_2=\mathbf{0}$ and $\theta_2=0$,
and set $x_\lambda:=x+(\theta/\sqrt{\lambda},\mathbf{0})$.
By (\ref{eqn:leading order term}), the leading order term for the asymptotic expansion
of $\Pi_{\chi,\lambda}(x_\lambda,x)$ is
$$
\frac{1}{\sqrt{2\,\pi}}\cdot\left(\frac{\lambda}{2\,\pi\tau}\right)^{d-1}\,
e^{\imath\,\sqrt{\lambda}\,\frac{\theta}{\tau}}
\,
e^{\imath\,\frac{\theta^2}{\tau^4}\,a_x}
\,\chi (0)\,
e^{-\frac{1}{4\,\tau}\,\|J_x(\mathbf{a})\|^2},
$$
while the one for $\Pi_{\chi,\lambda}(x,x_\lambda)$ is
$$
\frac{1}{\sqrt{2\,\pi}}\cdot\left(\frac{\lambda}{2\,\pi\tau}\right)^{d-1}\,
e^{-\imath\,\sqrt{\lambda}\,\frac{\theta}{\tau}}
\,\chi (0)\,
e^{-\frac{1}{4\,\tau}\,\|J_x(\mathbf{a})\|^2}.
$$
Since $\Pi_{\chi,\lambda}(x,x_\lambda)=
\overline{\Pi_{\chi,\lambda}(x_\lambda,x)}$  for any
$\theta$, $a_x=0$.
\end{proof}

\begin{lem}
\label{lem:A_x=0}
$\mathbf{A}_x=0$.
\end{lem}

\begin{proof}
Given Lemma \ref{lem:integral curve Rlocal}
and Corollary \ref{cor:integral curve geodlocal},
$$
\Gamma^\tau_{\theta/\sqrt{\lambda}}(x)=x+
\left(-\frac{\tau\,\theta}{\sqrt{\lambda}}+R_3\left(\frac{\tau\,\theta}{\sqrt{\lambda}}\right),\mathbf{R}_2\left(\frac{\tau\,\theta}{\sqrt{\lambda}}\right)\right).
$$
%
Hence, in view of Lemma \ref{lem:axvanishes},
(\ref{eqn:smoothed kernel reexp4}),
and (\ref{eqn:leading order term})
\begin{eqnarray*}
\Pi^\tau_{\chi,\lambda}\left(
\Gamma^\tau_{\theta/\sqrt{\lambda}}(x),
x
\right)
&\asymp&\frac{1}{\sqrt{2\,\pi}}\cdot\left(\frac{\lambda}{2\,\pi\tau}\right)^{d-1}\,
e^{-\imath\,\sqrt{\lambda}\,\theta}
\,\chi (0)\,
e^{\frac{1}{\tau}\,\Psi_2\big((-\tau\,\theta,\mathbf{0}),(0,\mathbf{0})\big)}\\
&=&\frac{1}{\sqrt{2\,\pi}}
\cdot\left(\frac{\lambda}{2\,\pi\tau}\right)^{d-1}\,
e^{-\imath\,\sqrt{\lambda}\,\theta}
\,\chi (0)\,
e^{-\frac{\theta^2}{4\,\tau}\,\| \mathbf{A}_x\|^2},
\end{eqnarray*}
where $\asymp$ means \lq asymptotic to leading order\rq.
The estimate holds uniformly for $|\theta|\le
C\,\lambda^\delta$.

On the other hand, by Theorem 1.1 of \cite{cr2} 
in the same range
we also have
\begin{eqnarray*}
\left|\Pi^\tau_{\chi,\lambda}\left(
\Gamma^\tau_{\theta/\sqrt{\lambda}}(x),
x
\right)\right|
&\asymp&
C_{d,\tau}'\,
\lambda^{d-1}\,\chi (0),
\end{eqnarray*}
for some constant 
$C'_{d,\tau}$,
and the claim follows.
\end{proof}

Let us consider the lower order terms. The coefficient of 
$\lambda^{d-1-\frac{k}{2}}$ in the expansion (\ref{eqn:smoothed kernel reexp4})
is a linear combination of Gaussian integrals 
of the form
\begin{equation}
\label{eqn:jth gaussian integral}
\theta_j^{l}\cdot\mathbf{v}_j^{L} \int_{\mathbb{R}^{2\,d-2}}\,e^{S_c(\theta_j,\mathbf{v}_j,\mathbf{u})}
\mathbf{u}^{L'}\,
\mathrm{d}\mathbf{u},
\end{equation}
where $L$ and $L'$ are multi-indexes, $l+|L|+|L'|\le 3\,j$,
and $l+|L|+|L'|$ has the same parity as $j$. In turn, in view
of (\ref{eqn:Sc2}), the Gaussian integral 
in (\ref{eqn:jth gaussian integral}) may be written as a linear combination of
summands of the form
\begin{equation}
\label{eqn:jth gaussian integral1}
\theta_j^{a}\cdot\mathbf{v}_j^{B} \,D^{(1)}_{\theta_j,\mathbf{v}_j}\circ
\cdots\circ D^{(|C|)}_{\theta_j,\mathbf{v}_j} \left[\int_{\mathbb{R}^{2\,d-2}}\,e^{S_c(\theta_j,\mathbf{v}_j,\mathbf{u})}\,
\mathrm{d}\mathbf{u}\right],
\end{equation}
where each $D^{(a)}_{\theta_j,\mathbf{v}_j}$ a first order 
differential operator
with constant coefficients and no zeroth order term
in $(\theta_j,\mathbf{v}_j)$, with $a+|B|+|C|=l+|L|+|L'|$.
It then follows from  (\ref{eqn:defn di PSI0})
and (\ref{eqn:leading integral}) that the coefficient of
$\lambda^{d-1-\frac{k}{2}}$ has the form
$$
P_k(x;\mathbf{v}_j,\theta_j)\,e^{\frac{1}{\tau}\,\Psi_2\big((\theta_1,\mathbf{v}_1),(\theta_2,\mathbf{v}_2)\big)},
$$
where again $P_k$ is a polynomial of degree $\le 3\,k$
and parity $k$.

\end{proof}

\section{Proof of Theorem \ref{thm:Poisson main thm}}

\subsection{Preliminaries for Theorem \ref{thm:Poisson main thm}}
\label{scnt:prel wave poisson}
In addition to the general setting of Theorem \ref{thm:absolute main thm},
we need the description, also
due to Zelditch, of the wave group in the complex domain as a dynamical Toeplitz operator (\cite{z10}, Proposition 7.1; \cite{z12}, especially \S 8-9;
\cite{z14}, \S 4;
see also the discussions in \cite{cr1} and \cite{cr2}). 
This is the analogue 
of the description
of
$U_{\sqrt{\rho}}^\tau(t)$ as a dynamical Toeplitz operator
recalled in \S \ref{sctn:dyn toepl ops}.

For $t\in \mathbb{R}$ and $\tau>0$ sufficiently small, Zelditch considers the complexified Poisson wave kernel
\begin{equation}
\label{eqn:poisson wave complex}
U_{\mathbb{C}}(t+2\,\imath\,\tau)=P^\tau\circ U(t)\circ {P^\tau}^*,
\end{equation}
where $P^\tau$ is as in (\ref{eqn:complexified eigenfunction poisson}).
The distributional kernel of (\ref{eqn:poisson wave complex})
admits the expansion
\begin{equation}
\label{eqn:poisson wave complex eigenf expansion}
U_{\mathbb{C}}(t+2\,\imath\,\tau,x,y)=
\sum_j e^{(-2\,\tau+\imath\,t)\,\mu_j}\,
\tilde{\varphi}_j(x)\,\overline{\tilde{\varphi}_j(y)}.
\end{equation}
Arguing as for (\ref{eqn:smoothed kernel deriv}), one may 
then rewrite 
(\ref{eqn:smooth tempered spproj0})
as
\begin{equation}
\label{eqn:smooth tempered spproj UC}
P_{\chi,\lambda}(x,y)
=\frac{1}{\sqrt{2\,\pi}}\,\int_{-\infty}^{+\infty}\chi(t)\,e^{-\imath\,\lambda\,t}\,
U_{\mathbb{C}}(t+2\,\imath\,\tau,x,y)\,\mathrm{d}t.
\end{equation}

Alternatively, $U_{\mathbb{C}}(t+2\,\imath\,\tau)$ 
is a Fourier integral operator with complex phase of positive type 
on $X^\tau$, of degree
$-(d-1)/2$; in the terminology of
\cite{bg}, it is in fact a Fourier-Hermite operator adapted to 
the symplectomorphism $\Sigma^\tau\rightarrow \Sigma^\tau$ 
induced by
the homogeneous geodesic flow at time $t$.
More precisely, for every $t\in \mathbb{R}$ the real locus of the 
(complex) canonical relation of $U_{\mathbb{C}}(t+2\,\imath\,\tau)$
is as follows.

Let $\gamma_t:T^*M\setminus (0)\rightarrow T^*M\setminus (0)$ denote the
(genuine) homogeneous geodesic flow; thus $\gamma_t$ 
gets intertwined  by $E^\epsilon$ with 
Hamiltonian flow of $\sqrt{\rho}$ on $\tilde{M}\setminus M$.
Furthermore, recall that 
$\Gamma^\tau_t:X^\tau\rightarrow X^\tau$ 
is the restriction of the latter flow 
of $\upsilon_{\sqrt{\rho}}$ on $\tilde{M}\setminus M$,
and $\alpha^\tau$ is invariant under $\Gamma^\tau_t$.
Therefore, the cotangent lift of $\Gamma^\tau_t$ 
restricts on $\Sigma^\tau\subset T^*X^\tau\setminus (0)$ to
a flow $\Sigma^\tau\rightarrow \Sigma^\tau$
by
homogeneous symplectomorphisms, that by abuse of language we
shall also denote $\Gamma^\tau_t$
($\Sigma^\tau$ as in (\ref{eqn:symplectic cone})).

Let 
$\iota_\tau:T^*M\setminus (0)\rightarrow \Sigma_\tau$ be given by
\begin{equation}
\label{eqn:iotatausymplectomorphism}
\iota_\tau (m,\beta):=\left( 
E^\epsilon\left(m,\tau\,\frac{\beta}{\|\beta\|}\right),
\|\beta\|\,\alpha_{E^\epsilon\left(x,\tau\,\frac{\beta}{\|\beta\|}\right)}\right).
\end{equation}
Then $\iota_\tau$ is a homogeneous symplectomorphism,
and intertwines 
the $\gamma_t:T^*M\setminus (0)$
with $\Gamma^\tau_t$. 
Then the wave front of $P^\tau$ is (see e.g. \cite{z10})
\begin{equation}
\label{eqn:can rel Ptau}
\mathrm{WF}'\left(P^\tau\right)=
\left\{\big(\iota_\tau (m,\beta),\,(m,\beta)\big)\,:\,
(m,\beta)\in T^*M\setminus (0) \right\}.
\end{equation}
On the other hand, the wave front of $U(t)$ is 
(see e.g. \cite{dg})
\begin{equation}
\label{eqn:can rel Ptau}
\mathrm{WF}'\big(U(t)\big)=
\left\{\big(\gamma_t (m,\beta),\,(m,\beta)\big)\,:\,
(m,\beta)\in T^*M\setminus (0) \right\}.
\end{equation}
Given  
(\ref{eqn:poisson wave complex})
and the composition law for wave fronts
we conclude
\begin{eqnarray}
\label{eqn:WF'U}
\mathrm{WF}'\big( U_{\mathbb{C}}(t+2\,\imath\,\tau)  \big)&=&
\left\{\left(\Gamma^\tau_t
\left(x,\xi\right), (x,\xi)\right)\,:\,
(x,\xi)\in \Sigma^\tau\right\}.
\end{eqnarray}
This is however the same wave front of the composition
$\Pi^\tau\circ \Pi^\tau_{-t}$, whence of any composition
$\Pi\circ Q\circ \Pi^\tau_{-t}$, with $Q$ a
pseudodifferential operator which is elliptic on a
conic neighourhood of $\Sigma^\tau$.
Hence, for any such $Q$, 
$U_{\mathbb{C}}(t+2\,\imath\,\tau)$
and $\Pi\circ Q\circ \Pi^\tau_{-t}$ are both
Fourier integral operators of Hermite type with the same
real canonical relation; it follows that they are associated
to the same complex Lagrangian.
This is the basis for the following analogue of 
Theorem \ref{thm:dyn Usqrtrho}
(notation is as in
\S \ref{sctn:dyn toepl ops}). 


\begin{thm} (Zelditch)
\label{thm:dyn wave poisson}
There exist a polyhomogeneous classical symbol on $X^\tau\times \mathbb{R}_+$
of the form 
$$\gamma ^\tau_t(x,r)\sim \sum_{j\ge 0}\gamma_{t,j}^\tau (x)\,
r^{-\frac{d-1}{2}-j},
$$ 
and a pseudo-differential operator
$Q^\tau_t
\sim  
\gamma^\tau_t(x,D_{\sqrt{\rho}}^\tau)$ of degree
$-(d-1)/2$ on $X^\tau$, 
such that up to smoothing operators 
\begin{equation}
\label{eqn:poisson wave dynamical}
U_{\mathbb{C}}(t+2\,\imath\,\tau)\sim 
\Pi^\tau\circ Q^\tau_t\circ \Pi^\tau_{-t}=
\Pi^\tau\circ \mathcal{Q}^\tau_t,
\end{equation}
where
$\mathcal{Q}^\tau_t:=
Q^\tau_t\circ \Pi_{-t}^{\tau}$.
\end{thm}

The coefficients $\gamma_{t,j}^\tau$ depend on the choice of volume form.
We shall determine $\gamma_{0,0}^\tau$ \textit{a posteriori} by deriving
a local Weyl law for the complexified eigenfunctions and comparing it to
the one in Proposition 3.8 of \cite{z14}.

\begin{rem}
Up to smoothing operators that do not affect the
asymptotics, the composition
(\ref{eqn:poisson wave dynamical}) only depends on the behaviour
of $Q^\tau_t$ in a small conic neighborhood of $\Sigma^\tau$, 
where
$D_{\sqrt{\rho}}^\tau$ may be assumed to be microlocally
elliptic.
\end{rem}

Using known results on the composition of 
pseudodifferential and Fourier integral operators,
in view of 
(\ref{eqn:symbol Wtoeplitz}) and  Corollary \ref{cor:diffpsi diagonal} one obtains from Theorem \ref{thm:dyn wave poisson} the following analogue
of Lemma \ref{lem:calPtaut}.

\begin{lem}
\label{lem:calQtaut}
Up to smoothing terms, the Schwartz kernel 
of 
$\mathcal{Q}^\tau_t$ has the form
\begin{equation}
\mathcal{Q}^\tau_t(x_1,x_2)\sim\int_0^{+\infty}
e^{\imath\,u\,\psi^\tau\left(\Gamma^\tau_{-t}(x_1),x_2  \right)}\,
q^\tau_t\left(x_1,x_2,u\right)\,\mathrm{d}u,
\end{equation}
where
$$
q^\tau_t\left(x_1,x_2,u\right)\sim 
\sum_{j\ge 0}\,u^{\frac{d-1}{2}-j}\,q^\tau_{t\,j}(x_1,x_2),
$$
with
$$
q^\tau_{0\,0}(x_1,x_2)=
\frac{1}{\tau^{\frac{d-1}{2}}}\cdot\gamma_{0,0}^\tau(x)\,
s^\tau_0\left(\Gamma^\tau_{t}(x_1),x_2\right).
$$
\end{lem}

\subsection{Proof of Theorem \ref{thm:Poisson main thm}}
\begin{proof}
By the considerations in \S \ref{scnt:prel wave poisson},
the analysis of
$P^\tau_{\chi,\lambda}(x_{1\lambda},x_{2\lambda})$ 
(notation as in (\ref{eqn:xjlambda}))
parallels the one for 
$\Pi^\tau_{\chi,\lambda}(x_{1\lambda},x_{2\lambda})$ in Theorem
\ref{thm:absolute main thm}, with the following change.

The leading order term of the amplitude has been
multiplied by a factor
$\gamma_{0,0}^\tau(x)\cdot (u\,\tau)^{-\frac{d-1}{2}}$. 
In view of 
the rescaling $u\mapsto \lambda\,u$
in (\ref{eqn:smoothed kernel reexp1}), 
this change entails an additional factor 
$\gamma_{0,0}^\tau(x)\cdot\lambda^{-\frac{d-1}{2}}$ in front of the
resulting asymptotic expansion.
Furthermore, by Lemma \ref{lem:critical point non deg}
at the critical point $u\,\tau=1$. 

Hence,
the leading order term of the asymptotic expansion
for $P_{\chi,\lambda}(x_{1\lambda},x_{2\lambda})$
only differs from the one of
$\Pi^\tau_{\chi,\lambda}(x_{1\lambda},x_{2\lambda})$
by the factor $\gamma_{0,0}^\tau(x)\cdot\lambda^{-\frac{d-1}{2}}$.
\end{proof}

\subsection{The pointwise Weyl law}

\begin{proof}
[Proof of Proposition \ref{prop:weyl}]
We shall prove (\ref{eqn:weyl_D}); the same argument, with
obvious adaptations, also proves
(\ref{eqn:weyl poisson}).

Let us choose
$\chi\in \mathcal{C}^\infty_c\big((-\epsilon,\epsilon)\big)$
with $\hat{\chi}>0$. Define 
$f:\mathbb{R}\times \mathbb{R}\rightarrow [0,+\infty)$
by 
$$
f(t,s):=\hat{\chi}(t)\,H(\lambda-s-t),
$$
where $H$ is the Heaviside function.

Let us consider the following
positive measures $\mathcal{L}$ and $\mathcal{T}_x^\tau$ on $\mathbb{R}$. 
First, $\mathcal{L}$ is the Lebesgue measure. Second,
\begin{equation}
\label{eqn:defn of Txtau}
\mathcal{T}_x^\tau:=\sum_{j\ge 1}\,
\Pi_j^\tau(x,x)\,
\delta_{\lambda_j}.
\end{equation}
Let us endow $\mathbb{R}\times \mathbb{R}$ with the product measure 
$\mathcal{L}\times \mathcal{T}$.


By the Fubini Theorem (see e.g. Ch. 8 of \cite{r}), 
\begin{equation}
\label{eqn:Fubini positive}
\int_{\mathbb{R}}\,\mathrm{d}\mathcal{T}_x(s)\,\left[\int_{\mathbb{R}}\,f(t,s)\,
\,\mathrm{d}\mathcal{L}(t) \right] =
\int_{\mathbb{R}}\,\mathrm{d}\mathcal{L}(t)\,\left[\int_{\mathbb{R}}\,f(t,s)\,
\,\mathrm{d}\mathcal{T}_x(s) \right];
\end{equation}
the claim will follow by comparing both sides of (\ref{eqn:Fubini positive}).

The former integral in (\ref{eqn:Fubini positive}) is
\begin{eqnarray}
\label{eqn:st order fub}
\lefteqn{ 
\int_{\mathbb{R}}\,\mathrm{d}\mathcal{T}_x(s) \,\left[ \int_{\mathbb{R}}\,f(t,s) \,\mathrm{d}\mathcal{L}(t)  \right]  
= \int_{\mathbb{R}}\,\mathrm{d}\mathcal{T}_x(s) \,\left[ 
 \int_{-\infty}^{\lambda-s}\,\hat{\chi}(t)\,\mathrm{d}t \right]}\nonumber
\\
&=& \int_{\mathbb{R}}\,\mathrm{d}\mathcal{T}_x(s) \,\left[ 
 \int_{-\infty}^{\lambda}\,\hat{\chi}(t-s)\,\mathrm{d}t\right]=\sum_{j\ge 1}\,
\Pi_j^\tau(x,x)
\cdot \int_{-\infty}^{\lambda}\,\hat{\chi}(t-\lambda_j)\,\mathrm{d}\,t \nonumber\\
&=&\int_{-\infty}^{\lambda}\,\left[\sum_{j\ge 1}\,
\Pi_j^\tau(x,x)
\cdot \hat{\chi}(t-\lambda_j)\right]\,\mathrm{d}\,t=
\int_{-\infty}^{\lambda}\,\Pi^\tau_{\chi,t}(x,x)
\,\mathrm{d}\,t;
\end{eqnarray}
on the last line, we have made use of
(\ref{eqn:smoothed kernel}) and Theorem 1.27 of \cite{r}.
In view of Corollary \ref{cor:on diagonal asy}, (\ref{eqn:st order fub}) implies that
as $\lambda\rightarrow+\infty$
\begin{eqnarray}
\label{eqn:st order fubest}
\int_{\mathbb{R}}\,\mathrm{d}\mathcal{T}_x(s) \,\left[ \int_{\mathbb{R}}\,f(t,s) \,\mathrm{d}\mathcal{L}(t)  \right] 
=\frac{\sqrt{2\,\pi}}{d\cdot (2\,\pi)^d}\cdot\frac{\lambda^d}{\tau^{d-1}}
\cdot
\chi (0)
+O\left(\lambda^{d-1}\right).
\end{eqnarray}

On the other hand, the latter integral in (\ref{eqn:Fubini positive}) is
\begin{eqnarray}
\label{eqn:second integral fubini}
\lefteqn{
\int_{\mathbb{R}}\,\mathrm{d}\mathcal{L}(t)\,\left[\int_{\mathbb{R}}\,f(t,s)\,
\,\mathrm{d}\mathcal{T}_x(s) \right] 
=\int_{-\infty}^{+\infty}\left[ \sum_j\,\Pi_j^\tau(x,x)\,f(t,\lambda_j)    \right]\,\mathrm{d}\,t}\\
&=& \int_{-\infty}^{+\infty}\left[ \sum_j\,\Pi_j^\tau(x,x)\,H(\lambda-t-\lambda_j) \right]\,
\hat{\chi}(t)\,\mathrm{d}\,t  \nonumber\\
&=&\int_{-\infty}^{+\infty}\,\mathcal{W}_x(\lambda-\tau)\,
\hat{\chi}(t)\,\mathrm{d}\,t  \nonumber\\
&=&\mathcal{W}_x(\lambda)\,\int_{-\infty}^{+\infty}\,
\hat{\chi}(t)\,\mathrm{d}\,t  +
\int_{-\infty}^{+\infty}\,
\big[\mathcal{W}_x(\lambda-\tau)-\mathcal{W}_x(\lambda)\big]\,
\hat{\chi}(t)\,\mathrm{d}\,t\nonumber\\
&=&\sqrt{2\,\pi}\, \chi (0)\cdot \mathcal{W}_x(\lambda)+
\int_{-\infty}^{+\infty}\,
\big[\mathcal{W}_x(\lambda-\tau)-\mathcal{W}_x(\lambda)\big]\,
\hat{\chi}(t)\,\mathrm{d}\,t.\nonumber
\end{eqnarray}

\begin{lem}
\label{lem:bound difference lambda}
For $\lambda\rightarrow+\infty$, we have
$$\int_{-\infty}^{+\infty}\,
\big[\mathcal{W}_x(\lambda-\tau)-\mathcal{W}_x(\lambda)\big]\,
\hat{\chi}(t)\,\mathrm{d}\,t=O\left(\lambda^{d-1}\right).
$$
\end{lem}

\begin{proof}
It follows from Corollary \ref{cor:on diagonal asy}
that for $\lambda\gg 0$
$$
\mathcal{W}_x(\lambda+1)-\mathcal{W}_x(\lambda)=\sum_{\lambda\le \lambda_j\le \lambda+1}\,
\Pi_j(x,x)=O\left(\lambda^{d-1}\right).
$$
Hence there exist $C_0,\,C_1>0$ such that
$$
\mathcal{W}_x(\lambda+1)-\mathcal{W}_x(\lambda)\le 
C_1\,|\lambda|^{d-1}+C_0\qquad 
\forall\, \lambda\in \mathbb{R}.
$$
Thus for suitable $C',\,C''>0$ for any $\lambda,\,t\in \mathbb{R}$
$$
\big|\mathcal{W}_x(\lambda-t)-\mathcal{W}_x(t)\big|\le C'\cdot |t|\,\left[|\lambda|^{d-1}+
|\lambda-t|^{d-1}\right]+C''.
$$
Therefore for $\lambda\gg 0$
\begin{eqnarray*}
\lefteqn{ \left|\int_{-\infty}^{+\infty}\,
\big[W_x(\lambda-\tau)-W_x(\lambda)\big]\,
\hat{\chi}(t)\,\mathrm{d}\,t  \right| }\\
&\le&C'\,\int_{-\infty}^{+\infty}\,|t|\,\left[|\lambda|^{d-1}+
|\lambda-t|^{d-1}\right]\,\hat{\chi}(t)\,\mathrm{d}\,t+C'''
\le A\,\lambda^{d-1}+B,
\end{eqnarray*}
for appropriate constants $A,B>0$.

\end{proof}

Comparing (\ref{eqn:st order fub}) and (\ref{eqn:second integral fubini}),
we conclude that 
$$
\sqrt{2\,\pi}\, \chi (0)\cdot W_x(\lambda)=
\frac{\sqrt{2\,\pi}}{d\cdot (2\,\pi)^d}\cdot\frac{\lambda^d}{\tau^{d-1}}
\cdot
\chi (0)
+O\left(\lambda^{d-1}\right),
$$
whence 
$$
 W_x(\lambda)=\frac{\tau}{d}\cdot\left(\frac{\lambda}{2\,\pi\,\tau}\right)^d
+O\left(\lambda^{d-1}\right).
 $$
This establishes (\ref{eqn:weyl_D}). To obtain (\ref{eqn:weyl poisson}),
we need only run over the same argument, with the following changes:
first, in (\ref{eqn:defn of Txtau})
replace $\Pi_j^\tau(x,x)$ by 
$P^\tau_{\chi,\lambda}\left(
x,
x \right)$; second, in the derivation of (\ref{eqn:st order fubest}), invoke
Corollary \ref{cor:on diagonal asy poisson} in place of Corollary \ref{cor:on diagonal asy}.
\end{proof}

\bigskip 
The author states that there is no conflict of interest.

\bigskip

\noindent
\textbf{Acknowledgments.}
The author is a member of GNSAGA
(Gruppo Nazionale per le Strutture Algebriche, Geometriche e le loro Applicazioni) of INdAM 
(Istituto Nazionale di Alta Matematica "Francesco Severi"), and thanks
the group for its support.

\end{document}